\documentclass[11pt,reqno]{amsart}
\usepackage[utf8]{inputenc}

\usepackage[
	theoremdefs,
	final
]{latexdev}
\usepackage[numbers]{natbib} 
\usepackage[all]{xy}
\usepackage{booktabs}
\usepackage{amssymb}
\usepackage{color}
\usepackage{xcolor} 
\usepackage{amsxtra} 
\usepackage{tikz}
\usepackage{mathrsfs} 

\usepackage{epigraph}
\numberwithin{equation}{section}
\numberwithin{theorem}{section}

\usepackage{amsmath,amssymb}

\DeclareMathOperator{\grad}{grad}
\DeclareMathOperator{\divv}{div}
\DeclareMathOperator{\curl}{curl}

\newcommand{\ee}{\mathrm{e}}
\newcommand{\ii}{\mathrm{i}}
\newcommand{\pair}[1]{\left\langle #1 \right\rangle}
\newcommand{\poisson}[1]{\{ #1 \}}

\providecommand{\norm}[1]{\lVert#1\rVert}
\providecommand{\abs}[1]{\lvert#1\rvert}

\newcommand{\ud}{\mathrm{d}}

\newcommand{\RR}{{\mathbb R}}
\newcommand{\CC}{{\mathbb C}}

\newcommand{\vol}{\mu}
\newcommand{\met}{\mathsf{g}}
\newcommand{\Diff}{\mathrm{Diff}}
\newcommand{\Xcal}{\mathfrak{X}}
\newcommand{\Diffvol}{{\Diff_\vol}}
\newcommand{\Xcalvol}{{\Xcal_\vol}}

\newcommand{\LieD}{\mathcal{L}}
\newcommand{\interior}{\iota}

\newcommand{\Dens}{\mathrm{Dens}}

\DeclareMathOperator{\ad}{ad}
\DeclareMathOperator{\Ad}{Ad}

\newcommand*\id{\mathrm{id}}

\newcommand{\Met}{\mathsf{G}}
\newcommand{\MetW}{ {\bar{\mathsf{G}}} }
\newcommand{\MetF}{ {\bar{\mathsf{G}}} }

\newcommand{\Jac}[1]{\operatorname{Jac}(#1)}

\newcommand{\sslash}{\mathbin{\mkern-3mu/\mkern-5mu/\mkern-2mu}}
\newcommand{\IMH}{\mathrm{IMH}}
\newcommand{\CMH}{\mathrm{CMH}}

\newcommand{\marginnote}[1]
{
}

\newcounter{gm}

\newcounter{bk}
\newcommand{\bk}[1]
{\stepcounter{bk}$^{\bf BK\thebk}$%
\footnotetext{\hspace{-3.7mm}$^{\blacksquare\!\blacksquare}$
{\bf BK\thebk:~}#1}}

\newcounter{km}

\title[Geometric Hydrodynamics and infinite-dimensional Newton's equations]{Geometric Hydrodynamics and \\infinite-dimensional Newton's equations}

\author{Boris Khesin}
\address{Department of Mathematics, University of Toronto, Toronto, ON M5S 2E4, Canada}
\email{khesin@math.toronto.edu}

\author{Gerard \texorpdfstring{Misio\l ek}{Misiolek}}
\address{Department of Mathematics, University of Notre Dame, Notre Dame, IN 46556, USA}
\email{gmisiole@nd.edu}

\author{Klas Modin}
\address{Department of Mathematical Sciences, Chalmers University of Technology and University of Gothenburg, SE-412 96 Gothenburg, Sweden}
\email{klas.modin@chalmers.se}

\begin{document}

\rightline{\it To the memory of Vladimir Arnold and Jerry Marsden,}
\rightline{\it  pioneers of geometric hydrodynamics,}
\rightline{\it who left  in 2010, ten years ago.}
~\\~\\

\begin{abstract} 
We revisit the geodesic approach to ideal hydrodynamics and
present a related geometric  framework for Newton's equations on groups of diffeomorphisms 
and spaces of probability densities. The latter setting is sufficiently general to include 
equations of compressible and incompressible fluid dynamics, magnetohydrodynamics, 
shallow water systems and equations of relativistic fluids. 
We illustrate this with a survey of selected examples, as well as with new results, 
using the tools of infinite-dimensional information geometry, optimal transport, 
the Madelung transform, and the formalism of symplectic and Poisson reduction.  
\end{abstract} 

\vspace*{\fill}
\epigraph{The only way to get rid of dragons is to have one of your own.}{Evgeny Schwartz, \textit{The Dragon}}

\maketitle

\tableofcontents
\section{Introduction}\label{sec:intro} 
The Euler equations of hydrodynamics describe motions of an incompressible and inviscid fluid 
occupying a fixed domain (with or without boundary). 
In the 1960s V.~I.~Arnold discovered that these equations are precisely the geodesic equations 
of a right-invariant metric on the group of diffeomorphisms preserving the volume of the domain \cite{Ar1966}. 
This beautiful observation, combining the early work of Hadamard on geodesic flows on surfaces 
with the dynamical systems ideas of Poincare and Kolmogorov and 
using analogies with classical mechanics of rigid bodies,  inspired many researchers 
-- one of the first was J.~E.~Marsden. Their combined efforts led to remarkable developments such as 
formulation of new stability criteria for fluid motions \cite{Ar1965, Ar1989, FrStVi1997, FrSh2001}; 
explicit calculation of the associated Hamiltonian structures and first integrals \cite{MaWe1974, MaWe1983, ArKh1998};
development of symplectic reduction methods \cite{MaWe1974, MaRaWe1984b} introduction of Riemannian geometric techniques to the study of diffeomorphism groups 
including explicit computations of curvatures, conjugate points, diameters \cite{Ar1966, Mi1996, Sh1985, MiMu2013};
detailed studies of regularity properties of the solution maps of the fluid equations 
in Lagrangian and Eulerian coordinates \cite{EbMa1970, EbMiPr2006}
construction of similar configuration spaces for other partial differential equations of hydrodynamic origin \cite{ArKh1998, HoMaRa1998}, etc. 

In this paper, based on  the research pioneered and explored by Arnold, Marsden and many others, 
we present a broad geometric framework which includes an infinite-dimensional generalization of 
the classical Newton's equations of motion to the setting of diffeomorphism groups and spaces of probability densities. 
This approach has a wide range of applicability and covers a large class of important equations of 
mathematical physics. 
Our goal is twofold. 
We start by presenting a concise survey of various geodesic and Newton's equations 
thus introducing the reader to the rapidly expanding field of geometric hydrodynamics
and revisiting a few standard examples from the point of view advocated here. 
We then also include a number of selected new results to illustrate the flexibility and utility of this approach. 

We focus primarily on the geometric aspects and emphasize formal procedures 
leaving until the end analytic issues which in most cases can be resolved using standard methods 
once an appropriate functional-analytic setting (e.g., Fr\'echet, H\"older, or Sobolev) is adopted. 
The corresponding tame Fr\'echet framework is described in more detail in \autoref{sect:tame}.
Our main tools include the Wasserstein metric of optimal transport, 
the infinite-dimensional analogue of the Fisher-Rao information metric, 
the Madelung transform 
and 
the formalism of symplectic and Poisson reduction, 
all of which are defined in the paper. 
The early sections should be accessible to mathematicians with only general background 
in geometry. 
In later sections some acquaintance with the basic material found for example in the monographs \cite{Ar1989,
ArKh1998, MaRa1999}  will be helpful. 

Needless to say, it is not possible to give a comprehensive survey of such a vast area 
of geometric hydrodynamics in such a limited space, therefore our emphasis on certain topics 
and the choice of examples are admittedly subjective. 
(The epigraph to the paper is our take on the Laws of Nature, on the tamed structures discussed below, 
as well as a counterpoint to the beautiful epigraph in the monograph  \cite{BrLa1978}, 
quoted here in the footnote.\footnote{
``There once lived a man 

who learned how to slay dragons 

and gave all he possessed 

to mastering the art. 

\medskip

After three years 

he was fully prepared but, 

alas, he found no opportunity 

to practise his skills." 

\smallskip

\qquad\qquad
Dschuang Dsi. 
})
We nevertheless hope that this paper provides 
a flavour of some of the results in this beautiful area pioneered by V.~Arnold and J.~Marsden.




\subsection{Geodesics and Newton's equations: finite-dimensional examples} 
\label{sect:newton-ex} 
A curve $q(t)$ is a geodesic in a Riemannian manifold $Q$ if it satisfies the {\it equation of geodesics}, 
namely 
\begin{equation} \label{eq:geodesic_on_Q} 
\nabla_{\dot q} \dot q =0 \,, 
\end{equation} 
where $\nabla$ stands for the covariant derivative on $Q$ and the dot denotes the $t$-derivative.
If the Riemannian manifold is flat then the geodesic equation becomes 
the familiar  $\ddot q =0$ 
in any local Euclidean coordinates on $Q$. 

From the point of view of classical mechanics, the geodesic equation~\eqref{eq:geodesic_on_Q} 
describes motions $q(t)$ of a system driven only by a kinetic energy.
More general systems may depend also on a potential energy.
Indeed, if $Q$ is a configuration space of some physical system (a Riemannian manifold) 
and $U\colon Q \to \mathbb{R}$ represents its potential energy (a differentiable function) 
then $q(t)$ satisfies the \emph{Newton's equations} in $Q$ of the form 
\begin{equation} \label{eq:newton_on_Q} 
\nabla_{\dot q} \dot q = -\nabla U(q). 
\end{equation} 
%

One of the classical examples of Newton's equations is the $N$-body system in $\RR^3$. 
Introducing coordinates $q=(q_1, ..., q_N)\in \RR^{3N}$ one can regard $Q=\RR^{3N}$ 
to be the configuration space of the system. 
If the bodies have masses $m_k$ then their kinetic energy is 
$T(\dot{q})=\sum_{k=1}^{N} m_k {\|\dot q_k\|^2}/{2}$ 
and hence corresponds to a Riemannian metric on $Q$ of the form 
$\|\dot{q}\|^2=2T(\dot{q})$. 
If ${\rm G}$ denotes the gravitational constant then the potential energy is given by the expression
$$
U(q)=- \sum_{i <j }\frac{{\rm G} m_i m_j}{\|q_i-q_j\|}
$$
which becomes infinite on the diagonals $q_i=q_j$. The corresponding Lagrangian function 
is $L=T-U$, while the total energy of the system (its Hamiltonian) is $H=T+U$.
We shall revisit this system in a fluid dynamical context below.

\medskip

Another  classical example is provided by the C. Neumann problem \cite{Ne1856} 
describing the motion of a single particle on an $n$-sphere under the influence of 
a quadratic potential energy. 
Here, the configuration space is the unit sphere $Q=S^n$ in $\mathbb{R}^{n+1}$ 
while the phase space is the tangent bundle  $TS^n$ of the sphere. 
The potential energy of the system is given by $U(q) = \frac{1}{2}q \cdot Aq$, 
where $q\in S^n$ and $A$ is a positive-definite symmetric matrix. 
As before, the Lagrangian function is the difference of the kinetic and the potential energies 
\begin{equation*} 
L(q,\dot q) 
= 
\frac{\abs{\dot q}^2}{2} - \frac{q \cdot A q}{2}, 
\quad \text{for} \;\; 
q \in S^n \subset \RR^{n+1}. 
\end{equation*} 
The C. Neumann system is related to the geodesic flow on the ellipsoid defined by the equation 
$x\cdot A^{-1} x = 1$, 
see e.g., Moser \cite[Sec.~3]{Mo1983}. 
The corresponding Hamiltonian system on the cotangent bundle $T^*S^n$ is integrable 
and, if the eigenvalues 
$\alpha_1,\ldots,\alpha_{n+1}$ of $A$ 
are all different, then the first integrals, expressed in canonical coordinates $q$ and $p=\dot q$, are explicitly given by 
\begin{equation*} 
F_k(q,p) = q_k^2 + \sum_{j\neq k} \frac{p_j q_k - p_k q_j}{\alpha_k - \alpha_j} \,,
\end{equation*} 
%
where $q_k$ and $p_k$ are the components of $q$ and $p$ with respect to the eigenbasis of $A$. 
We will see that an infinite-dimensional analogue of the C. Neumann problem naturally arises in the context of information geometry, while its integrability  in infinite dimensions remains an intriguing open problem.

\subsection{Three motivating examples from hydrodynamics} 
\label{sec:zoo} 
We now make a leap from finite to infinite dimensions. 
Our aim is to show that many well-known PDEs of hydrodynamical pedigree 
can be cast as Newton's equations on infinite-dimensional manifolds.
Indeed, groups of smooth diffeomorphisms arise naturally as configuration spaces of compressible and incompressible fluids. 
Begin with three famous examples.
Consider a connected compact Riemannian manifold $M$ of dimension $n$ 
(for our purposes $M$ may be a domain in $\RR^n$) 
and assume that it is filled with an inviscid fluid (either a gas or a liquid). 
When the group of diffeomorphisms of $M$ is equipped with an $L^2$ metric 
(essentially, the metric corresponding to fluid's kinetic energy, as we shall discuss later) 
its geodesics describe the motions of noninteracting particles in $M$ 
whose velocity field $v$ satisfies the inviscid Burgers equation 
\begin{equation} 
\dot v+\nabla_v v=0. 
\end{equation} 

When the $L^2$ metric is restricted to the subgroup of diffeomorphisms of $M$ 
which preserve the Riemannian volume form $\mu$ 
then its geodesics describe the motions of an ideal (that is, inviscid and incompressible) fluid in $M$ 
whose velocity field satisfies the incompressible Euler equations 
\begin{equation} \label{eq:euler}
\begin{cases} 
\begin{aligned} 
&\dot v+\nabla_v v+\nabla P=0 
\\ 
&{\rm div}\, v=0. 
\end{aligned} 
\end{cases} 
\end{equation} 
Here $P$ is the pressure function whose 
gradient $\nabla P$ is defined uniquely by the divergence-free condition 
on the velocity field $v$ and can be viewed as a constraining force. 
(If $M$ has a nonempty boundary then $v$ is also required to be tangent to $\partial M$). 

As we shall see below, both of the above equations turn out to be examples of equations of geodesics 
on diffeomorphism groups with Lagrangians given by the corresponding kinetic energy. 
However, the Lagrangian in our next example will include also a potential energy.
Consider the equations of a compressible (barotropic) fluid describing the evolution of 
a velocity field $v$ and a density function $\rho$, namely 
\begin{equation} \label{eq:comp} 
\begin{cases} 
\begin{aligned} 
&\dot v+\nabla_v v+\tfrac 1\rho\nabla P(\rho)  = 0 \\
&\dot\rho + \divv(\rho v) = 0\,. 
\end{aligned} 
\end{cases} 
\end{equation} 
These equations can be interpreted as Newton's equations on the full diffeomorphism group of $M$. 
In this case the pressure is a prescribed function $P=P(\rho)$ of density $\rho$ 
and this dependence, called the equation of state, determines fluid's potential energy. 
In sections below we shall also consider general equations with the term $\tfrac 1\rho\nabla P(\rho)$ 
replaced by the gradient $\nabla W$, where $W(\rho)$ denotes an arbitrary thermodynamical work function, 
cf. \autoref{sub:compressible_euler_equations}.

\subsection{Riemannian metrics and their geodesics on spaces of diffeomorphisms and densities} 
\label{sect:otto} 
Let us next see how differential geometry of diffeomorphism groups manifests itself in the above equations. 
Given a Riemannian manifold $M$  we equip the group $\Diff(M)$ of all diffeomorphisms of $M$ with a (weak) Riemannian metric and a natural fibration. 

Namely, assume that the Riemannian volume form  $\mu$ has the unit total volume (or total mass) and regard it 
 as a reference  density on $M$. 
 Now  consider the projection $\pi : \Diff(M) \to \Dens(M)$ of diffeomorphisms 
onto the space $\Dens(M)$ of (normalized) smooth
densities on $M$. The  diffeomorphism group $\Diff(M)$
is fibered over $\Dens(M)$ by means of this projection
$\pi$ as follows: the fiber of $\mu$ is the subgroup $\Diff_\mu(M)$, while 
the fiber over a volume form $\tilde \mu$  consists of all 
diffeomorphisms $\varphi$ that push $\mu$ to $\tilde \mu$,  $\varphi_*\mu=\tilde \mu$ or, equivalently, $\mathrm{Jac}(\varphi^{-1})\mu = \tilde\mu$. (Note that diffeomorphisms from $\Diff(M)$ act transitively on smooth normalized densities, 
according to Moser's theorem.) In other words, 
two diffeomorphisms $\varphi_1$ and $\varphi_2$ 
belong to the same fiber 
if and only if $\varphi_1=\varphi_2\circ\phi$ for some 
diffeomorphism $\phi\in \Diff_\mu(M)$.

\smallskip

\begin{remark}
It is worth to compare ``the functional dimensions" of the fiber $\Diff_\mu(M)$ and the base  $\Dens(M)$.
The space of densities $\Dens(M)$ can be thought of as the space of functions of $n$ variables, where $\dim M=n$.
On the other hand, the  group $\Diff_\mu(M)$ consists of $i)$ isometries in dimension $n=1$ (e.g. for $M=S^1$ it is 
$\Diff_\mu(M)\approx {\rm Iso}(M)=S^1$), $ii)$ symplectic diffeomorphisms  in dimension $n=2$ (e.g. for $M=S^2$
these are Hamiltonian diffeomorphisms, locally described by a function of 2 variables), and $iii)$ in dimensions $n\ge 3$
these diffeomorphisms $\varphi\in \Diff_\mu(M)$ are subject to the only constraint on the Jacobian: $\Jac{\varphi}\equiv 1$, i.e. 
one equation on $n$ functions of $n$ variables. Therefore, in the fibration $\pi : \Diff(M) \to \Dens(M)$
the fiber is small comparatively to the base  in dimension $n=1$, the fiber and the base are about the same size
 in dimension $n=2$, and the fiber becomes much bigger than the base starting with the dimension $n=3$.
\end{remark}

\begin{definition}
Now define an $L^2$-{\it metric} on ${\rm Diff}(M)$ by the formula
$$
G_{\varphi}(\dot\varphi,\dot\varphi) = \int_M  |\dot\varphi|^2 \,\mu\,,
$$
where $\dot\varphi$ is a tangent vector at the point $\varphi\in \Diff(M)$, i.e., a map $\dot\varphi\colon M\to TM$ such that $\dot\varphi(x)\in T_{\varphi(x)}M$ for each $x\in M$, while $|\dot\varphi|^2 $ stands for the pointwise Riemannian product at the point $\varphi(x)\in M$.
\end{definition}

One can see that for a flat manifold $M$ this is a flat metric on ${\rm Diff}(M)$, as it is the $L^2$-metric on diffeomorphisms
$\varphi$ regarded as vector functions $x\mapsto \varphi(x)$.
This metric is   right-invariant for the ${\rm Diff}_\mu(M)$-action (but not the ${\rm Diff}(M)$-action):
$G_{\varphi}(\dot\varphi,\dot\varphi) =G_{\varphi\circ\eta}(\dot\varphi\circ\eta,\dot\varphi\circ\eta)$
for $\eta\in {\rm Diff}_\mu(M)$, since the change of coordinates leads to the factor $\Jac{\eta} \equiv 1$ in the integrand.

\begin{remark}
Consider the following  {\it optimal mass transport problem}: Find a  map $\varphi:M\to M$ that pushes
the  measure
$\mu$ forward to another measure $\tilde \mu$ of the same total volume and attains the minimum
of the $L^2$-cost functional 
$\int_M \operatorname{dist}^2(x,\varphi(x))\mu$ among all such maps ($\operatorname{dist}$ denotes here the Riemannian distance function on $M$). 
The minimal cost of transport defines the following {\it Kantorovich--Wasserstein} distance 
on the space of densities $\Dens(M)$:
\begin{equation}\label{optimal}
{\operatorname{Dist}}^2(\mu, \tilde \mu)
:=\inf_\varphi\Big\{\int_M \operatorname{dist}^2(x,\varphi(x))\mu~
|~\varphi_*\mu=\tilde \mu\Big\}\,.
\end{equation}

The mass transport problem admits a unique solution for Borel maps and densities 
(defined up to measure-zero sets), called the optimal map $\tilde\varphi$, see, e.g.,  \cite{Br1991, Mc2001, Vi2009}. 
In the smooth setting the Kantorovich-Wasserstein distance $\operatorname{Dist}$  is  
generated by a (weak) Riemannian metric on the space $\Dens$ of smooth densities~\cite{BeBr2000,Ot2001}, which we call the
{\it Wasserstein-Otto metric} and described in detail in \autoref{sub:fibration}. Thus both $\Diff(M)$ and $\Dens(M)$ can be regarded as infinite-dimensional Riemannian manifolds for the $L^2$ and Wasserstein-Otto metrics respectively.
\end{remark}

\begin{remark}
Later we will see (following \citet{Ot2001})  that the corresponding projection $\pi\colon\Diff(M)\to \Dens(M)$ is a
{\it Riemannian submersion} from the diffeomorphism group $\Diff(M)$ 
onto  the density space $\Dens(M)$, i.e., the map respecting the above metrics.
Recall that for two Riemannian manifolds $P$ and $B$ a {\it submersion}
$\pi: P\to B$ is a smooth map which has a surjective differential and preserves 
lengths of horizontal tangent vectors to~$P$. 
For a bundle $P\to B$ this means that on $P$ there is a distribution 
of horizontal 
spaces orthogonal to fibers and projecting isometrically to the 
tangent spaces to $B$.
Geodesics on $B$ can be lifted to horizontal geodesics in $P$, and the lift is unique for a given initial point in~$P$.


Note also that  horizontal (i.e., normal to fibers) spaces in the bundle $\Diff(M) \to \Dens(M)$ consist of right-translated gradient fields.
In short, this follows from  the Hodge decomposition  ${\rm Vect}={\rm Vect}_\mu\oplus_{L^2} {\rm Grad}$
for vector fields on $M$: any vector field $v$ decomposes uniquely into 
the sum $v=w+\nabla p$ of a divergence-free field $w$
and a gradient field $\nabla p$, which are $L^2$-orthogonal to each other,
$\int_M(w,\nabla p)\,\mu=0$. The vertical tangent space at the identity coincides with  ${\rm Vect}_\mu(M)$,
while the horizontal space is ${\rm Grad}(M)$. The vertical space (tangent to a fiber)
at a point $\varphi\in \Diff(M)$ consists of $w\circ\varphi$, divergence-free vector fields $w$
right-translated by the diffeomorphism $\varphi$, while the horizontal space is given by the right-translated 
gradient fields, $(\nabla p)\circ\varphi$.
The $L^2$-type metric $G_{\varphi}(\dot\varphi,\dot\varphi)$ on  horizontal 
spaces for different points of the same
fiber projects isometrically to one and the same metric
on the base, due to the $\Diff_\mu(M)$-invariance of the metric. Now the Riemannian submersion property follows from the observation that the   Wasserstein--Otto metric is Riemannian and generated by the $L^2$ metric 
on gradients, see \cite{BeBr2000}.
\end{remark} 
\begin{example}
Geodesics in the full diffeomorphism group $\Diff(M)$ with respect to the above $L^2$-metric 
have a particularly simple description for a flat manifold $M$, cf. \cite{EbMa1970, Br1989}. 
In that case the group $\Diff(M)$ is locally (a dense subset of) the $L^2$-space of vector-functions $x\mapsto g(x)$, 
and hence is flat, while its geodesics are straight lines.
If $v(t,x)$ is the velocity field of the flow $g(t,x)$ in $M$ defined by $\dot g(t,x)=v(t,g(t,x))$ then 
the geodesic equation $\nabla_{\dot q} \dot q =0$ becomes $\ddot g(t,x)=0$, 
which in turn is equivalent to the inviscid {\it Burgers equation}
$$
\dot v+\nabla_v v=0\,.
$$ 
 
Furthermore, from the viewpoint of exterior geometry, the Euler equation 
$\dot v+\nabla_v v=-\nabla p$ 
can be regarded as an equation with a constraining force $-\nabla p$ 
acting orthogonally to the submanifold of volume-preserving diffeomorphisms 
$\Diff_\mu(M)\subset \Diff(M)$ 
and keeping the geodesics confined to that submanifold.
\end{example} 
\begin{remark} 
Analytical studies of the differential geometry of the incompressible Euler equations 
began with the paper of \citet{EbMa1970} and continued with \cite{Sh1985, Mi1996, EbMiPr2006} and others. 
The approach via generalized flows was proposed by \citet{Br1989}.
Many aspects of this approach to the group of all diffeomorphisms and their relation to 
the Kantorovich-Wasserstein space of densities and problems of optimal mass transport 
are discussed in \cite{Vi2009, Yo2010, Lo2008}.
There is also a finite-dimensional matrix version of the submersion framework 
and decomposition of diffeomorphsims, see \cite{Br1991}. 
In the finite-dimensional optimal mass transport on $\RR^n$ discussed in \cite{Mo2017} 
the probability distributions are multivariate Gaussians and the transport maps are linear transformations. 
The corresponding dynamics turned out to be closely related  to many finite-dimensional flows studied in the literature: 
Toda-lattice, isospectral flows, and an entropy gradient interpretation of the Brockett flow for matrix diagonalization.
A sub-Riemannian version of the exterior geometry of $\Diff(M)$ with vector fields tangent to 
a bracket generating distribution in $M$, as well as a nonholonomic version of Moser lemma, 
is described in \cite{AgCa2009,  KhLe2009}. 
For a symplectic reduction formulation to the above Riemannian submersion see \autoref{sub:poisson_reduction}.
\end{remark} 
%


\subsection{First examples of Newton's equations on diffeomorphism groups}
\label{first_examples}
\begin{example}[Shallow water equation as a Newton's equation]\label{sub:shallow_water}
We next proceed to describe Newton's equations 
$\nabla_{\dot q} \dot q = -\nabla U(q)$ 
on the diffeomorphism group $\Diff(M)$. 
To this end we consider the case of a potential on $\Diff(M)$ 
which depends only on the density $\varphi^*\mu$ carried by a diffeomorphism $\varphi\in \Diff(M)$, 
i.e. the potential $U(\varphi):=\bar U(\rho)$ for $\rho\mu=\varphi^*\mu$ is a pullback for the projection 
$\Diff(M)\to\Dens(M)$, where as $\bar U$ we take a simple quadratic function 
\begin{equation} \label{eq:quadratic_func} 
\bar U(\rho) = \frac{1}{2}\int_M \rho^2\vol 
\end{equation} 
on the space of densities. It turns out that with this potential we obtain shallow water equations. 
There are several equivalent formulations, depending on the functional setting. 
\end{example} 
\begin{proposition} 
Newton's equations with respect to the $L^2$-metric \eqref{eq:L2met} 
and the potential \eqref{eq:quadratic_func} take the following forms: 
\begin{itemize} 
\item 
on $\Diff(M)$ 
\begin{equation} 
\nabla_{\dot\varphi}\dot\varphi + \nabla\rho\circ\varphi = 0 
\end{equation} 
where $\rho = \Jac{\varphi^{-1}}$, 
%
\item 
the shallow water equations on $\Xcal(M)$ 
\begin{align} \label{eq:shallow} 
\begin{cases} 
\dot v + \nabla_v v + \nabla \rho = 0  
\\ 
\dot\rho + \divv(\rho v) = 0 
\end{cases} 
\end{align} 
where 
$v = \dot{\varphi}\circ\varphi^{-1}$ is the horizontal velocity field and $\rho$ is the water depth, 
\item 
for the gradient velocity $v=\nabla \theta$ it assumes the Hamilton-Jacobi  form
\begin{align} 
\begin{cases} 
\dot\theta + \frac{1}{2}\abs{\nabla \theta}^2 + \rho = 0 \label{eq:theta_shallow_water} 
\\ 
\dot\rho + \divv(\rho \nabla\theta) = 0\,.
\end{cases} 
\end{align} 

\end{itemize} 
\end{proposition} 

\begin{remark}
The latter form can be regarded as an equation on $T\Dens(M)$. 
Since $\bar U$ is a quadratic function equations \eqref{eq:theta_shallow_water} can be interpreted 
as a Hamiltonian form of an infinite-dimensional harmonic oscillator 
with respect to the Wasserstein-Otto metric \eqref{eq:otto_metric}. 
We will prove this theorem in a more general setting of a barotropic fluid, cf. equation \eqref{eq:comp}, 
with an arbitrary potential $\bar U(\rho)$ in \autoref{sub:compressible_euler_equations}; 
here $\delta \bar U/\delta \rho=\rho$. 
\end{remark}





%

\begin{example}[The $N$-body problem as a Newton's equation]\label{Nbody}
Newton's  law of gravitation states that for a body with mass distribution $\rho$ the associated potential  is 
$V=4\pi {\rm G}(\Delta^{-1}\rho)$, where ${\rm G}>0$ is the gravitational constant.
Following the above framework, the potential function on $\Dens(M)$ is given by
\begin{equation}\label{eq:gravitypot}
\bar U(\rho) = 2\pi {\rm G} \int_M \rho\, (\Delta^{-1}\rho)\,\vol, 
\end{equation}
where $\Delta^{-1}$ is a (suitably defined) inverse Laplacian with appropriate boundary conditions. 

The corresponding fluid system is described by 
\begin{align} \label{eq:newton_newton} 
\begin{cases} 
\dot v + \nabla_v v + \nabla V = 0  
\\ 
\dot\rho + \divv(\rho v) = 0 .
\end{cases} 
\end{align} 
Thus, we have arrived at a fluid dynamics formulation of a continuous Newton mass system 
under the influence of gravity: 
a ``fluid particle'' positioned at $x$ experiences a gravitational pull corresponding to the potential $V(x) = 4\pi {\rm G}(\Delta^{-1}\rho)(x)$.
In particular, if $M=\RR^3$ we obtain the well-known Green's function for the Laplacian
\begin{equation}\label{eq:greensR3}
	V(x) = - \int_{\RR^3} \frac{{\rm G}}{\abs{x-y}}\rho(y)\vol(y).
\end{equation}

We now wish to study weak solutions to these equations where the mass distribution $\rho$ 
is replaced by an atomic measure 
\begin{equation}\label{eq:atomic}
	\rho(x) = \sum_{k=1}^N m_k\,\delta_{x_k}(x) 
\end{equation}
for $N$ point masses $m_k>0$ positioned at $x_k\in M$.
The group geometry of this setting is as follows.
We have a Riemannian metric on $\Diff(M)$ (the Wasserstein--Otto metric) 
and a potential function on $\Dens(M)$ (the Newton potential).
The group $\Diff(M)$ acts on the (finite-dimensional) manifold $\mathcal M^N$ of atomic measures with $N$ particles.
Clearly, we have $\mathcal M^N \simeq M^N\backslash \{\text{pairwise diagonals}\}$.
The isotropy subgroup for this action on $(x_1,\ldots,x_N)$ is 
\[
\Diff_{(x_1,\ldots,x_N)}(M) = \{ \varphi\in \Diff(M) \mid \varphi(x_i) = x_i, i=1,\ldots,N \}. 
\] 
Although the horizontal distribution is not defined rigorously, it is formally given by vectorfields 
with support on $\{x_1,\ldots,x_N\}$. 
With this notion of horizontality, the projection $\Diff(M)\to \mathcal M^N$, 
given by $\varphi \mapsto \sum_{k=1}^N m_k \delta_{\varphi(x_k)}$, 
is a Riemannian submersion with respect to the weighted Riemannian structure on 
$\mathcal M^N$, given by 
$\abs{(\dot x_1,\ldots,\dot x_N)}^2 = \sum_{k=1}^N m_k \abs{\dot x_k}^2$. 
For $M=\RR^3$ substituting the atomic measure into the formula~\eqref{eq:gravitypot} and using the Green's function for $\Delta^{-1}$ gives
\[
	\bar U(\rho) = \bar U\left(\sum_{k=1}^Nm_k \delta_{x_k}\right) = -\sum_{k< l} \frac{\mathrm{G}m_k m_l}{\abs{x_k-x_l}}.
\]
The resulting finite-dimensional Riemannian structure together with this potential function is exactly the kinetic and potential energies giving rise to the $N$-body problem.
\end{example}


\begin{remark}
In the wake of Arnold's work, various approaches to infinite dimensional generalizations of 
Newton's equations \eqref{eq:newton_on_Q} have been considered in special settings. 
Perhaps those of most interest from our point of view 
were proposed by Smolentsev \cite{Sm1979, Sm2007} who used diffeomorphism groups 
to describe the motions of a barotropic fluid 
and by Ebin \cite{Eb1975} who used a similar framework to study, among others, 
the incompressible limit of slightly compressible fluids. 
In the early 1980s, Doebner, Goldin and Sharp~\cite{DoGo1992, GoSh1989}
began to develop links between representations of diffeomorphism groups, ideal fluids 
and nonlinear quantum systems, 
revisiting in the process the classical transform of \citet{Ma1927,Ma1964}. 
More recently, 
motivated by the problems of optimal transport, \citet{Re2012}
used it to relate the Schr\"odinger equations with a variant of Newton's equations 
defined on the space of probability measures 
(see \autoref{sec:madelung} below for details). 
A similar objective, but driven partly by motivation from information geometry and statistics, 
can be found in a recent paper of Molitor \cite{Mo2015b}.

\smallskip

\begin{table}
	\centering
	\begin{tabular}{ll}
		\toprule
		Wasserstein-Otto geometry & Fisher-Rao geometry \\
		\midrule
		\multicolumn{2}{c}{\it Newton's equations on $\Diff(M)$} \\[0.5ex]
		$\bullet$ Inviscid Burgers'  equation (\autoref{sub:burgers}) & $\bullet$ $\mu$-Camassa-Holm equation (\autoref{sub:muCH})\\
		$\bullet$ Classical mechanics (\autoref{sub:hamilton_jacobi})  & $\bullet$ Optimal information transport (\autoref{sec:fisher_rao}) \\
		$\bullet$ Barotropic inviscid fluid (\autoref{sub:compressible_euler_equations}) & \\
		$\bullet$ Fully compressible fluid (\autoref{sect:fully}) & \\	
		$\bullet$ Magnetohydrodynamics (\autoref{sect:comprMHD}) & \\[1ex]
		\multicolumn{2}{c}{\it Newton's equations on $\Dens(M)$ } \\[0.5ex]
		$\bullet$ Hamilton-Jacobi equation (\autoref{sub:hamilton_jacobi}) & $\bullet$ $\infty$-dim Neumann problem (\autoref{sect:neumann}) \\
		$\bullet$ Linear Schr\"odinger equation (\autoref{sect:NLS}) & $\bullet$ Klein-Gordon equation (\autoref{sec:klein_gordon})\\
		$\bullet$ Non-linear Schr\"odinger (\autoref{sect:NLS}) & $\bullet$ 2-component Hunter-Saxton (\autoref{sub:kahler_properties_of_madelung}) \\
		\bottomrule
	\end{tabular}
	\smallskip
	\caption{Examples of Newton's equations.}\label{tab:equations} 
\end{table} 

In what follows we will systematically describe how one can conveniently study various equations of mathematical physics, including all the examples listed in Table \ref{tab:equations}, from a unified point of view as certain Newton's equations. 
Our goal is to present a rigorous infinite-dimensional geometric framework that unifies 
Arnold's approach to incompressible and inviscid hydrodynamics and its relatives 
with 
various generalizations of Newton's equations \eqref{eq:newton_on_Q} 
such as those mentioned above, 
to provide a very general setting for systems of hydrodynamical origin 
on diffeomorphism groups and spaces of probability densities. We will also survey the setting of the Hamiltonian reduction, which establishes a correspondence between various representations of  these equations.
\end{remark}

\begin{remark} 
More precisely, 
given a compact $n$-dimensional manifold $M$ 
we will equip the group of diffeomorphisms $\Diff(M)$ 
and the space of nonvanishing probability densities $\Dens(M)$ 
with the structures of 
smooth infinite-dimensional manifolds 
(see \autoref{sect:tame} for details) 
and study Newton's equations on these manifolds viewed as the associated configuration spaces. 

As a brief preview of what follows,  let $\Diffvol(M)$ be the subgroup of diffeomorphisms 
preserving the Riemannian volume form $\mu$ of $M$.  
Consider the fibration of the group of all diffeomorphisms 
over the space of densities 
\begin{equation} \label{eq:moser} 
\Diff(M)/\Diffvol(M)\simeq \Dens(M), 
\end{equation} 
discussed by Moser \cite{Mo1965}, 
whose cotangent bundles $T^*\Diff(M)$ and $T^*\Dens(M)$ are related by a symplectic reduction, 
cf. \autoref{sect:hamiltons} below. 
Moser's construction can be used to introduce two different algebraic objects: 
the first is obtained by identifying $\Dens(M)$ with the \emph{left cosets} 
\begin{equation} \label{eq:left} 
\Diff(M)/\Diffvol(M) = \big\{ \varphi\circ\Diffvol(M)\mid \varphi\in\Diff(M) \big\}
\end{equation} 
and the second by identifying it with the \emph{right cosets} 
\begin{equation} \label{eq:right} 
\Diffvol(M)\backslash\Diff(M) =  \big\{ \Diffvol(M)\circ\varphi\mid \varphi\in\Diff(M) \big\}. 
\end{equation} 
In this paper we will make use of both identifications. 

In order to define Newton's equations on $\Diff(M)$ and $\Dens(M)$ and to investigate their mutual relations 
we will choose Riemannian metrics on both spaces so that the natural projections $\pi$ 
corresponding to \eqref{eq:left} or \eqref{eq:right} become (infinite-dimensional) Riemannian submersions. 
We will consider two such pairs of metrics. 
In \autoref{sec:wasserstein}, using left cosets, we will study a non-invariant $L^2$-metric on $\Diff(M)$ 
together with the Wasserstein-Otto metric on $\Dens(M)$. 
In \autoref{sec:fisher_rao}, using right cosets, we will focus on a right-invariant $H^1$ metric on $\Diff(M)$ 
and the Fisher-Rao information metric on $\Dens(M)$. 
Extending the results of \cite{Re2012}, 
we will then derive in \autoref{sec:madelung} various geometric properties of the Madelung transform. 
This will allow us to represent Newton's equations on $\Dens(M)$ as Schr\"odinger-type equations for wave functions. 

\end{remark} 


\subsection{Other related equations}

Newton's equations for fluids discussed in the present paper are assumed to be 
conservative systems with a potential force. 
However, the subject concerning Newton's equations is broader 
and we mention briefly two topics related to non-conservative Newton's equations 
for compressible and incompressible fluids that are beyond the scope of this paper.

First, observe that the dissipative term $\Delta v$ in the viscous Burgers equation 
$$
\dot v + \nabla_v v = \gamma \Delta v 
$$
can be viewed as a (linear) friction force while the equation itself can be seen as Newton's equation 
on $\Diff(M)$ with a non-potential force. 
Similarly, observe that the Navier-Stokes equations of a viscous incompressible fluid 
$$
\dot v + \nabla_v v + \nabla P = \gamma \Delta v, 
\quad 
\mathrm{div}\, v = 0 
$$
can be seen as Newton's equations on $\Diffvol(M)$ with a non-potential friction force. 
There is a large literature treating the Navier-Stokes equations within a stochastic framework 
where the geodesic setting of the Euler equations is modified by adding a random force 
which acts on the fluid, see \cite{Go2005, Ho2015}. 

\medskip

The second topic is related to a recently discovered flexibility and non-uniqueness 
of weak solutions of the Euler equations. The constructions in \cite{Sc1993, Sh1996, LeSz2017} 
exhibit compactly supported weak solutions 
describing a moving fluid that comes to rest as $t\to\pm\infty$.  
Such constructions can be understood by introducing a special forcing term $F$ 
(sometimes referred to as the ``black noise") 
into the equations 
$\dot v+\nabla_v v+\nabla P=F$, 
and require that it is ``$L^2$-orthogonal to all smooth functions." 
(More precisely, one constructs a family of solutions with increasingly singular and oscillating force 
and the ``black noise" is a residual forcing observed in the limit, cf. \cite{Ts2004}.) 
Using the standard definition of a weak solution this force is thus not detectable 
upon multiplication by smooth test functions 
and hence the existence of such solutions to the Euler equations becomes less surprising. 
Constructions of similar weak solutions to other PDEs 
rely on intricate limiting procedures involving possibly more singular and less detectable forces. 
The study of the geometry of Newton's equations with ``black noise" on diffeomorphism groups 
seems to be a promising direction of future research. 
\begin{remark} 
We should mention that, in addition to Newton's equations, there is another class of 
natural evolution equations on Riemannian manifolds given by the \emph{gradient flows} 
\begin{equation} \label{eq:gradient_flows} 
\dot q = -\nabla V(q) \,,
\end{equation} 
where a given potential function $V$ determines velocity rather than acceleration. 
An interesting example 
can be found in \cite{Ot2001} 
where the heat flow 
on $\Dens(M)$ is described as 
the gradient flow of the relative entropy functional 
providing a geometric interpretation of the second law of thermodynamics, 
cf.\ \autoref{rem:heat} on its relation to a Hamiltonian setting. 
\end{remark}

\smallskip 


\subsection{An overview and main results}
The goal of this paper is twofold. 
First, we present a 
survey of the differential geometric approach 
to several hydrodynamical equations emphasizing the setting of Newton's equations. 
Second, we describe new results obtained by implementing this tool. 

Here are some highlights of this paper, where the survey topics are intertwined with
new contributions: Geometry of the Euler equations as geodesic equations along with their Hamiltonian formulation;
Riemannian geometry of the spaces of diffeomorphisms and densities and their relation to problems of optimal mass transport;
Newton equations in infinite dimensions and their appearance in the geometry of compressible fluids;
Semidirect product groups in relation to compressible fluids and magnetohydrodynamics;
Fisher-Rao geometry on the spaces of densities and diffeomorphisms;
Geometric properties of the Madelung transform;
Casimirs of compressible and incompressible fluids and magnetohydrodynamics; 
we also recall briefly the symplectic and Poisson reductions in relation to diffeomorphism groups.
In more detail:

\begin{enumerate} 
\item 
Following \cite{Sm1979} and \cite{Eb1975} we  revisit the case of the compressible barotropic Euler equations 
as a Poisson reduction of Newton's equations on $\Diff(M)$ with the symmetry group $\Diffvol(M)$ 
and show that 
the Hamilton-Jacobi equation of fluid mechanics corresponds to its horizontal solutions  \autoref{sub:compressible_euler_equations}. 
We then describe the framework of Newton's equations 
for fully compressible (non-barotropic) fluids \autoref{sect:fully} 
and magnetohydrodynamics \autoref{sect:comprMHD}.

\item 
After reviewing the semidirect product approach to these equations we relate it to our approach 
in \autoref{sec:semi_direct_reduction}. 
We point out that 
the Lie-Poisson semidirect product algebra associated with the compressible Euler equations 
appears naturally in the Poisson reduction setting $T^*\Diff(M)\to T^*\Diff(M)/\Diffvol(M)$. 
We then show that the semidirect product structure is consistent with the symplectic reduction 
at zero momentum 
for $T^*\Diff(M) \sslash \Diffvol(M) \simeq T^*\Dens(M)$, 
see \autoref{sect:general-semi} and \autoref{sub:symplectic_reduction}.

\item 
We develop a reduction framework for relativistic fluids in \autoref{sec:relativistic_euler} 
and show how the relativistic Burgers equation arises in this context. 
We relate it to the relativistic approaches in optimal transport in \cite{Br2003c} 
and ideal hydrodynamics in \cite{LaLi1959,HoKu1984}.

\item 
Along with the $L^2$ and the Wasserstein-Otto geometries 
we also describe the geometry associated with 
the Sobolev $H^1$ and the {Fisher-Rao} metrics, see \autoref{sec:fisher_rao}. 
We show that infinite-dimensional Neumann systems are (up to time rescaling) 
Newton's equations for quadratic potentials in these metrics 
(in suitable coordinates the Fisher information functional is an example of such a potential), 
see \autoref{sect:neumann}. 

\item 
Using the approach presented in this paper 
we derive stationary solutions of the Klein-Gordon equation 
and show that 
they satisfy 
a stationary infinite-dimensional Neumann problem, 
see \autoref{sec:klein_gordon}. 
We also show that the generalized two-component Hunter-Saxton equation is a Newton's equation 
in the Fisher-Rao setting, see \autoref{sub:kahler_properties_of_madelung}. 

\item 
We review the properties of the Madelung transform which relates linear and nonlinear Schr\"odinger equations 
to Newton's equations on $\Dens(M)$ and can be used to describe horizontal solutions to Newton's equations 
on $\Diff(M)$ with $\Diffvol$-invariant potentials, see \autoref{sec:madelung} and \cite{Re2012,KhMiMo2019} 
as well as the so-called Schr\"odinger smoke \cite{ChKnPiScWe2016}. 

\item 
Finally, we describe the Casimirs for compressible barotropic fluids, compressible and incompressible 
magnetohydrodynamics, see \autoref{sec:casimirs}. 
\end{enumerate} 

\medskip
\noindent\textbf{Notations.} 
Unless indicated otherwise $M$ stands for a compact oriented Riemannian manifold. 
The spaces of smooth $k$-forms on $M$ are denoted by $\Omega^k(M)$, 
the spaces of smooth vector fields by $\Xcal(M)$ 
and those of smooth functions by $C^\infty(M)$. 
Given a Riemannian metric $\met$ on $M$ the symbol $\nabla$ stands for 
the gradient as well as for the covariant derivative of $\met$. 
The Riemannian volume form is denoted by $\vol$ and is assumed to be normalized: $\int_M\vol=1$. 
To simplify notation, we will often use typical vector calculus conventions 
$\abs{v}^2 = \met(v,v)$ and $u\cdot v = \met(u,v)$. 
The Lie derivative along a vector field $v$ will be denoted by $\LieD_v$. 
In our computations we will assume all the functionals to be differentiable with variational derivatives 
belonging to the corresponding smooth duals, unless indicated otherwise.

A Riemannian metric $\met$ on $M$ defines an isomorphism between the tangent and cotangent bundles. 
For a vector field $v$ on $M$ we will denote by $v^\flat$ the corresponding 1-form on $M$, namely 
$v^\flat = \met(v, \cdot)$. 
As usual, the inverse map will be denoted by $\sharp$. 
The pullback and pushforward maps of a tensor field $\beta$ by a diffeomorphism $\varphi$ 
will be denoted by $\varphi^*\beta$ and $\varphi_*\beta$ respectively.

\medskip

\noindent\textbf{Spaces of densities.}
The space of smooth probability densities on $M$ will be denoted by $\Dens(M)$ 
and will play an important role in the paper. 
It can be viewed in two ways. 
Either, as is common among mathematical physicists, an element of $\Dens(M)$ is a smooth real-valued function on $M$, constrained to be strictly positive everywhere and of unit mass with respect to the reference volume form $\vol$ on $M$.
Or, as is common among differential geometers, the elements of $\Dens(M)$ can be viewed as normalized volume forms 
on $M$.
The latter is geometrically more natural since a probability density transforms as a volume form. 
However, either view has its pros and cons making various formulas look simpler or more familiar
depending on the context. 
	Therefore, we shall retain both conventions in this paper and distinguish between them as follows: 
	a density thought of as a function will be denoted by $\rho$, 
	whereas the corresponding volume form will be denoted by $\varrho = \rho\vol$.
	Notice that if $\varphi$ is a diffeomorphism, then the equality $\varrho = \varphi_*\vol$ corresponds to 
	$\rho = \Jac{\varphi^{-1}}$, where $\operatorname{Jac}$ is the Jacobian with respect to $\vol$.  


\medskip 

\noindent\textbf{Acknowledgements.} 
The authors are grateful to the anonymous referee for many helpful remarks.
Part of this work was done while B.K. held the Pierre Bonelli Chair at the IHES. 
He was also partially supported by an NSERC research grant and a Simons Fellowship.
Part of this work was done while G.M. held the Ulam Chair visiting Professorship in University of Colorado at Boulder. 
K.M. was supported by the Swedish Foundation for International Cooperation in Research and Higher Eduction (STINT) grant No PT2014-5823, by the Swedish Research Council (VR) grant No 2017-05040, and by the Knut and Alice Wallenberg Foundation grant No WAF2019.0201.

\section{Wasserstein-Otto geometry} \label{sec:wasserstein} 

\subsection{Newton's equations on \texorpdfstring{$\Diff(M)$}{Diff(M)} } 
\label{sub:newton_s_equation} 

In this section we describe Newton's equations on the full diffeomorphism group. 
Following \cite{Ar1966,EbMa1970} we first introduce a (weak) Riemannian structure\footnote{A rigorous 
infinite-dimensional setting for diffeomorphisms and densities will be given in \autoref{sect:tame} below. 
Here, for simplicity, we emphasize only the underlying geometric structure, leaving aside technical issues.}. 
\begin{definition} \label{def:L2met} 
The {\it $L^2$-metric} on $\Diff(M)$ is given by 
\begin{equation} \label{eq:L2met} 
\Met_{\varphi}(\dot\varphi,\dot\varphi) 
= 
\int_M \abs{\dot\varphi(x)}^2 d\vol(x) 
= 
\int_M \abs{\dot\varphi}^2 \vol 
\end{equation} 
or, equivalently, after a change of variables 
\begin{equation} \label{eq:L2met_alt} 
\Met_{\varphi}(\dot\varphi, \dot\varphi) 
= 
\int_M \abs{v}^2\varphi_*\vol, 
\end{equation} 
where $\varphi \in \Diff(M)$, $\dot\varphi = v\circ\varphi\in T_\varphi\Diff(M)$ 
and $v$ is a vector field on $M$. 
\end{definition} 
%


{\it Newton's equation} on $\Diff(M)$ is a second order differential equation of the form 
\begin{equation} \label{eq:newton_L2} 
\nabla^{\Met}_{\dot\varphi}\dot\varphi 
= 
-\nabla^{\Met} U(\varphi), 
\end{equation} 
where 
$U\colon\Diff(M)\to\RR$ is a potential energy function 
and 
$\nabla^{\Met}$ is the covariant derivative of the $L^2$-metric. 
We are interested in the case in which potential energy depends on $\varphi$ 
implicitly via the associated density, i.e., 
\begin{equation} \label{eq:Ubar} 
U(\varphi) = \bar U(\rho), 
\end{equation} 
where $\rho = \Jac{\varphi^{-1}}$ and $\bar U\colon \Dens(M)\to\RR$ is a given functional. 
We always assume that $\bar U$ for each $\rho\in\Dens(M)$ has a variational derivative given as a smooth function $\frac{\delta\bar U}{\delta\rho}\in C^\infty(M)$.

%
%


A more explicit form of \eqref{eq:newton_L2} is given by the following theorem. 
\begin{theorem}[\cite{Sm1979, Sm2007}] \label{thm:newton_for_Ubar} 
Newton's equations on $\Diff(M)$ for the metric \eqref{eq:L2met} and a potential function \eqref{eq:Ubar} 
can be written as 
\begin{equation} \label{eq:newton_for_ubar} 
\nabla_{\dot\varphi}\dot\varphi 
= 
-\nabla \frac{\delta \bar U}{\delta \rho}\circ\varphi. 
\end{equation} 
In reduced variables $v=\dot\varphi\circ\varphi^{-1}$ and $\rho = \Jac{\varphi^{-1}}$ 
the equations assume the form 
\begin{equation} \label{eq:reduced_newton_diff} 
\left\{ 
\begin{aligned} 
&\dot v + \nabla_v v + \nabla  \frac{\delta \bar U}{\delta\rho} = 0 
\\ 
&\dot\rho + \divv(\rho v) = 0. \vphantom{\frac{\delta}{U}} 
\end{aligned} 
\right. 
\end{equation} 
%
\end{theorem} 
The right-hand side of the equations in \eqref{eq:newton_for_ubar} is a result of a direct calculation 
which we state in a separate lemma. 
%


%
\begin{lemma} \label{lem:nablaU} 
If $U$ is of the form \eqref{eq:Ubar} then 
\begin{equation*} 
\nabla^{\Met} U(\varphi) 
= 
\left( \nabla  \frac{\delta \bar U}{\delta \rho} \right) \circ \varphi, 
\end{equation*} 
where $\rho=\Jac{\varphi^{-1}}$. 
\end{lemma} 
\begin{proof} 
This lemma is essentially the divergence theorem on the infinite-dimensional space. 
The proof in terms of variations of diffeomorphisms and densities mimics the finite-dimensional one.

Since $\nabla^{\Met} U$ stands here for the gradient of $U$ in the $L^2$-metric \eqref{def:L2met} 
and $\varphi(t)$ is the flow of the vector field $v$, we have 

\begin{align*} \label{eq:def_gradG} 
\Met_{\varphi} \big( \nabla^{\Met} U(\varphi),\dot \varphi \big) 
&= 
\frac{\ud}{\ud t} \bar U(\underbrace{\Jac{\varphi^{-1}}}_{\rho}) 
= 
\pair{\tfrac{\delta \bar U}{\delta \rho},-\divv(\rho v)} 
\\ 
&= 
\pair{ \nabla\tfrac{\delta \bar U}{\delta \rho},\rho v} 
= 
\int_M  
\nabla \tfrac{\delta \bar U}{\delta \rho} \cdot (\dot\varphi\circ\varphi^{-1} )
 \varphi_*\vol  = \Met_\varphi\big( (\nabla\tfrac{\delta \bar U}{\delta \rho})\circ\varphi, \dot\varphi\big)
\end{align*} 
where we used that
$\frac{\ud}{\ud t}\Jac{\varphi^{-1}}\mu = -\LieD_v\varphi_*\mu = -\divv(\rho v)\mu$.
%
\end{proof} 
\begin{proof}[Proof of \autoref{thm:newton_for_Ubar}] 
The equations in \eqref{eq:newton_for_ubar} follow directly from \autoref{lem:nablaU} 
and the fact that the covariant derivative with respect to the $L^2$-metric is just the pointwise 
covariant derivative on $M$. 
%
The reduced equations in \eqref{eq:reduced_newton_diff} are derived in the Hamiltonian setting 
in \autoref{sub:poisson_reduction} below. 
\end{proof} 

The following special class of solutions to Newton's equations is of particular interest. 
\begin{proposition} \label{prop:gradient_solutions} 
The gradient fields $v = \nabla\theta$ form an invariant set of solutions of the reduced equations 
\eqref{eq:reduced_newton_diff}.
Expressed in $\theta$ and $\rho$, these solutions fulfill the Hamilton-Jacobi equations 
\begin{align} 
\begin{cases} 
\displaystyle
\dot\theta + \frac{1}{2}\abs{\nabla\theta}^2 = - \frac{\delta \bar U}{\delta \rho} 
\\
\dot\rho + \,\divv(\rho \nabla\theta) = 0. 
\end{cases} 
\end{align} 
\end{proposition} 
\begin{proof} 
This follows from a direct computation using the identity 
$ 
\nabla_{\nabla\theta}\nabla\theta = \frac{1}{2}\nabla\abs{\nabla\theta}^2. 
$ 
A geometric explanation for the appearance of the Hamilton-Jacobi equation will be given in the next section. 
\end{proof} 

An important point we want to emphasize in this survey is that a large number of 
interesting systems in mathematical physics originate as Newton's equations on $\Diff(M)$ 
corresponding to different choices of potential functions. 
A partial list of examples discussed here is given in \autoref{tbl:Wasserstein_examples}. 
We will also describe other systems on $\Diff(M)$ including the MHD equations or the relativistic as well as the fully compressible Euler equations.

We have already seen two different formulations of Newton's equations: 
the second order (Lagrangian) representation in \eqref{eq:newton_for_ubar} 
and the reduced first order (Eulerian) respresentation in \eqref{eq:reduced_newton_diff}. 
In order to obtain all the equations listed in \autoref{tbl:Wasserstein_examples} we will need two further formulations: 
one defined on the space of densities and another defined on the space of wave functions. 
We begin with the former, postponing wave functions until \autoref{sec:madelung}.


\begin{table}
	\begin{center}
	\begin{tabular}{lll}
	\toprule
	Equation on $\Diff(M)$  & Potential $\bar U(\rho)$ & Section \\
	\midrule
	inviscid Burgers'  & $0$ &  \autoref{sub:burgers}\\
	\addlinespace[0.5em]
	Hamilton-Jacobi & $\displaystyle \int_M V \rho\,\vol$, \; $V\in C^\infty(M)$ & \autoref{sub:hamilton_jacobi}\\
	\addlinespace[0.5em]
	shallow-water & $\displaystyle \int_M \frac{\rho^2}{2} \,\vol$ & \autoref{first_examples}\\
	\addlinespace[0.5em]
	barotropic compressible Euler & $\displaystyle \int_M e( \rho)\rho\,\vol$, \; $e\in C^\infty(\RR)$  & \autoref{sub:compressible_euler_equations}\\
	\addlinespace[0.5em]
	linear Schr\"odinger & $\displaystyle \int_M \Big(\abs{\nabla \sqrt{\rho}}^2+V\rho\Big)\vol$ & \autoref{sect:NLS}\\
	\addlinespace[0.5em]
	nonlinear Schr\"odinger (NLS) & $\displaystyle \int_M \Big(\abs{\nabla \sqrt{\rho}}^2 + \frac{\kappa \rho^2}{2}\Big)\vol$ & \autoref{sect:NLS}\\
	\addlinespace[0.5em]
	\bottomrule
	\end{tabular}
	\end{center}
	\caption{Various PDEs as Newton's equations on $\Diff(M)$.}
	\label{tbl:Wasserstein_examples}
\end{table}

\subsection{Riemannian submersion over densities} 
\label{sub:fibration} 

The space $\Dens(M)$ of smooth probability densities on $M$ is an open subset of the affine subspace of all smooth function (or $n$-forms) that integrate to one.
It can be given the structure of an infinite-dimensional manifold whose tangent bundle is trivial 
\begin{equation*} 
T\Dens(M)=\Dens(M)\times C^\infty_0(M) 
\end{equation*} 
where
$ 
C^\infty_0(M) 
= 
\big\{ f \in C^\infty(M)\mid \int_M f\,\vol = 0 \big\}. 
$ 
\begin{definition} 
The \emph{left coset projection} $\pi\colon\Diff(M)\to \Dens(M)$ between the space of diffeomorphisms 
and the space of probability densities is given by
\begin{equation}\label{eq:projection} 
\pi(\varphi) = \Jac{\varphi^{-1}} = \rho. 
\end{equation} 
or, equivalently, by pushforward of the reference volume form $\pi(\varphi) = \varphi_*\vol = \varrho$. 
\end{definition} 
This projection relates the $L^2$-metric \eqref{eq:L2met} and the following metric 
on the space of densities. 
\begin{definition} \label{def:otto_metric} 
The \emph{Wasserstein-Otto metric} is a Riemannian metric on $\Dens(M)$ given by 
\begin{equation} \label{eq:otto_metric} 
\MetW_\rho(\dot\rho,\dot\rho) 
= 
\int_M \theta\dot\rho\,\vol , 
\end{equation} 
where $\theta\in C^\infty(M)/\RR$ is defined by the transport equation
$$ 
\dot\rho + \divv(\rho \nabla\theta) = 0 
$$ 
and $\dot\rho \in C^\infty_0(M)$ is a tangent vector at $\rho \in \Dens(M)$.
\end{definition} 

\begin{figure}
	\centering
	%
	\begin{tikzpicture} 
		\node[anchor=south west, inner sep=0] (image) at (0,0) {\includegraphics[width=0.6\textwidth]{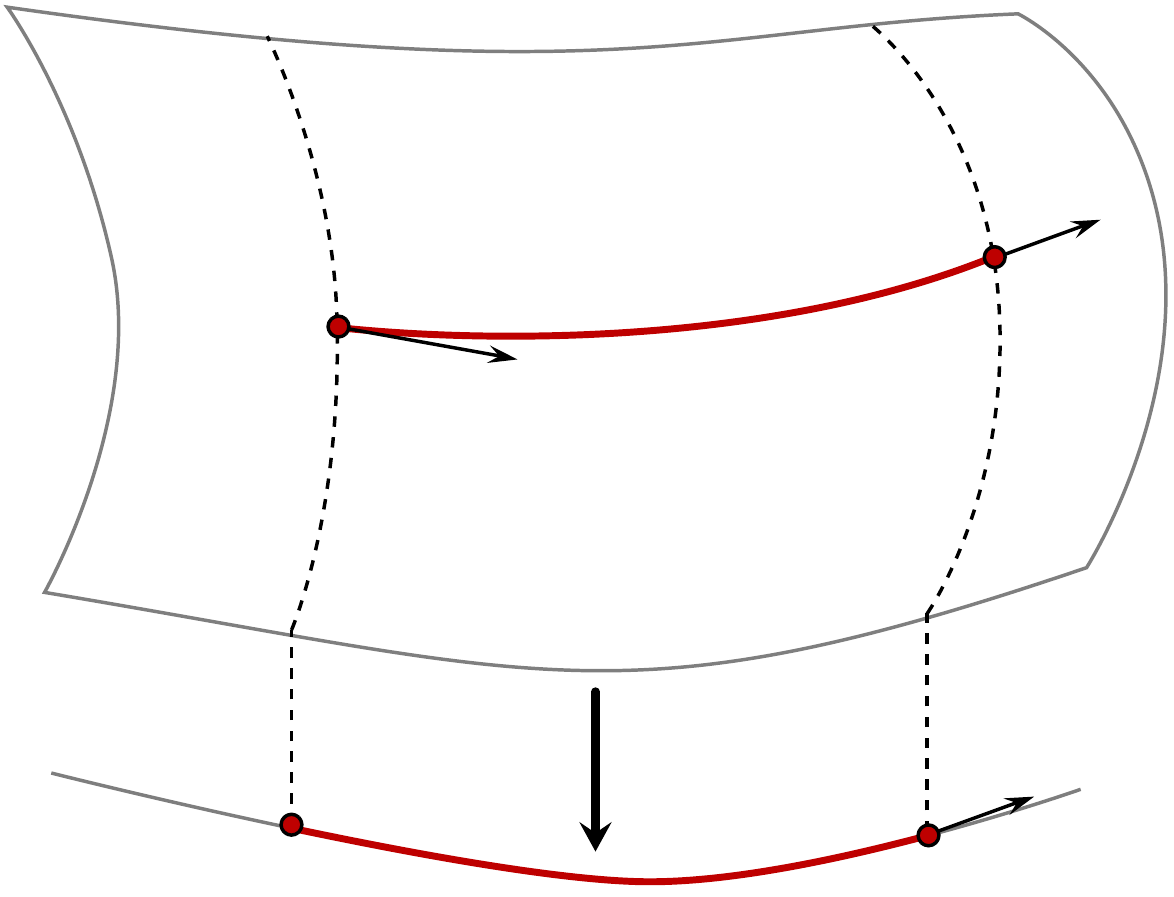}};
		\begin{scope}[x={(image.south east)},y={(image.north west)}]
			\coordinate (left_fiber) at (0.29,0.29) {};
			\coordinate (left_diffmu) at (0.21,0.3) {};
			\coordinate (right_fiber) at (0.84,0.33) {};
			\coordinate (hor_geo) at (0.57,0.64) {};
			\coordinate (geo) at (0.45,0.04) {};
			\coordinate (projection) at (0.51,0.16) {};
			\coordinate (diff) at (0.1,0.95) {};
			\coordinate (dens) at (0.1,0.1) {};
			\coordinate (id) at (0.29,0.64) {};
			\node[left, rotate=0] at (id) {$\mathrm{id}$};
			\coordinate (mu) at (0.245,0.07) {};
			\node[below, rotate=0] at (mu) {$1$};
			\coordinate (rho) at (0.79,0.065) {};
			\node[below, rotate=0] at (rho) {$\rho$};
			\coordinate (dotrho) at (0.86,0.155) {};
			\node[right, rotate=0] at (dotrho) {$\dot\rho$};
			\coordinate (u) at (0.43,0.59) {};
			\node[right, rotate=0] at (u) {$v=\nabla\theta$};
			\coordinate (phi) at (0.84,0.725) {};
			\node[below right, rotate=0] at (phi) {$\varphi$};
			\coordinate (phidot) at (0.92,0.77) {};
			\node[right, rotate=0] at (phidot) {$\dot\varphi$};
			\node[right, rotate=77] at (left_fiber) {fiber};
			\node[right, rotate=77] at (left_diffmu) {$\Diff_\mu$};
			\node[right, rotate=72] at (right_fiber) {fiber};
			\node[above, rotate=5] at (hor_geo) {horizontal geodesic};
			\node[below, rotate=-5] at (geo) {geodesic};
			\node[left, rotate=0] at (projection) {$\pi$};
			\node[above, rotate=-6] at (diff) {$\mathrm{Diff}(M)$};
			\node[above, rotate=-13] at (dens) {$\mathrm{Dens}(M)$};
		\end{scope}
	\end{tikzpicture}
	\caption{Illustration of the Riemannian submersion in \autoref{thm:otto_riemannian_metric}.
	Horizontal geodesics on $\Diff(M)$ (potential solutions) are transversal to the fibres and project to geodesics on $\Dens(M)$.
	Note that the point in $\Dens(M)$ denoted by $1$ corresponds to the reference volume form $\mu$, while $\rho$ corresponds to the volume density $\varrho$.
	}\label{fig:submersion}
\end{figure}

The Riemannian distance defined by the metric \eqref{eq:otto_metric} on $\Dens(M)$ 
is precisely the $L^2$ Kantorovich-Wasserstein distance of optimal transport, 
see \citet{BeBr2000}, \citet{Ot2001}, \citet{Lo2008} or \citet{Vi2009}.
\begin{theorem}[\cite{Ot2001}] \label{thm:otto_riemannian_metric} 
The projection \eqref{eq:projection} is an (infinite-dimensional) Riemannian submersion 
with respect to the $L^2$-metric $\Met$ on $\Diff(M)$ and the Wasserstein-Otto metric $\MetW$ on $\Dens(M)$. 
Namely, given a horizontal\footnote{That is, such that $\Met_\varphi(\dot\varphi,V) = 0$ for all $V\in \ker(T_\varphi\pi)$.} 
vector $\dot\varphi\in T_\varphi\Diff(M)$ one has 
\begin{equation} 
\Met_\varphi(\dot\varphi,\dot\varphi) = \MetW_{\pi(\varphi)}(\dot\rho,\dot\rho), 
\end{equation} 
where $\dot\rho = T_\varphi\pi (\dot\varphi)$. 
\end{theorem} 
An illustration of this theorem is given in \autoref{fig:submersion}.
The proof is based on two lemmas.
Recall that the left coset projection is the pushforward action of $\Diff(M)$ on $\vol$. 
The corresponding isotropy group is the subgroup of volume-preserving diffeomorphisms 
\begin{equation} 
\Diffvol(M) = \big\{ \varphi\in\Diff(M)\mid \varphi_*\vol = \vol \big\} 
\end{equation} 
so that, if $[\varphi]$ is a left coset in $\Diff(M)/\Diffvol(M)$ then $\varphi'\in [\varphi]$ 
if and only if there exists $\eta\in\Diffvol(M)$ such that $\varphi\circ\eta = \varphi'$. 

The first lemma states in particular that the action of $\Diff(M)$ on $\Dens(M)$ is transitive. 
%
\begin{lemma} \label{lem:moser} 
Let $\pi\colon\Diff(M)\to \Dens(M)$ be the left coset projection \eqref{eq:projection}. 
Then 
\begin{equation} \label{eq:fibration} 
\xymatrix@C-0.1pc{ 
 \Diff(M)\ar[d]^{\pi} & ~~\Diffvol(M) \ar@{_{(}->}[l] 
 \\ 
\Dens(M) & 
} 
\end{equation} 
is a principal bundle. 
Consequently, the quotient space $\Diff(M)/\Diffvol(M)$ of left cosets is isomorphic to $\Dens(M)$. 
\end{lemma} 
\begin{proof} 
Surjectivity of $\pi$ is a consequence of Moser's lemma \cite{Mo1965}. 
The fact that $\pi$ defines an infinite-dimensional principal bundle in the category of tame Fr\'echet manifolds 
(cf.\ \autoref{sect:tame} below) follows from a standard argument using the Nash-Moser-Hamilton theorem, 
cf. \cite{Ha1982}. 
\end{proof} 
%


The second lemma states that the $L^2$-metric on $\Diff(M)$ is compatible 
with the principal bundle structure above.
%
%
\begin{lemma} 
The $L^2$-metric \eqref{eq:L2met} is right-invariant with respect to the $\Diffvol(M)$ action, 
namely 
\begin{equation} \label{eq:L2invariance} 
\Met_{\varphi}(u,v) = \Met_{\varphi\circ\eta}(u\circ\eta,v\circ\eta) 
\end{equation} 
for any $u, v \in T_{\varphi}\Diff(M)$ and $\eta \in \Diffvol(M)$. 
\end{lemma} 
\begin{proof} 
Since $\eta_\ast \mu = \mu$ the result follows at once from \eqref{eq:L2met_alt}. 
\end{proof} 

In \cite{Ar1966} Arnold used the $L^2$-metric \eqref{eq:L2met} to show that 
its geodesic equation on $\Diffvol(M)$, when expressed in Eulerian coordinates, 
yields the classical Euler equations of an ideal fluid. 
This marked the beginning of \emph{geometric and topological hydrodynamics}, 
cf.\ \cite{ArKh1998} or \autoref{sect:incompressible}. 

The Riemannian submersion framework described above concerns objects that are 
extrinsic to Arnold's (intrinsic) point of view. 
More precisely, rather than restricting to the vertical directions tangent to the fibre $\Diffvol(M)$, 
we consider the horizontal directions in the total space $\Diff(M)$ and use the fact that 
any structure on $\Diff(M)$ which is invariant under the right action of $\Diffvol(M)$ 
induces a corresponding structure on $\Dens(M)$ by \autoref{lem:moser}. 

We are now ready to  prove the main result of this subsection.
%
%
\begin{proof}[Proof of \autoref{thm:otto_riemannian_metric}] 
Given $\xi \in T_\varphi\Diff(M)$ let 
$v = \xi\circ\varphi^{-1}$ and $\varrho = \varphi_*\vol$ with $\varrho = \rho\,\vol$ as before. 
Then 
\begin{equation} \label{eq:Tpi} 
\begin{split} 
T_\varphi \pi \cdot \xi 
&= 
-\LieD_v(\varphi_*\vol) = -\LieD_u\varrho 
\\ 
&= 
- \big( \rho\divv{v} + \iota_v\ud\rho \big) \vol 
\\ 
&= 
-\divv(\rho v)\vol. 
\\ 
\end{split} 
\end{equation} 
%
The kernel of $T_\varphi\pi$ is $T_{\varphi}\Diff_{\varrho}(M)$ and defines a vertical distribution. 
On the other hand, the horizontal distribution is 
$$ 
\mathcal{H}_{\varphi} 
= 
\big\{ \xi \in T_{\varphi}\Diff(M) ~|~ \xi\circ\varphi^{-1} = \nabla p ~{\text{ for}}~p\in C^{\infty}(M) \big\}.
$$ 
Indeed, if $\divv(\rho v) = 0$ then from \eqref{eq:L2met} we have 
\begin{equation} \label{eq:horiz_calculations} 
\mathsf{G}_\varphi \big( \nabla p\circ\varphi,v\circ\varphi \big) 
= 
\int_M \met(\nabla p,v)\varrho 
= 
\int_M (\LieD_v p) \varrho = -\int_M p\, \LieD_v\varrho = 0 
\end{equation} 
and it follows that $T_{\varphi}\pi\colon \mathcal{H}_{\varphi} \to T_{\varrho}\Dens$ is an isometry. 
Its inverse is 
\begin{equation} \label{eq:lift_proj_inverse} 
T_\rho\Dens(M)\ni\dot{\rho} \mapsto \nabla(-\Delta_\rho^{-1} \dot\rho)\circ\varphi \in \mathcal{H}_{\varphi} 
\end{equation} 
where $\Delta_\rho = \mathrm{div}\, \rho \nabla$. 
From \eqref{eq:L2met} we now compute 
\begin{equation*} \label{eq:wasserstein_metric} 
\begin{split} 
{\mathsf{G}}_{\varphi}(\nabla ( \Delta_\rho^{-1}\dot \rho )\circ\varphi, \nabla (\Delta_\rho^{-1}\dot \rho )\circ\varphi) 
&= 
\int_M \met \big( \nabla ( \Delta_\rho^{-1}\dot \rho ), \nabla ( \Delta_\rho^{-1}\dot\rho ) \big) \rho\mu 
\\ 
&= 
\int_M -\divv \big( \rho\nabla ( \Delta_\rho^{-1}\dot\rho ) \big) \Delta_\rho^{-1}\dot\rho\,\vol 
\\
&= 
\int_M -\dot\rho\Delta_\rho^{-1}\dot\rho\,\vol 
=
\bar{\mathsf{G}}_{\rho}(\dot{\rho},\dot{\rho}) 
\end{split} 
\end{equation*} 
where the last equality follows from the definition of $\bar{\mathsf{G}}$.
Thus, the projection $\pi$ is a Riemannian submersion.
\end{proof} 
\begin{remark} 
If $E$ is a smooth functional on $\Dens(M)$ with a variational derivative $\frac{\delta E}{\delta\rho}\in C^\infty(M)$ for every $\rho\in\Dens(M)$ then from the above expression we have 
\begin{equation} \label{eq:wasserstein_gradient_calc} 
\bar{\mathsf{G}}_{\rho} \big( \nabla^\MetW E(\rho),\dot\rho \big) 
= 
\pair{\frac{\delta E}{\delta \rho},\dot\rho}_{L^2} 
= 
\int_M \Delta_{\rho}^{-1}\Delta_{\rho}\frac{\delta E}{\delta \rho} \, \dot\rho 
= 
\bar{\mathsf{G}}_{\varrho} 
\big(  {-}\Delta_{\rho}\tfrac{\delta E}{\delta \rho} , \dot\rho \big) 
\end{equation} 
which gives the following formula for the gradient of $E$ in the Wasserstein-Otto metric 
\begin{equation} \label{eq:wasserstein_gradient} 
\nabla^\MetW E(\rho) 
= 
\Big( 
-\Delta_{\rho}\frac{\delta E}{\delta \rho} 
+ 
\int_M \Delta_{\rho}\frac{\delta E}{\delta \rho}\vol 
\Big) \vol 
\end{equation} 
since every vector tangent to $\Dens(M)$ has zero mean. 
In particular, if $E$ is the relative entropy $S(\rho) = \int_M \ln(\rho)\rho\,\vol$, 
then 
${\delta S}/{\delta \rho} = \ln \rho$ 
and since 
\begin{equation} 
\Delta_\rho \ln\rho = \divv(\rho\nabla(\ln\rho)) = \divv\nabla\rho = \Delta\rho 
\end{equation} 
we recover the formula $\nabla^\MetW S(\rho) = -\Delta\rho$, 
i.e., the Wasserstein gradient flow of entropy corresponds to the heat flow on the space of densities, 
cf. \cite{Ot2001}. 
\end{remark} 
%

\section{Hamiltonian setting} \label{sect:hamiltons} 

The point of view of incompressible hydrodynamics as a Hamiltonian system on the cotangent bundle of $\Diffvol(M)$ 
described by Arnold \cite{Ar1966} turned out to be remarkably useful in applications involving invariants and stability 
(this is reviewed in \autoref{sect:incompressible}).
In the next sections we develop the framework for Newton's equations (adding a potential energy term 
to the kinetic energy which yields geodesics) on the group $\Diff(M)$ of all diffeomorphisms 
(rather than volume-preserving ones).

\subsection{Hamiltonian framework for the incompressible Euler equations}\label{sect:incompressible}  
\citet{Ar1966} suggested to use the following general framework on an arbitrary group describing a geodesic flow with  respect to a suitable one-sided invariant Riemannian metric on this group. (Similar ideas can be traced back to S.~Lie and H.~Poincar\'e \cite{Li1890, Po1901}.)

Let a (possibly infinite-dimensional) Lie group $G$
be the configuration space of some physical system. The tangent space at the identity element $e$ 
is the corresponding Lie algebra
$\mathfrak g=T_eG$. Fix a positive definite quadratic form (the ``energy'') 
$E(v)=\frac 12\langle v,Av\rangle$
on  $\mathfrak g$ and right translate it to the tangent space $T_aG$ at any point $a\in G$ 
(this is ``translational symmetry'' of the energy).
In this way the energy defines a  right-invariant
Riemannian metric on the group. 
The geodesic flow on $G$ with respect to this energy metric
represents  extremals of the least action principle, i.e.,
actual motions of the physical system.

The operator   $A\colon\mathfrak g\to \mathfrak g^*$ defining the energy $E$ (and called the inertia
operator) allows one to rewrite the Euler equation on the dual space
 $\mathfrak g^*$. 
The Euler equation on $\mathfrak g^*$ turns out to be 
Hamiltonian with respect to the natural Lie--Poisson structure on
the dual space. The corresponding Hamiltonian function is  the energy
quadratic form lifted from the Lie algebra to its
dual space by the same identification:
$H(m)=\frac 12\langle A^{-1}m,m\rangle$, where $m=Av$.
Now the Euler equation on $\mathfrak g^*$ corresponding to the 
right-invariant metric $E$ on the group 
is given by 
\begin{equation}\label{eq:linEuler}
\dot m=-{\rm ad}^*_{A^{-1}m}m,
\end{equation}
as an evolution of a point $m\in \mathfrak g^*$, see e.g.\ \cite{ArKh1998}.
Here ${\rm ad}^*$ is the operator of the coadjoint representation of the Lie algebra 
$\mathfrak g$ on its dual $\mathfrak g^*$: 
$\langle {\rm ad}^*_v u,w\rangle := \langle u, [v,w]\rangle$ for any elements
$v,w\in \mathfrak g$ and $u\in \mathfrak g^*$.

\smallskip

Applied to the group $G=\Diffvol(M)$ of volume preserving diffeomorphisms on $M$,
this framework provides an infinite-dimensional Riemannian setting for the Euler equations \eqref{eq:euler} 
of an ideal fluid in $M$. 
Namely, the right-invariant energy metric is given here by 
the $L^2$-inner product on divergence-free vector fields on $M$ that constitute the Lie algebra 
$\mathfrak g = {\Xcalvol}(M) = \{ v\in \Xcal(M)~|~\mathcal \divv(v) = 0 \}$. 
The equations~\eqref{eq:linEuler} in this particular setting then correspond to the incompressible Euler equations \eqref{eq:euler}.
The approach also provides the following Hamiltonian framework for classical hydrodynamics. 
\begin{theorem} {\rm (see e.g. \cite{ArKh1998})} \label{thm:dual_inc} 

a) The dual space to the Lie algebra ${\Xcalvol}(M)$ 
is ${\Xcal}_\mu^*(M)=\Omega^1(M)/\ud C^\infty(M)$, 
the space of cosets of 1-forms on $M$ modulo exact 1-forms. 
The coadjoint action of $\Diffvol(M)$ is given by change of coordinates in a 1-form, 
while the coadjoint action of $\Xcalvol(M)$ is given by the Lie derivative along a vector field 
${\rm ad}^*_v=\mathcal L_v$; 
it is well-defined on the cosets in $\Omega^1(M)/\ud C^\infty(M)$. 

b) The inertia operator $A: {\Xcalvol}(M) \to {\Xcal}_\mu^*(M)$ is defined by 
assigning to a given divergence-free vector field $v$ the coset $\alpha = [v^\flat]$ in $\Omega^1(M)/\ud C^\infty(M)$. 

c) The incompressible Euler equations \eqref{eq:euler} on the dual space ${\Xcal}_\mu^*(M)$ have the form 
\begin{equation} \label{eq:alfa} 
\frac{\ud}{\ud t}[\alpha] = -\LieD_v[\alpha], 
\end{equation} 
where $[\alpha] \in \Omega^1(M)/\ud C^\infty(M)$ and $\alpha = v^\flat$. 
\end{theorem} 
The proof follows from the fact that the map $v\mapsto \iota_v\mu$ provides an isomorphism of 
the space of divergence-free vector fields and the space of closed $(n-1)$-forms on $M$, 
i.e., 
$\mathfrak g = {\Xcalvol}(M) \simeq \Omega^{n-1}_{cl}(M)$, 
since $\ud(\iota_v\mu)=\mathcal L_v\mu=0$. 
The dual space is 
$\mathfrak g^* = ( \Omega^{n-1}_{cl}(M) )^* =  \Omega^1(M)/\ud C^\infty(M)$ 
and the pairing is given by 
$$ 
\langle v, [\alpha] \rangle = \int_M ( \iota_v \alpha ) \,\mu. 
$$ 
For more details we refer to \cite{ArKh1998}. 
\begin{remark} \upshape 
Equation \eqref{eq:alfa} can be rewritten in terms of a representative 1-form 
and a differential of a (pressure) function 
$$ 
\dot\alpha + \mathcal L_v \alpha = -\ud P 
$$ 
which is a more familiar form of the Euler equations of an ideal fluid. 

Note that each coset $[\alpha]$ contains a unique coclosed 1-form $\bar\alpha \in [\alpha]$ 
which is related to a {\it divergence-free} vector field $v$ by means of the metric on $M$, 
namely $\bar\alpha = v^\flat$. 
Such a choice of a representative $\bar\alpha$ defines the (pressure) function $P$ 
uniquely modulo a constant since $\Delta P = \delta \ud P$ is prescribed for each time $t$. 
\end{remark} 
%

\subsection{Hamiltonian formulation and Poisson reduction} 
\label{sub:poisson_reduction} 

Newton's equations \eqref{eq:newton_for_ubar} can  be viewed as a canonical Hamiltonian system 
on $T^*\Diff(M)$. To write down this system we identify each co-tangent space 
 $T_\varphi^*\Diff(M)$ 
with the dual of the space of vector fields $\Xcal^*(M) = T^*_{\id}\Diff(M)$. 
The (smooth part of the) latter space consists of differential 1-forms with values in the space of densities 
\begin{equation} 
\Xcal^*(M) = \Omega^1(M)\otimes \Dens(M) 
\end{equation} 
where the tensor product is taken over the ring $C^\infty(M)$. 
The natural pairing between $\dot\varphi\in T_\varphi\Diff(M)$ and $m =\alpha\otimes\vol \in T_\varphi^*\Diff(M)$ 
is given by 
\begin{equation} \label{eq:pairing_diff} 
\pair{m, \dot\varphi}_\varphi = \int_M \interior_{\dot\varphi\circ\varphi^{-1}} m = \int_M (\interior_{\dot\varphi\circ\varphi^{-1}}\alpha) \vol
\end{equation} 
(when $\varphi = \mathrm{id}$ we will sometimes omit the subscript). 
This pairing does not depend on the Riemannian metric $\met$ on $M$.

\begin{figure}
\begin{equation*}
	\xymatrix@C-2.0pc{
 	& 	T^*\Diff(M)  \ar[dr]^{\text{}}\ar[rr] &  	& 					\Diff(M)\ar[dd]			 \\
	J^{-1}([0])\ar@{^{(}->}[ur] \ar[dr]^{\text{}} 	& 		& \Dens(M)\times\Xcal^*(M) 	& 								 \\
	& 	T^*\Dens(M) \ar@{^{(}->}[ur]^{\text{}} \ar[rr]^{\text{}} & 		&  \Dens(M) \\ 
	}
\end{equation*}

\caption{Relation between various phase space representations of Newton's equations on $\Diff(M)$ and $\Dens(M)$.}	
\label{fig:diagram}
\end{figure}
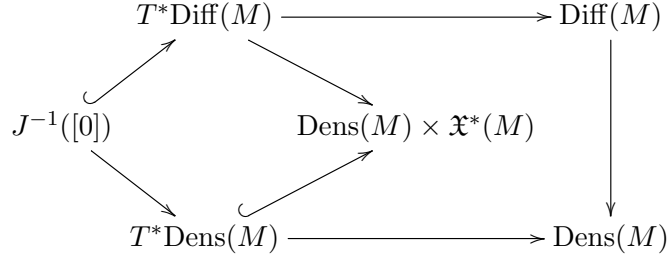

\begin{remark}
The spaces discussed and maps between them are summarized in the following commutative diagram in 
\autoref{fig:diagram}.
The right column of the diagram describes a natural projection $\pi \colon \Diff(M) \to \Dens(M)$ 
from the diffeomorphism group to the space of normalized smooth densities on $M$ 
with fibers that consist of all those diffeomorphisms which push a given reference density to any other density. 
As we discussed, this projection is a Riemannian submersion  for a (non-invariant) $L^2$-metric on $\Diff(M)$ 
and the (Kantorovich-)Wasserstein-Otto metric on $\Dens(M)$ used in the optimal mass transport, see \cite{Ot2001}.
The symplectic viewpoint on the Riemannian submersion leads naturally to the Hamiltonian description of the corresponding equations and the appearance of the momentum map, as we discuss below and in \autoref{sub:symplectic_reduction}.

The same diagram  arises for a different Riemannian submersion, 
when $\Diff(M)$ is equipped with a right-invariant homogeneous Sobolev $\dot{H}^1$-metric 
and $\Dens(M)$ with the Fisher-Rao (information) metric, 
which plays an important role in geometric statistics, as we discuss in \autoref{sec:fisher_rao}, cf. \cite{KhLeMiPr2013}.
\end{remark}

Consider the standard Lagrangian on $T\Diff(M)$ in the kinetic-minus-potential energy form 
$$ 
L(\varphi,\dot\varphi) = \frac{1}{2} \Met_\varphi(\dot\varphi,\dot\varphi) - \bar U(\varphi_*\vol). 
$$ 
As usual, the passage to the Hamiltonian formulation on $T^\ast\Diff(M)$ is obtained through 
the Legendre transform which in this case is given by 
\begin{equation} 
m = v^\flat\otimes \varrho 
\quad\text{where}\quad 
v = \dot\varphi\circ\varphi^{-1} 
\;\; \text{and} \;\; 
\varrho = \varphi_*\vol. 
\end{equation}
(In this section it is more convenient to work with the volume form $\varrho$ instead of the density function $\rho$.)
\begin{lemma} \label{lem:legendre_on_diff} 
The Hamiltonian corresponding to the Lagrangian $L$ is 
\begin{equation} \label{eq:ham_Ubar} 
H(\varphi,m) = \frac{1}{2}\pair{m,v} + \bar U(\varphi_*\vol). 
\end{equation} 
\end{lemma} 
\begin{proof} 
In the above notation for $\dot\varphi\in T_\varphi\Diff(M)$ we have 
\begin{equation} 
\Met_\varphi(\dot\varphi,\dot\varphi) 
= 
\int_M \abs{v}^2\varrho 
= 
\int_M (\interior_v v^\flat) \varrho 
= 
\big\langle v^\flat\otimes \varrho, \dot\varphi \big\rangle_{\varphi}. 
\end{equation} 
The result follows since $\Met_\varphi$ is quadratic and $\bar U$ is independent of $\dot\varphi$. 
\end{proof} 

We can now turn to Newton's equations on $\Diff(M)$. 
\begin{theorem} \label{prop:ham_form_newton_on_diff} 
The Hamiltonian form of the equations \eqref{eq:newton_for_ubar} is 
\begin{align} \label{eq:ham_form_of_newton_otto} 
\begin{cases} 
&\dot m 
= 
- \LieD_v m 
+ 
\ud \Big( \frac{1}{2} \abs{v}^2
- 
\frac{\delta \bar U}{\delta\varrho} (\varphi_*\vol) \Big) \otimes \varphi_*\vol 
\\
&\dot\varphi = v\circ\varphi 
\end{cases} 
\end{align} 
where $m = v^\flat \otimes \varphi_*\vol$. 
\end{theorem} 
\begin{proof} 
In canonical coordinates $(\varphi,m_\varphi)$ on $T^*\Diff(M)$ the Hamiltonian equations take the form 
\begin{align*} 
\dot m_\varphi 
= 
-\frac{\delta H}{\delta\varphi} 
\quad \text{and} \quad 
\dot \varphi 
= 
\frac{\delta H}{\delta m_\varphi} 
\end{align*} 
where $m_\varphi$ are the canonical momenta satisfying $m = m_\varphi\circ\varphi^{-1}$. 

Given a Hamiltonian $H$ on $T^\ast\Diff(M) \simeq \Diff(M) \times \Xcal^*(M)$ and 
a variation $\epsilon \to m_{\varphi,\epsilon}$ we have 
\begin{equation} 
\frac{\ud}{\ud\epsilon}H \big( \varphi,m_{\varphi,\epsilon}\circ\varphi^{-1} \big) 
= 
\Big\langle 
\frac{\delta H}{\delta m}, \frac{\ud}{\ud\epsilon}m_{\varphi,\epsilon}\circ\varphi^{-1} 
\Big\rangle_{\id} 
= 
\Big\langle 
\frac{\delta H}{\delta m}\circ\varphi, \frac{\ud}{\ud\epsilon} m_{\varphi,\epsilon} 
\Big\rangle_{\varphi} 
\end{equation} 
and thus 
$\dot\varphi = \frac{\delta H}{\delta m_\varphi} = v \circ \varphi$ 
where $v = \frac{\delta H}{\delta m}$. 
Differentiating $m_\varphi = m\circ\varphi$ with respect to the $t$ variable we obtain 
\begin{equation} 
\dot m 
= 
-\LieD_v m + \dot{m}_\varphi\circ\varphi^{-1} 
= 
-\LieD_v m - \frac{\delta H}{\delta \varphi}\circ\varphi^{-1}. 
\end{equation} 
As before, writing the Hamiltonian in \eqref{eq:ham_Ubar} as 
$H(\varphi, m) = \bar{H}(\varrho, m)$ where $\varrho = \varphi_\ast \mu$ 
and letting $\epsilon \to \varphi_\epsilon$ be a variation of $\varphi_0 = \varphi$ generated by the field $v$ 
we find 
\begin{equation} 
\frac{\ud}{\ud\epsilon}H(\varphi_\epsilon,m) 
= 
\Big\langle \ud \frac{\delta \bar H}{\delta \varrho},v \Big\rangle_{\id} 
= 
\Big\langle \ud \frac{\delta \bar H}{\delta \varrho}\circ\varphi,\dot\varphi \Big\rangle_{\varphi}. 
\end{equation} 
Thus 
$\frac{\delta H}{\delta\varphi} = \ud \frac{\delta \bar H}{\delta \varrho} \circ\varphi$ 
and the equation for $\dot m$ becomes 
\begin{equation} 
\dot m = -\LieD_v m - \ud \frac{\delta \bar H}{\delta \varrho}. 
\end{equation} 

Finally, a straightforward computation using \eqref{eq:ham_Ubar} gives 
\begin{equation} 
\frac{\delta \bar H}{\delta \varrho}= - \frac{\abs{v}^2}{2} + \frac{\delta \bar U}{\delta\varrho}\,,
\end{equation} 
which concludes the proof. 
\end{proof} 

Rewriting the system \eqref{eq:ham_form_of_newton_otto} in terms of $\varrho = \varphi_\ast \mu$ and $m$ 
provides an example of Poisson reduction with respect to $\Diffvol(M)$ as the symmetry group. 
From \eqref{eq:pairing_diff} we obtain a formula for the cotangent action of this group on $T^\ast\Diff(M)$, 
namely 
\begin{equation} \label{lem:cotangent_left_action} 
\eta \cdot (\varphi,m) = \big( \varphi\circ\eta^{-1},m \big). 
\end{equation} 
\begin{theorem}[Poisson reduction] \label{thm:poisson_reduction} 
The quotient space $T^*\Diff(M)/\Diffvol(M)$ 
is isomorphic to $\Dens(M)\times \Xcal^*(M)$. 
The isomorphism is given by the projection
$\Pi(\varphi,m) = (\varphi_*\vol,m)$. 
Furthermore, $\Pi$ is a Poisson map with respect to the canonical Poisson structure on $T^*\Diff(M)$ 
and the Poisson structure on $\Dens(M)\times\Xcal^*(M)$ given by 
\begin{equation} \label{eq:poisson_bracket_reduced} 
\{ F , G \}(\varrho, m) 
= 
\Big\langle 
\varrho, \LieD_{\frac{\delta F}{\delta m}}\tfrac{\delta G}{\delta \varrho} 
- 
\LieD_{\frac{\delta G}{\delta m}}\tfrac{\delta F}{\delta \varrho} \Big\rangle 
+ 
\Big\langle m, \LieD_{\frac{\delta F}{\delta m}}\tfrac{\delta G}{\delta m} \Big\rangle. 
\end{equation} 
\end{theorem} 
In \autoref{sect:tame} below we provide an alternative construction in the setting of Fr\'echet spaces. 
\begin{proof} 
From \eqref{lem:cotangent_left_action} and \autoref{lem:moser} it follows that 
\begin{equation} 
T^*\Diff(M)/\Diffvol(M) \simeq \Diff(M)/\Diffvol(M)\times \Xcal^*(M) 
\simeq 
\Dens(M)\times \Xcal^*(M) 
\end{equation} 
with the projection given by $\Pi$. The fact that $\Pi$ is a Poisson map (in fact, a Poisson submersion) 
follows from the consideration in  \autoref{sub:symplectic_reduction}.
\end{proof} 
\begin{remark} 
The bracket \eqref{eq:poisson_bracket_reduced} is the classical Lie-Poisson structure on the dual of the semidirect product 
$\Xcal(M) \ltimes  C^\infty(M)$. 
\end{remark} 
\begin{corollary} \label{cor:poisson_system_reduced} 
Let $H$ be a Hamiltonian function on $T^*\Diff(M)$ satisfying 
\begin{equation} 
H(\varphi,m) = H(\varphi\circ\eta,m) 
\quad \text{for all} \quad 
\eta\in\Diffvol(M). 
\end{equation} 
Then $H = \bar H \circ \Pi$ for some function $\bar H\colon \Dens(M)\times\Xcal^*(M)\to \RR$. 
In reduced variables $\varrho = \varphi_*\vol$ and $m$, the Hamiltonian equations assume the form 
\begin{equation} \label{eq:poisson_system_reduced} 
\left\{\; 
\begin{aligned} 
&\dot m = - \LieD_v m - \ud \tfrac{\delta \bar H}{\delta \varrho} \otimes \varrho \\ 
& \dot \varrho = - \LieD_v \varrho	\vphantom{\frac{\delta}{U}} 
\end{aligned} 
\right. 
\end{equation} 
where $v = \frac{\delta \bar H}{\delta m}$. 
\end{corollary} 
\begin{proof} 
Using the Poisson form of the Hamiltonian equations 
$\dot F = \left\{H,F\right\}$ 
with 
$F(\rho,m) = \pair{m,u} + \pair{\varrho, \theta}$ 
we obtain from \eqref{eq:poisson_bracket_reduced} the weak form of the equations 
$$ 
\langle \dot m, u \rangle + \langle \dot\varrho, \theta \rangle 
= 
\Big\langle \varrho, \LieD_{\frac{\delta \bar H}{\delta m}}\theta - \LieD_{u}\tfrac{\delta \bar H}{\delta \varrho} \Big\rangle 
+ 
\Big\langle m, \LieD_{\frac{\delta \bar H}{\delta m}}u \Big\rangle 
$$ 
for any $u \in \Xcal(M)$ and $\theta \in C^\infty(M)/\RR$. 
Rewriting the right-hand side as 
\begin{align} 
\Big\langle {-} \LieD_{\frac{\delta \bar H}{\delta m}}\varrho, \theta \Big\rangle 
+ 
\Big\langle {-} \LieD_{\frac{\delta \bar H}{\delta m}} m - \ud \tfrac{\delta \bar H}{\delta \varrho} \otimes\varrho, u \Big\rangle
\end{align} 
completes the proof. 
\end{proof} 

The following is the Hamiltonian analogue of \autoref{prop:gradient_solutions}. 
\begin{proposition} \label{prop:symplectic_leaf} 
The product manifold 
$$ 
\Dens(M)\times (\ud C^\infty(M)\otimes \varrho) 
= 
\big\{ (\varrho,\ud \theta\otimes\varrho)\mid \theta\in C^\infty(M) \big\} 
$$ 
is a Poisson submanifold of $\Dens(M)\times\Xcal^*(M)$. 
\end{proposition} 
\begin{proof} 
From \eqref{eq:poisson_system_reduced} we find that the momenta $m=\ud C^\infty(M)\otimes \varrho$ 
form an invariant set in $\Omega^1(M)\otimes \varrho$ for any choice of Hamiltonian $\bar H$. 
\end{proof} 

It turns out that the submanifold in \autoref{prop:symplectic_leaf} is symplectic, as we shall discuss below. 



\subsection{Newton's equations on \texorpdfstring{$\Dens(M)$}{Dens(M)}} \label{sub:symplectic?} 

Poisson reduction with respect to the cotangent action of $\Diffvol(M)$ on $T^*\Diff(M)$ 
leads to reduced dynamics on the Poisson manifold $T^*\Diff(M)/\Diffvol(M)\simeq\Dens(M)\times \Xcal^*(M)$ 
(cf. \autoref{thm:poisson_reduction}). 
This Poisson manifold is a union of symplectic leaves 
one of which can be identified with $T^*\Dens(M)$ equipped with the canonical symplectic structure. 
Indeed, the latter turns out to be  the symplectic quotient $T^*\Diff(M)\sslash\Diffvol(M)$ 
corresponding to the zero-momentum leaf, see  \autoref{sub:symplectic_reduction}. 
Here we identify $T^*\Dens(M)$ as a symplectic submanifold of $\Dens(M)\times \Xcal^*(M)$. 
\begin{lemma} \label{lem:momentum_map_on_dens} 
The (smooth part of the) cotangent bundle of $\Dens(M)$ is 
\begin{equation} 
T^*\Dens(M) 
= 
\Dens(M)\times C^\infty(M)/\RR 
\simeq 
\Dens(M)\times (\ud C^\infty(M)\otimes \rho\vol). 
\end{equation} 
Furthermore, $T^*\Dens(M)$ can be regarded as a symplectic leaf in the Poisson manifold $\Dens(M)\times \Xcal^*(M)$ via the mapping
\begin{equation*}
	(\rho,\theta)\mapsto (\rho, \rho\, \ud\theta\otimes\vol).
\end{equation*}
\end{lemma} 
\begin{proof} 
%
Since the space $T_\varrho \Dens(M) = C^\infty_0(M)$ of zero-mean functions is a subspace of $C^\infty(M)$ 
it follows that 
\begin{equation} 
\begin{split}
T^*\Dens(M) &= \Dens(M)\times C^\infty(M)^*/\ker(\pair{\,\cdot\,,C^\infty_0(M)}) \\ 
&= \Dens(M)\times C^\infty(M)/\RR. 
\end{split} 
\end{equation} 
That $T^*\Dens(M)$ is a symplectic leaf in $\Dens(M)\times \Xcal^*(M)$ now follows from \autoref{prop:symplectic_leaf} since the mapping $(\rho,\theta)\mapsto (\rho,\rho\,\ud\theta\otimes\vol)$ is bijective and Poisson.
(Even more geometrically, one can regard $T^*\Dens(M)$ as the zero-momentum symplectic reduction leaf as outlined in \autoref{sub:symplectic_reduction}.)
\end{proof}

Next, we turn to Newton's equations on $\Dens(M)$ for the Wasserstein-Otto metric \eqref{eq:otto_metric}. 
%
\begin{corollary} \label{cor:newton_on_dens_hamiltonian_form} 
The Hamiltonian on $T^*\Dens(M)$ corresponding to Newton's equations on $\Dens(M)$ 
with respect to the Wasserstein-Otto metric \eqref{eq:otto_metric} is 
\begin{equation} 
{\widetilde{H}}(\rho,\theta) 
= 
\frac{1}{2}\int_M \rho \abs{\nabla\theta}^2\,\vol + \bar U(\rho) 
\end{equation} 
%
and the Hamiltonian equations are 
\begin{align} \label{eq:ham_eq} 
\begin{cases} 
&\dot\theta + \frac{1}{2}\abs{\nabla\theta}^2 + \frac{\delta \bar U}{\delta\rho} = 0,\\ 
&\dot\rho + \divv(\rho \nabla\theta) = 0. 
\end{cases} 
\end{align} 
Solutions of \eqref{eq:ham_eq} correspond to horizontal solutions of Newton's equations 
\eqref{eq:newton_for_ubar} on $\Diff(M)$, or, equivalently, 
to zero-momentum solutions of the reduced equations \eqref{eq:poisson_system_reduced} 
with Hamiltonian 
\begin{equation} 
\bar{H}(\rho,m) = \frac{1}{2}\pair{m,v} + \bar U(\rho), 
\qquad 
m=\rho v^\flat\otimes\vol. 
\end{equation} 
\end{corollary} 
\begin{proof} 
Given the Hamiltonian \eqref{eq:ham_Ubar} on $T^*\Diff(M)$, the result follows 
directly from \autoref{thm:poisson_reduction} and the zero-momentum reduction result 
\autoref{thm:symplectic_reduction} in \autoref{sub:symplectic_reduction}. 
\end{proof} 
%

\section{Wasserstein-Otto examples} \label{sect:WOexamples}

In this section we provide and study examples of Newton's equations on $\Diff(M)$ 
with respect to the $L^2$ metric \eqref{eq:L2met} and $\Diffvol(M)$-invariant potentials. 
We also derive the corresponding Poisson reduced equations on $\Dens(M)\times \Xcal^*(M)$ 
(cf.~\autoref{sub:poisson_reduction}) and symplectic reduced equations on $T^*\Dens(M)$ 
corresponding to Newton's equations for the Wasserstein-Otto metric \eqref{eq:otto_metric} 
(cf. \autoref{sub:symplectic?} above). 

\subsection{Inviscid Burgers' equation} \label{sub:burgers} 
We start with the simplest case when the potential function is zero. 
The corresponding Newton's equations are the geodesic equations on $\Diff(M)$. 
\begin{proposition} 
Newton's equations with respect to the $L^2$-metric \eqref{eq:L2met} and with zero potential $\bar U =0$ 
admit the following formulations: 
\begin{itemize} 
\item 
the $L^2$ geodesic equations on $\Diff(M)$ 
\begin{equation} 
\nabla_{\dot\varphi}\dot\varphi = 0, 
\end{equation} 
\item 
the inviscid Burgers equations on $\Xcal(M)$ 
\begin{equation} 
\dot v + \nabla_v v = 0 
\end{equation} 
where $v = \dot{\varphi}\circ\varphi^{-1}$, 
\item 
the Poisson reduced equations on $\Dens(M)\times\Xcal^*(M)$ 
\begin{align} 
\begin{cases} 
\dot m +\LieD_v m - \rho\,\ud \big(  \frac{1}{2}  |v|^2 \big) \otimes \vol = 0 
\\ 
\dot\rho + \divv(\rho v) =0 
\end{cases} 
\end{align} 
where $m = \rho\, v^\flat \otimes \vol$, 
\item 
the symplectically reduced equations on $T^*\Dens(M)$ 
\begin{align} 
\begin{cases} 
\dot\theta +\frac{1}{2}\abs{\nabla \theta}^2 = 0 \label{eq:theta_otto_geodesics} 
\\ 
\dot\rho + \divv(\rho\nabla \theta) = 0 
\end{cases} 
\end{align} 
corresponding to the Hamiltonian form of the geodesic equations 
for the Wasserstein-Otto metric \eqref{eq:otto_metric}. 
\end{itemize} 
\end{proposition} 
Observe that the system in \eqref{eq:theta_otto_geodesics} consists of the Hamilton-Jacobi equation 
for the kinetic energy Hamiltonian $H(x,p) = \frac{1}{2}\met_x (p^\sharp, p^\sharp)$ on $M$ 
together with the transport equation for $\rho$. 
\begin{proof} 
The results follow directly from \autoref{thm:newton_for_Ubar}, \autoref{thm:poisson_reduction} 
and \autoref{cor:newton_on_dens_hamiltonian_form} after setting $\bar{U}=0$. 
\end{proof} 
%




\subsection{Classical mechanics and Hamilton-Jacobi equations} \label{sub:hamilton_jacobi} 

Let $V$ be a smooth potential function on $M$ and consider the corresponding potential function 
on the space of densities 
\begin{equation} \label{eq:classical_mechanics_potential} 
\bar U(\rho) = \int_M V \rho \, \vol
\end{equation} 
where $\rho \in \Dens(M)$. 
\begin{proposition} 
Newton's equations with respect to the $L^2$-metric \eqref{eq:L2met} 
and the potential $\bar U$ in \eqref{eq:classical_mechanics_potential} 
admit the following formulations: 
\begin{itemize} 
\item 
the $L^2$ geodesic equations with potential on $\Diff(M)$ 
\begin{equation} 
\nabla_{\dot\varphi}\dot\varphi + \nabla V\circ\varphi = 0, 
\end{equation} 
\item 
the inviscid Burgers equations with potential on $\Xcal(M)$ 
\begin{equation} 
\dot v + \nabla_v v + \nabla V = 0 
\end{equation} 
where $v = \dot{\varphi} \circ\varphi^{-1}$, 
\item 
the Poisson reduced equations on $\Dens(M)\times\Xcal^*(M)$ 
\begin{align} 
\begin{cases} 
\dot m +\LieD_v m - \rho \,\ud \big( \frac{1}{2}|v|^2 - V \big) \otimes\vol = 0 
\\ 
\dot\rho + \divv(\rho v) =0 
\end{cases} 
\end{align} 
where $m=\rho v^\flat \otimes \vol$,
\item 
the symplectically reduced equations on $T^*\Dens(M)$ 
\begin{align} 
\begin{cases} 
\dot\theta + \frac{1}{2}\abs{\nabla \theta}^2 + V = 0 \label{eq:theta_classical_mechanics} 
\\ 
\dot\rho + \divv(\rho\nabla \theta) = 0. 
\end{cases} 
\end{align} 
\end{itemize} 
\end{proposition} 
Observe that the system \eqref{eq:theta_classical_mechanics} consists of the Hamilton-Jacobi equation 
for the classical Hamiltonian $H(x,p) = \frac{1}{2}\met_x(p^\sharp,p^\sharp) + V(x)$ 
together with the transport equation for $\rho$. 
\begin{proof} 
Since 
$\frac{\delta}{\delta \varrho} \bar{U} = V$ 
the proposition follows by combining \autoref{thm:newton_for_Ubar}, \autoref{thm:poisson_reduction} 
and \autoref{cor:newton_on_dens_hamiltonian_form}. 
\end{proof} 
%
%

\subsection{Barotropic fluid equations} 
\label{sub:compressible_euler_equations} 

The motion of {\it barotropic fluids} is characterized by a functional relation between the pressure and the fluid's density. 
The corresponding equations on a Riemannian manifold expressed in terms of the velocity field $v$ 
and the density function $\rho$ have the form 
\begin{align} \label{eq:contin} 
\begin{cases} 
\dot v + \nabla_v v + \rho^{-1} \nabla P(\rho) = 0  
\\ 
\dot\rho + \divv(\rho v) = 0. 
\end{cases} 
\end{align} 
The function $P\in C^\infty(\RR)$ relates $\rho$ and the pressure function $p = P(\rho)$.  
This relation depends on the properties of the fluid and is called the barotropic equation of state. 
Note that the equations of {\it barotropic gas dynamics} are usually specified by a particular choice 
$P(\rho)=const\cdot\rho^a$ 
(where, e.g., $a = 7/5$ corresponds to the standard approximation for atmospheric air.) 

To connect these objects with our framework 
we let $e\colon\RR_+\to \RR_+$ be a function describing the internal energy $e(\rho)$ 
of a barotropic fluid per unit mass 
and consider a general potential 
\begin{equation} \label{eq:barotropic_func} 
\bar U(\rho) = \int_M \Phi(\rho)\,\mu 
\end{equation} 
where $\Phi(\rho)= e(\rho)\rho$.
The relation between pressure and the internal energy is given by 
\begin{equation} 
P(\rho) = e'(\rho)\rho^2. 
\end{equation} 
We also define the \emph{thermodynamical work function} as 
\begin{equation} \label{eq:thermodyn_work} 
W(\rho) = \frac{\partial \Phi}{\partial \rho}=e'(\rho)\rho + e(\rho) = \rho^{-1}P(\rho) + e(\rho). 
\end{equation} 
We have $\rho^{-1}\nabla P(\rho) = \nabla W(\rho)$ which helps explain the idea of 
introducing the work function $W$ in that the force in \eqref{eq:contin} becomes a pure gradient 
(here $\nabla W(\rho)$ is understood as the gradient $ \nabla( W\circ \rho)$ of a function on $M$). 
This can be arranged if the internal energy $e$ depends functionally on $\rho$. 
As we have seen in the general form of \eqref{eq:reduced_newton_diff} 
the internal work function $W$ is more fundamental than the pressure function $P$ 
in the following sense: 
when the internal energy depends on the derivatives of $\rho$ it may not be possible to find 
the pressure as a differential operator on $\rho$ unlike the the work function.

%
\begin{proposition} 
Newton's equations for the $L^2$-metric \eqref{eq:L2met} and the potential \eqref{eq:barotropic_func} 
admit the following formulations: 
\begin{itemize} 
\item 
on $\Diff(M)$ 
\begin{equation} 
\nabla_{\dot\varphi}\dot\varphi +  \Big( \rho^{-1} \nabla P(\rho) \Big) \circ\varphi = 0 
\end{equation} 
%
%
\item 
the barotropic compressible fluid equations \eqref{eq:contin} on $\Xcal(M)$ 
for the velocity field $v = \dot\varphi\circ\varphi^{-1}$ and the density function $\rho$, 
%
\item 
the Poisson reduced equations on $\Dens(M)\times\Xcal^*(M)$ 
\begin{align} 
\begin{cases} 
\dot m +\LieD_v m - \rho\,\ud \big( \frac{1}{2} |v|^2 - W(\rho) \big) \otimes\vol = 0 
\\ 
\dot\rho + \divv(\rho v) =0 
\end{cases} 
\end{align} 
where $m=\rho v^\flat\otimes\vol$ and $W(\rho)$ is the work function \eqref{eq:thermodyn_work}, 
\item 
the symplectically reduced form of the barotropic compressible fluid equations on $T^*\Dens(M)$ 
\begin{align} 
\begin{cases} 
\dot\theta + \frac{1}{2}\abs{\nabla \theta}^2 + W(\rho) = 0 \label{eq:theta_barotropic} 
\\ 
\dot\rho + \divv(\rho \nabla\theta) =0 . 
\end{cases} 
\end{align} 
%
\end{itemize} 
\end{proposition} 
\begin{proof} 
The energy function of a compressible barotropic fluid with velocity $v$ and density $\rho$ is
\begin{equation} \label{eq:energy_compressible} 
E = \frac{1}{2} \int_M \abs{v}^2 \rho\, \vol + \int_M e(\rho) \rho \,\vol, 
\end{equation} 
where the first term corresponds to fluid's kinetic energy and the second is the potential energy 
under the barotropic assumption. 
Introducing the momentum variable $m\in \Xcal^*(M)$ we obtain a Hamiltonian 
on $\Dens(M)\times\Xcal^*(M)$ of the form 
\begin{equation} \label{eq:compressibleHamiltonian} 
\bar H(\rho,m) = \tfrac{1}{2}\pair{m,v} + \pair{\rho,e(\rho)} 
\qquad 
m = \rho v^\flat\otimes\vol. 
\end{equation} 
It is clear that 
$\frac{\delta}{\delta m} \bar{H} = v$. 
Furthermore, we have 
\begin{equation}
\Big\langle \frac{\delta H}{\delta\rho}, \dot\rho \Big\rangle 
 = 
 \big\langle {-} \tfrac{1}{2}\abs{v}^2, \dot\rho \big\rangle 
 + 
 \langle \Phi'(\rho), \dot\rho \rangle 
= 
\big\langle {-} \tfrac{1}{2}\abs{v}^2+\Phi'(\rho), \dot\rho \big\rangle. 
\end{equation} 
Substituting into \eqref{eq:poisson_system_reduced} we arrive at the system 
\begin{align} \label{eq:momentum_compressible_euler1} 
\begin{cases} 
\dot{m} = -\LieD_{v}m - \rho \,\ud \big( {-}\frac{1}{2} |v|^2 + \Phi'(\rho) \big) \otimes\vol  
\\ 
\dot\rho + \divv(\rho v) =0 . 
\end{cases} 
\end{align} 

To obtain the compressible Euler equations we rewrite the second term in the first equation 
using \eqref{eq:thermodyn_work} as 
\begin{equation} 
\begin{split} 
\rho\,\ud \Phi'(\rho) \otimes \vol 
&= \Big(  \ud \big( \rho \Phi'(\rho) \big) - \Phi'(\rho) \ud \rho \Big) \otimes\vol 
\\ 
&=  \ud \big(\rho \Phi'(\rho)-\Phi(\rho) \big) \otimes\vol 
\\ 
&= \ud P(\rho)\otimes\vol 
\end{split} 
\end{equation} 
to get 
\begin{align} \label{eq:momentum_compressible_euler2} 
\begin{cases} 
\dot{m} = -\LieD_{v}m + \frac{1}{2}\ud\interior_{v}v^{\flat}\otimes\varrho - \ud P\otimes\vol 
\\ 
\dot\rho + \divv(\rho v) =0 . 
\end{cases} 
\end{align} 
Differentiating $m = \rho v^{\flat}\otimes\vol$ in the time variable 
\begin{equation} \label{eq:dotbaralpha_calc} 
\dot{m} 
= 
( \rho \dot{v}^\flat + \dot\rho v^\flat) \otimes \vol 
= 
\dot{v}^\flat \otimes \varrho - v^\flat \otimes \LieD_v \varrho 
\end{equation} 
and substituting into \eqref{eq:momentum_compressible_euler2} we obtain 
\begin{align} \label{eq:momentum_compressible_euler3} 
\begin{cases} 
\dot{v}^{\flat} \otimes\varrho 
= 
\big( 
{-} \LieD_{v}v^{\flat} + \frac{1}{2}\ud\interior_{v}v^{\flat} - \rho^{-1} \ud P \big) \otimes \varrho 
\\ 
\dot{\varrho} = -\LieD_{v}\varrho. 
	\end{cases}
\end{align} 
Using the identities 
$\LieD_v\varrho = \divv(\rho v)\vol$ 
and 
$(\nabla_v v)^\flat = \LieD_vv^{\flat} - \frac{1}{2}\ud i_{v}v^{\flat}$ 
we now recover the compressible Euler equations \eqref{eq:contin}, 
that is 
\begin{align} 
\begin{cases} 
\dot v + \nabla_v v = -\rho^{-1} \nabla P(\rho) 
\\ 
\dot \rho + \divv(\rho v) = 0. 
\end{cases} 
\end{align} 

To describe these equations as Newton's equations \eqref{eq:ham_form_of_newton_otto} on $\Diff(M)$ 
we consider the potential $\bar U(\rho) = \int_M \Phi(\rho)\vol $ given in \eqref{eq:barotropic_func}
with  $\frac{\delta}{\delta \rho} \bar{U} (\rho) = \Phi'(\rho) = e'(\rho)\rho + e(\rho)$. 
Note that this potential $\bar U\colon\Dens(M)\to\RR$ is of the form \eqref{eq:Ubar}.
From \autoref{thm:newton_for_Ubar} we find Newton's equations 
corresponding to the compressible Euler equations 
\begin{equation} \label{eq:Newton_form_compressible_Euler} 
\nabla_{\dot\varphi}\dot\varphi =- \nabla \big( \Phi'\big(\rho\big) \big) \circ\varphi. 
\end{equation} 
From \autoref{cor:newton_on_dens_hamiltonian_form} we get the symplectically reduced form on $T^*\Dens(M)$. 
\end{proof} 
%


\newcommand{\subG}{\mathcal{N}} 

\section{Semidirect product reduction} 
\label{sect:general-semi} 

In this section we recall one standard approach to the equations of compressible fluid dynamics 
using semidirect products, see \cite{ViDo1978, MaRaWe1984b}. 
Recall from the earlier sections that the barotropic Euler equations can be viewed as a mechanical system 
on the configuration space $\Diff(M)$ with the symmetry group $\Diffvol(M)$. 
On the other hand, such a system can be also obtained by a semidirect product construction 
as a so-called Lie-Poisson system provided that 
the configuration space is extended so that it coincides with the given symmetry group. 
We unify these approaches in \autoref{sec:semi_direct_reduction}: 
any Lie-Poisson system on a semidirect product can be viewed as a Newton system with a smaller symmetry group.
We begin with two standard examples.

\subsection{Barotropic fluids via semidirect products} 
\label{sub:compressible_semidirect} 
 
In order to describe a barotropic fluid \eqref{eq:contin} as a Lie-Poisson system one can introduce 
the semidirect product group $S=\Diff(M)\ltimes C^\infty(M)$ 
as a space of pairs $(\varphi,f)$ equipped with the group structure 
\begin{equation} \label{eq:def.5.1} 
(\varphi,f)\cdot(\psi,g)=(\varphi\circ\psi,\varphi_*g+f), 
\quad 
\varphi_\ast g = g \circ \varphi^{-1} 
\end{equation} 
which is smooth in the Fr\'echet topology, cf. \autoref{sect:tame}.
 
The Lie algebra $\mathfrak s=\Xcal(M)\ltimes C^\infty(M)$ is also a semidirect product 
with a commutator given by 
\begin{equation} \label{eq:diffad} 
 \ad_{(v,b)}(u,a) = ( -\LieD_v u, \LieD_u b-\LieD_va ). 
\end{equation} 

The corresponding (smooth) dual space is $\mathfrak s^*=\Xcal^*(M)\times C^\infty(M)$ 
whose elements are pairs $(m, \rho)$ 
with $m=\alpha \otimes \mu \in \Xcal^*(M)$ and $\rho \in C^\infty(M)$, 
where $\mu$ is a fixed volume form and $\alpha$ is a 1-form on $M$. 
The pairing between $\mathfrak{s}$ and $\mathfrak{s}^\ast$ is given by 
$$ 
\langle (v, b), (m, \rho) \rangle = \int_M (\iota_v \alpha)\,\vol + \int_M b \rho\,\vol. 
$$ 

The Lie algebra structure of $\mathfrak s$ determines the Lie-Poisson structure on $\mathfrak s^*$ 
and the corresponding Poisson bracket at $(m, \rho)\in \mathfrak{s}^*$ 
is given by the formula \eqref{eq:poisson_bracket_reduced}. 
It is sometimes called the \textit{compressible fluid bracket}. (We refer to \autoref{sec:semi_direct_reduction} 
for a general setting of semidirect products and explicit formulas.) 
Notice that $\mathfrak{s}^*$ is strictly bigger than $\Dens(M)\times\Xcal^*(M)$, since $\rho$ now can be any function (it does not have to be a probability density).

 
In order to define a dynamical system on $S$ consider a smooth function $P$ (relating pressure to fluid's density $\rho$, 
as in \autoref{sub:compressible_euler_equations}) of the form $P(\rho) = \rho^2 \Phi'(\rho)$ and
define the following energy function on $\mathfrak{s}$ 
$$ 
E(v,\rho) 
= 
\int_M \bigg( \frac12\abs{v}^2\,\rho + \rho\,\Phi(\rho) \bigg) \mu. 
$$ 
Lifting $E$ to the dual $\mathfrak s^*$ with the help of the inertia operator of the Riemannian metric 
we obtain the following Hamiltonian on $\mathfrak{s}^\ast$ 
\begin{equation} \label{eq:ham_barotropic_semidirect} 
H(m,\rho) 
= 
\int_M \bigg( \frac{1}{2\rho}\abs{m}^2 + \rho\,\Phi(\rho) \bigg) \vol. 
\end{equation} 
Observe that, by construction, the associated Hamiltonian system on the cotangent bundle $T^*S$ 
is right-invariant with respect to the action of $S$. 
\begin{theorem}[\cite{ViDo1978, MaRaWe1984b}]
The barotropic fluid equations \eqref{eq:contin} correspond to the Lie-Poisson system on $\mathfrak s^*$ 
with the Poisson bracket of type \eqref{eq:poisson_bracket_reduced} and 
the Hamiltonian \eqref{eq:ham_barotropic_semidirect}. 
\end{theorem} 

While the general barotropic equations described above are valid for any smooth initial velocity field, 
one is often interested only in \textit{potential} solutions of the system. 
These are obtained from initial conditions of the form 
$v_0=\nabla \theta_0$ where $\theta_0$ is a smooth function on $M$. 
As we have already seen, such solutions retain their gradient form for all times and the equations 
can be viewed as the Hamilton-Jacobi equations, see \eqref{eq:theta_shallow_water}. 
Potential solutions of this type arise naturally in the context of the Madelung transform, 
see \autoref{sec:madelung} below. 
\begin{remark} 
The semidirect product framework is a natural setting whenever the physical model contains 
a quantity transported by the flow, e.g., the continuity equation \eqref{eq:contin}.
However, while the Hamiltonian point of view works similarly to the case of incompressible fluids, 
the Lagrangian approach with semi-direct product Lie algebras encounters drawbacks, 
cf.\ \cite{HoMaRa1998}. 
These are mostly related to the fact that the Lagrangian is not quadratic 
and cannot be directly interpreted as a kinetic energy yielding geodesics on the group 
(for some attempts to bypass this problem using the Maupertuis principle see \citet{Sm1979}; 
for a geodesic formulation in an extended phase space see \citet{Pr2013}).  
Furthermore, there is no physical interpretation of the action of the full semidirect product on its dual space: 
the particle reparametrization symmetry is related only to  
the action of the first (diffeomorphisms) but not of the second (functions) factor 
in the product $S=\Diff(M) \ltimes C^\infty(M)$. 
One advantage of our point of view in \autoref{sub:compressible_euler_equations} using Newton's equations is that it resolves such issues.
\end{remark} 
%

\subsection{Incompressible magnetohydrodynamics} 
\label{sub:mhd} 

An approach based on semidirect products is also possible in the case of the equations of self-consistent  
magnetohydrodynamics (MHD). 
We start with the incompressible case and discuss the compressible case in detail in \autoref{sect:comprMHD}.
The underlying system describes an ideal fluid whose divergence-free velocity $v$ 
is governed by the Euler equations (see \autoref{sect:incompressible} for Lagrangian and Hamiltonian formulations). 
Assume next that the fluid has infinite conductivity and carries a (divergence-free) magnetic field $\mathbf{B}$. 
Transported by the  flow (i.e., frozen in the fluid) $\mathbf{B}$ acts reciprocally (via the Lorenz force) 
on the velocity field and the resulting MHD system on a three-dimensional Riemannian manifold $M$ 
takes the form 
\begin{align} \label{eq:mhd} 
\begin{cases} 
&\dot v + \nabla_v v + \mathbf{B}\times \curl\mathbf{B} +\nabla P =0 
\\ 
& \dot{\mathbf{B}} + \LieD_v \mathbf{B} = 0 
\\ 
& {\rm div}\, v=0 
\\
& {\rm div}\,{\mathbf{B}}=0\,.
\end{cases} 
\end{align} 
A natural configuration space for the system \eqref{eq:mhd} is the semidirect product of 
the group of volume-preserving diffeomorphisms 
and the dual of the space $\mathfrak X_\vol(M)$ of divergence-free vector fields 
on a three-fold $M$. 
The corresponding Lie algebra is the semidirect product of $\mathfrak{X}_\mu(M)$ and its dual. 
The group product and the algebra commutator are given by the formulas 
\eqref{eq:def.5.1} and \eqref{eq:diffad}, respectively. 

More generally, the configuration space of incompressible magnetohydrodynamics 
on a manifold $M$ of arbitrary dimension $n$ 
is the semidirect product group 
$\IMH = \Diffvol(M) \ltimes \Omega^{n-2}(M)/\ud\Omega^{n-3}(M)$ 
(which for $n=3$ reduces to $\Diff_\mu(M) \ltimes \mathfrak{X}_\mu^\ast(M)$) 
with its Lie algebra 
$\mathfrak{imh}=\mathfrak X_\vol(M)\ltimes \Omega^{n-2}(M)/\ud\Omega^{n-3}(M)$. 
Since the dual of $\Omega^{n-2}(M)/\ud\Omega^{n-3}(M)$ is the space $\Omega_{cl}^2(M)$ 
of closed 2-forms on $M$ we have 
$\mathfrak{imh}^* = \mathfrak X_\vol^*(M) \oplus \Omega_{cl}^2(M)$. 
Magnetic fields in $M$ can be viewed as 
either closed 2-forms $\beta\in \Omega_{cl}^2(M)$ 
or $(n-2)$ fields $\mathbf B$ that are related to $\beta$ by $\beta=\iota_{\mathbf B}\vol$. 
This latter point of view will be useful also for the description of compressible magnetohydrodynamics. 
 
The corresponding Poisson bracket on $\mathfrak{imh}^*$ is given by the formula 
\eqref{eq:poisson_bracket_reduced} interpreted accordingly. 
%


Finally, as the Hamiltonian function we take the sum of the kinetic and magnetic energies 
of the fluid, i.e. 
$$ 
E(v, \mathbf{B}) 
= 
\frac 12 \int_M  \left( \abs{v}^2  + \abs{\mathbf{B}}^2 \right) \mu 
$$ 
(here the Riemannian metric defines the inertia operator and hence the $L^2$ quadratic form 
on all spaces $\mathfrak X_\vol(M)$, $\mathfrak X_\vol^*(M)$ and $\Omega_{cl}^2(M)$, 
see, e.g., \cite{ArKh1998}). 
The Hamiltonian on $\mathfrak{imh}^*$ is 
\begin{equation} \label{eq:mhd_hamiltonian} 
H(m, \mathbf B) 
= 
\frac 12 \int_M  \left( \abs{m}^2  + \abs{\mathbf{B}}^2 \right) \mu.
\end{equation}
\begin{theorem}[\cite{ViDo1978, ArKh1998}]
The incompressible MHD equations \eqref{eq:mhd} correspond to the Lie-Poisson system 
on $\mathfrak{imh}^*$ for the Hamiltonian \eqref{eq:mhd_hamiltonian}. 
\end{theorem} 
An analogue of this equation for compressible fluids in an $n$-dimensional manifold 
will be discussed in Section \ref{sect:comprMHD}.

\section{More general Lagrangians} 

\subsection{Fully compressible fluids} 
\label{sect:fully} 

For general compressible (non-barotropic) inviscid fluids the equation of state 
includes pressure $P=P(\rho, \sigma)$ as a function of both density $\rho$ and specific entropy $\sigma$ 
(defined as a smooth function on $M$ representing entropy per unit mass, cf.\ \citet[Sect.~3.2]{Do2013}). 
Thus, the equations of motion describe the evolution of three quantities: 
the velocity of the fluid $v$, its density $\rho$ and the specific entropy $\sigma$, namely 
\begin{equation} \label{eq:full_compressible_euler}
\left\{ 
 \begin{array}{l} 
			\dot v + \nabla_v v + \rho^{-1} \nabla P(\rho, \sigma) = 0 
			\\ 
			\dot\rho + \divv(\rho v) = 0 
			\\ 
                         \dot \sigma +\divv(\sigma v)=0. 
                         \\ 
\end{array} \right. 
\end{equation} 

The purpose of this section is to show that under natural assumptions this system also describes 
Newton's equations on $\Diff(M)$ but with potential function of more general form than \eqref{eq:Ubar}.
In a nutshell, a proper phase space for this equation is the reduction of $T^*\Diff(M)$ 
over a subgroup $\mathcal{N}$ which is {\it smaller than} $\Diffvol(M)$.  
In view of the results in \autoref{sec:semi_direct_reduction} the full compressible Euler equations 
are a semidirect product representation of a Newton system on $\Diff(M)$ whose symmetry group 
is a proper subgroup of $\Diffvol(M)$. 
\begin{theorem} \label{thm:full_compressible_euler} 
The fully compressible system \eqref{eq:full_compressible_euler} 
is obtained using an embedding into the Lie-Poisson space $\mathfrak{s}_{(2)}^*$ 
where $\mathfrak{s}_{(2)} = \Xcal(M)\ltimes C^\infty(M,\RR^2)$ 
(cf.\ \autoref{prop:semidirect_embedding}) 
from Newton's equations on $\Diff(M)$ with Lagrangian 
\begin{equation} \label{eq:L_full_compressible} 
L(\varphi,\dot\varphi) 
= 
\frac{1}{2}\int_M \abs{\dot\varphi}^2\vol - \bar U(\rho,\sigma) 
\end{equation} 
where $\bar U\colon \Dens(M)\times C^\infty(M)\to \RR$ 
is a potential function (of density $\rho = \Jac{\varphi^{-1}}$ and entropy density $\sigma = \varphi_*\varsigma_0/\vol$ for some fixed initial entropy density $\varsigma_0$) 
of the form 
\begin{equation} 
U(\rho,\sigma) = \int_M e(\rho,\sigma)\rho\,\vol 
\end{equation} 
and where the internal energy $e$ and pressure $P$ are related by 
\begin{equation} 
P(\rho,\sigma) 
= 
\rho^2 \frac{\partial e}{\partial\rho}(\rho,\sigma) + \sigma\rho \frac{\partial e}{\partial\sigma}(\rho,\sigma). 
\end{equation} 
\end{theorem} 
From the point of view of symplectic reduction in \autoref{sec:semi_direct_reduction} the symmetry subgroup $\mathcal{N}$ 
is given by $\mathcal{N}:=\Diffvol(M)\cap \Diff_{\varsigma_0}(M)$. Our aim is to embed 
$T^*\Diff(M)/\mathcal{N}$ in $\mathfrak{s}_{(2)}^* = \Xcal^*(M)\times (\Omega^n(M))^2$. (Notice that while the quotient $T^*\Diff(M)/\mathcal{N}$ might not be manifold, it can be viewed as an invariant set 
consisting of coadjoint orbits in the dual  space of an appropriate semidirect product Lie algebra, as discussed in  \autoref{sec:semi_direct_reduction}.)
To achieve this embedding we need to compute the momentum map for the cotangent lifted action of 
$\Diff(M)$ on $T^*(C^\infty(M))^2$.
\begin{lemma} \label{lem:momentum_full_euler} 
The momentum map for the cotangent action of $\Diff(M)$ on 
$T^*C^\infty(M)\times T^*C^\infty(M) = C^\infty(M)\times C^\infty(M)\times C^\infty(M)\times C^\infty(M)$ 
is 
\begin{equation} 
J(\rho,\theta,\sigma,\kappa) = \rho\, \ud \theta\otimes\vol + \sigma\,\ud\kappa \otimes\vol. 
\end{equation} 
\end{lemma} 
\begin{proof} 
From \autoref{sect:hamiltons} we already know the momentum map for the action on $T^*\Dens(M)$. 
This is the same as the action on $T^*C^\infty(M)$.
For diagonal actions we then just get a sum as stated in the lemma.
\end{proof} 
\begin{proof}[Proof of \autoref{thm:full_compressible_euler}] 
Any Hamiltonian system on the Poisson space $\mathfrak s_{(2)}^*$ has the form 
\begin{align*} 
&\dot m + \LieD_v m + J\Big( \rho, \frac{\delta H}{\delta \rho},\sigma, \frac{\delta H}{\delta\sigma} \Big) = 0, 
\quad 
\dot\rho + \divv(\rho v) = 0, 
\quad 
\dot\sigma + \divv(\sigma v) = 0 
\end{align*} 
where $v = \frac{\delta H}{\delta m}$. 
The Hamiltonian corresponding to the Lagrangian \eqref{eq:L_full_compressible} 
is the same as in \eqref{eq:compressibleHamiltonian} except that the potential energy $\bar U$ 
depends now also on $\varsigma$. 
By \autoref{lem:momentum_full_euler} the first equation then becomes 
\begin{equation} \label{eq:full_euler_ham_form} 
\dot m + \LieD_v m + \rho\,\ud\Big(\frac{\delta \bar U}{\delta\rho} 
- 
\frac{1}{2}\abs{v}^2 \Big)\otimes\vol + \sigma\ud \Big(\frac{\delta \bar U}{\delta \sigma} \Big) \otimes\vol
= 0. 
\end{equation} 
The variational derivatives are given by 
\begin{equation} 
\frac{\delta\bar U}{\delta\rho} 
= 
e(\rho,\sigma) + \rho\frac{\partial e}{\partial\rho}(\rho,\sigma) 
\quad\text{and}\quad 
\frac{\delta\bar U}{\delta\sigma} = \rho\frac{\partial e}{\partial\sigma}(\rho,\sigma). 
\end{equation} 
Using 
\begin{align*} 
\ud P(\rho,\sigma) 
&= 
\ud\Big(\rho^2 \frac{\partial e}{\partial\rho}(\rho,\sigma) + \sigma\rho \frac{\partial e}{\partial\sigma}(\rho,\sigma) \Big) 
\\ 
&= 
\rho \ud \Big( e(\rho,\sigma) + \rho\frac{\partial e}{\partial\rho}(\rho,\sigma)\Big) 
+ 
\sigma \ud \Big( \rho\frac{\partial e}{\partial\sigma}(\rho,\sigma)\Big) 
\\ 
&= 
\rho \ud \frac{\delta\bar U}{\delta\rho} + \sigma \ud \frac{\delta\bar U}{\delta\sigma} 
\end{align*} 
we then recover from \eqref{eq:full_euler_ham_form} 
the fully compressible Euler equations \eqref{eq:full_compressible_euler}. 
\end{proof} 
%


%
Observe that an invariant subset of solutions is given by those solutions with momenta 
$m=\rho\,\ud\theta\otimes\vol+\sigma\,\ud\kappa\otimes\vol$, where $\theta,\kappa \in C^\infty(M)$. 
They can be regarded as analogues of potential solutions of the barotropic fluid equations. 
We thereby obtain a canonical set of equations on $T^*(C^\infty(M))^2$ 
\begin{equation*} 
\left\{ 
\begin{array}{lcl} 
	\dot\rho = \dfrac{\delta \bar H}{\delta\theta} & 	\dot\sigma = \dfrac{\delta \bar H}{\delta\kappa} \\ 
	\dot\theta = -\dfrac{\delta \bar H}{\delta\rho} &\quad	\dot\kappa = -\dfrac{\delta \bar H}{\delta\sigma} 
\end{array} \right. 
\end{equation*} 
%
with the restricted Hamiltonian 
\begin{equation*} 
\bar H(\rho,\sigma,\theta,\kappa) = H(\rho\otimes\ud\theta + \sigma\otimes\ud\kappa, \rho,\sigma).
\end{equation*} 
We point out that the group $S_{(2)}=\Diff(M)\ltimes C^\infty(M,\RR^2)$ corresponding to $\mathfrak s_{(2)}$
is associated with a multicomponent version of the Madelung transform relating compressible fluids 
and the NLS-type equations, 
cf. the details in \autoref{sec:madelung} and see also \cite{KhMiMo2019}. 
Applying the multicomponent Madelung transform ${\bf M}^{(2)}$ one can also rewrite the fully compressible system 
on the space of rank-1 spinors $H^s(M,\CC^2)$. 

\begin{remark} 
Solutions of barotropic fluid equations are contained in the solution space of 
the fully compressible Euler equations as ``horizontal-within-horizontal" solutions in the following sense. 
Let the initial entropy function have the form $\sigma = s(\rho)$ for some function $s\in C^\infty(\RR_+,\RR)$. 
Then 
\[ 
\dot\sigma = s'(\rho)\dot\rho = -s'(\rho)\divv(\rho u), 
\] 
where the last equality follows from the evolution equation for $\rho$. 
From the equation for $\sigma$ we obtain 
\[
	\dot\sigma = -\divv(\sigma u) = -s'(\rho)\divv(\rho u).
\]
Thus, the entropy  remains in the form $\sigma = s(\rho)$ so that we obtain a barotropic flow 
with the pressure function $P\big( \rho,s(\rho) \big)$.
From a geometric point of view these solutions correspond to
a special symplectic leaf in $\mathfrak{s}^* = \Xcal^*(M)\times C^\infty(M)\times C^\infty(M)$.
\end{remark} 
%

\subsection{Compressible magnetohydrodynamics} 
\label{sect:comprMHD} 

Next, we turn to a description of compressible inviscid magnetohydrodynamics. 
A compressible fluid of infinite conductivity carries a magnetic field acting reciprocally on the fluid. 
The corresponding equations on a Riemannian 3-manifold $M$ have the form 
\begin{align} 
\begin{cases} \label{eq:comp_mhd} 
&\dot v + \nabla_v v + \rho^{-1} \mathbf{B}\times \curl\mathbf{B} + \rho^{-1} \nabla P(\rho) =0 
\\ 
& \dot\rho + \divv(\rho v) = 0 
\\ 
& \dot{\mathbf{B}} + \curl\mathbf{E} = 0, \qquad \mathbf{E} = \mathbf{B}\times v\,,
\end{cases} 
\end{align} 
where $v$ is the velocity and $\rho$ is density of the fluid, while $\mathbf{B}$ is the magnetic vector field. 
Note that these equations reduce to the incompressible MHD equations \eqref{eq:mhd}
when density $\rho$ is a constant. 

As mentioned before, it is more natural to think of magnetic fields as closed 2-forms. 
This becomes apparent when the equations are generalized to a compressible setting 
or to other dimensions. 
(For instance, a non-volume-preserving diffeomorphism violates the divergence-free constraint 
of a magnetic vector field but preserves closedness of differential forms.) 
In fact, let $\Omega^2_{cl}(M)$ denote the space of smooth closed differential 2-forms on an $n$-manifold $M$. 
The diffeomorphism group acts on $\Omega^2_{cl}(M)$ by push-forward 
and the (smooth) dual of $\Omega^2_{cl}(M)$ is the quotient $\Omega^{n-2}(M)/\ud\Omega^{n-3}(M)$. 

The cotangent lift of the left action of $\Diff(M)$ to 
$T^*\Omega^2_{cl}(M) \simeq \Omega^2_{cl}(M)\times \Omega^{n-2}(M)/\ud\Omega^{n-3}(M)$ 
is given by 
\begin{equation} \label{eq:cotangentaction_magnetic} 
\varphi \cdot (\beta, [P]) = (\varphi_*\beta, \varphi_* [P]). 
\end{equation} 
Observe that this is well-defined since push-forward commutes with the exterior differential. 
\begin{lemma} 
The momentum map 
$I\colon T^*\Omega^2_{cl}(M) \to \Xcal^*(M)$ 
associated with the cotangent action in \eqref{eq:cotangentaction_magnetic} is given by 
\begin{equation} 
I(\beta,[P]) = \iota_u \beta\otimes\vol, 
\end{equation} 
where the vector field $u$ is uniquely defined by $\iota_u \vol = \ud P$. 
\end{lemma} 
As expected, the map $I$ is independent of the choice of $\mu$ and a representative $P$. 
In what follows it will be convenient to replace $\mu$ by $\varrho$ 
- resulting in a different vector field $u$ but without affecting the momentum map. 
\begin{proof} 
The infinitesimal action of a vector field $v$ on $\beta$ is $-\LieD_v\beta$. 
Since it is a cotangent lifted action, the momentum map is given by 
\begin{align} 
\pair{I(\beta,[P]),v} 
&= 
\pair{-\LieD_v\beta,[P]} 
= 
\int_M -\LieD_v\beta\wedge P 
\\ 
&= 
\int_M -\ud\iota_v\beta\wedge P 
= 
\int_M -\iota_v\beta \wedge \ud P. 
\end{align} 
Now, if $\iota_u\vol = \ud P$, then 
\begin{align} 
\int_M -\iota_v\beta \wedge \ud P 
= 
\int_M (\iota_v\iota_u\beta) \vol 
= 
\pair{v,\iota_u\beta\otimes\vol}. 
\end{align} 
\end{proof} 

Consider a Lagrangian on $T\Diff(M)$ given by the fluid's kinetic and potential energies 
with an additional term involving the action on the magnetic field $\beta_0\in \Omega^2_{cl}(M)$, 
namely 
\begin{equation} \label{eq:mhd_lag} 
L(\varphi,\dot\varphi) 
= 
\frac{1}{2}\int_M \rho\abs{v}^2\,\vol - \int_M e(\rho)\rho\,\vol - \frac{1}{2}\int_M \beta\wedge\star\beta, 
\end{equation} 
where $v=\dot\varphi\circ\varphi^{-1}$, $\rho = \Jac{\varphi^{-1}}$ and $\beta = \varphi_*\beta_0$. 
As in \autoref{lem:legendre_on_diff} the corresponding Hamiltonian is 
\begin{equation} \label{eq:mhd_ham} 
H(\varphi,m) 
= 
\frac{1}{2}\pair{m,v} + \int_M e(\rho)\rho\,\vol + \frac{1}{2}\int_M \beta\wedge\star\beta 
\end{equation} 
where $m = \rho v^\flat\otimes\vol$. 
Letting $\Diff_{\beta_0}(M)$ denote the isotropy subgroup for the action of $\Diff(M)$, 
the (right) symmetry group of the Hamiltonian \eqref{eq:mhd_ham} is 
$$ 
\mathcal G = \Diffvol(M)\cap \Diff_{\beta_0}(M). 
$$ 
The corresponding Lie algebra consists of vector fields such that 
\begin{equation} 
\divv v = 0 \quad \text{and} \quad \LieD_v \beta_0 = 0. 
\end{equation} 
If $M$ is even-dimensional and $\beta_0$ is non-degenerate then the pair $(M,\beta_0)$ is a symplectic manifold 
and the Lie algebra consists of symplectic vector fields that also preserve the first integral $\beta_0^n/\vol$. 
%

Next, we proceed to carry out Poisson reduction, i.e., to compute the reduced equations 
on $T^*\Diff(M)/\mathcal G \simeq \Diff(M)/\mathcal G\times \Xcal^*(M)$. 
In contrast to the case $\mathcal G = \Diffvol(M)$ studied in \autoref{sect:hamiltons} 
there is no simple way to identify $\Diff(M)/\mathcal G$ 
and so it will be convenient to use the semidirect product reduction framework 
developed in \autoref{sec:semi_direct_reduction} above. 
(Similarly to  \autoref{sect:fully}
the quotient $T^*\Diff(M)/\mathcal{G}$ might not be manifold, but it can be regarded as an invariant set formed by coadjoint orbits in the dual  space of an appropriate semidirect product Lie algebra, 
see \autoref{sec:semi_direct_reduction}. For now one
can regard these considerations as taking place at a  ``smooth point" of the quotient.)
To this end, consider the semidirect product algebra 
$\mathfrak{cmh} = \Xcal(M)\ltimes (C^\infty(M)\oplus \Omega^{n-2}(M)/\ud\Omega^{n-3}(M))$ 
and its dual 
$$ 
\mathfrak{cmh}^* = \Xcal^*(M)\times (\Omega^n(M)\oplus \Omega_{cl}^2(M)). 
$$ 
We have a natural embedding of $T^*\Diff(M)/\mathcal G$ in $\mathfrak{cmh}^*$ 
via the map 
$([\varphi],m) \mapsto (m,\varphi_*\vol,\varphi_*\beta_0)$ 
and 
the corresponding Hamiltonian on $\mathfrak{cmh}^*$ is 
\begin{equation} 
\bar H(\rho,\beta,m) 
= 
\frac{1}{2}\pair{m,v} + \int_M e(\rho)\rho\,\vol + \frac{1}{2}\int_M \beta\wedge\star\beta. 
\end{equation} 
\begin{theorem} 
The Poisson reduced form on 
$$ 
T^*\Diff(M)/\mathcal G 
\simeq 
\Diff(M)/\mathcal G \times \Xcal^*(M) \subset \mathfrak{cmh}^* 
$$ 
of the Euler-Lagrange equations for the Hamiltonian \eqref{eq:mhd_ham} is 
\begin{align} \label{eq:mhd_poisson_reduced} 
\begin{cases} 
\dot m + \LieD_v m + \rho\,\iota_u\beta\otimes\vol  + \rho\,\ud \left( \frac{\delta H}{\delta \varrho} \right)\otimes \vol = 0, 
\\ 
\dot\rho + \divv(\rho v) = 0, 
\\ 
\dot\beta + \LieD_v \beta = 0, 
\end{cases} 
\end{align} 
where the field $u$ is defined by $\iota_u\varrho = \ud \big( \frac{\delta H}{\delta \beta} \big)$ 
and the momentum variable is $m =  \rho v^\flat\otimes\vol$.
For a three-fold $M$ these equations correspond to the equations of the compressible inviscid magnetohydrodynamics 
\eqref{eq:comp_mhd} where the magnetic field $\mathbf{B}$ is related to the closed 2-form $\beta$ 
by $\iota_{\mathbf B}\vol = \beta$. 
\end{theorem} 
\begin{proof} 
In general, if $\Diff(M)$ acts on a space $S$ from the left with the momentum map $I\colon T^*S \to \Xcal^*(M)$ 
then the Poisson reduced system is 
\begin{align} 
\begin{cases} 
\dot m + \LieD_v m - I\left(s, \frac{\delta L}{\delta s}\right) = 0, 
\\ 
\dot s + \LieD_v s = 0. 
\end{cases} 
\end{align} 
In our case, $S = \Dens(M) \times \Omega_{cl}^2(M)$ and the momentum map is 
\begin{equation} 
(\rho,\theta,\beta,[P]) \mapsto \rho\,(\ud \theta + \iota_u\beta)\otimes \vol , 
\quad 
\rho\,\iota_u\vol = \ud P. 
\end{equation} 
The rest of the proof follows from direct calculations. 
\end{proof} 
\begin{corollary} 
The equations \eqref{eq:mhd_poisson_reduced} admit special `horizontal' solutions 
corresponding to momenta of the form 
\begin{equation} 
m = \rho \big( \ud \theta + \iota_u\beta) \otimes\vol, 
\quad 
\iota_u\varrho = \ud P. 
\end{equation} 
These solutions can be expressed in the variables 
$(\rho,\beta,\theta,[P]) \in T^*(\Dens(M)\times \Omega_{cl}^2(M))$ 
as a canonical Hamiltonian system for the Hamiltonian 
\begin{equation} \label{eq:long_ham} 
\widetilde H(\rho,\beta,\theta,[P]) 
= 
\int_M \Big( \tfrac{1}{2}\iota_v (\ud \theta + \iota_u\beta) \rho\,\vol + e(\rho)\rho\,\vol + \tfrac{1}{2}\beta\wedge\star\beta \Big) 
\end{equation} 
where 
$v^\flat = \ud \theta + \iota_u\beta$ and $\iota_u\varrho = \ud P$. 
\end{corollary} 
\begin{proof} 
 The horizontal solutions correspond to the submanifold $J^{-1}(0)$, where $J$ is the momentum map associated with the subgroup $\mathcal G$. 
 We refer to \autoref{sub:symplectic_reduction} for details on symplectic reduction. 
The Hamiltonian \eqref{eq:long_ham} is just the restriction of $\bar H$ to the special momenta. 
\end{proof} 
%

\subsection{Relativistic inviscid Burgers' equation} \label{sec:relativistic_euler} 

In this section we present a relativistic version of the Otto calculus, motivated by the treatment in \citet{Br2003c}. 
We show that it leads to a relativistic Lagrangian on $\Diff(M)$ and employ Poisson reduction 
of \autoref{sub:poisson_reduction} to obtain the relativistic hydrodynamics equations. 

As in the classical case, we consider a path in the space of diffeomorphisms as a family of free relativistic particles. 
Given $\varphi\colon[0,1]\times M\to M$ the action is then given by 
\begin{equation} \label{eq:action_otto_relativistic} 
S(\varphi) 
= 
-\int_{0}^1 \int_M 
c^2 \sqrt{1 - \frac{1}{c^2} \, \met\Big(\frac{\partial\varphi}{\partial \tau},\frac{\partial\varphi}{\partial \tau}\Big)} 
\, \vol \,\ud \tau. 
\end{equation} 
It is natural to think of this action as the restriction to a fixed reference frame of 
the corresponding action functional $\mathbb S\colon \Diff(\bar M)\to \RR$ 
on the Lorentzian manifold $\bar M = [0,1]\times M$ equipped with the Lorentzian metric 
\begin{equation*} 
\bar\met\big((\dot \tau,\dot x),(\dot \tau,\dot x)\big) = c^2\dot \tau^2 - \met(\dot x,\dot x). 
\end{equation*} 
More explicitly, this extended action is given by 
\begin{equation} \label{eq:action_otto_relativistic_full} 
\mathbb S(\bar\varphi) 
= 
\int_{\bar M} 
\sqrt{\bar\met\Big( \dot{\bar\varphi}, \dot{\bar\varphi} \Big)} 
\bar \vol 
\end{equation} 
where $\bar\vol = -c\, \ud \tau\wedge\vol$ is the volume form 
associated with $\bar\met$.\footnote{While in classical mechanics the action stands for the length square, 
note that in the classical limit, i.e. for small velocities, 
$\sqrt{1-\frac{1}{c^2}\met\Big(\dot\varphi,\dot\varphi\Big)} 
\approx 
\left(1-\frac{1}{2c^2}\met\Big(\dot\varphi,\dot\varphi\Big)\right)$, 
so that formula \eqref{eq:action_otto_relativistic_full} leads to the classical action.} 
In contrast with the classical case, the action \eqref{eq:action_otto_relativistic_full} is left-invariant 
under the subgroup of Lorentz transformations 
$\Diff_{\bar\met}(\bar M) = \{ \bar\varphi\in\Diff(\bar M)\mid \bar\varphi^*\bar\met = \bar\met\}$ 
in the following sense: 
if $\bar\eta = (\tau,\eta)\in \Diff_{\bar\met}(\bar M)$ then 
\begin{equation*} 
\mathbb S(\bar\eta\circ\bar\varphi) 
= 
\int_{\bar M} 
\sqrt{\bar\met \big( T\bar\eta\cdot\dot{\bar\varphi}, T\bar\eta\cdot\dot{\bar\varphi} \big)} 
\, \bar\vol 
= 
\int_{\bar M} 
\sqrt{\bar\eta^*\bar\met\big( \dot{\bar\varphi}, \dot{\bar\varphi} \big)} 
\, \bar\vol 
= 
\mathbb S(\bar\varphi). 
\end{equation*} 
Returning to \eqref{eq:action_otto_relativistic}, the associated Lagrangian on $\Diff(M)$ is 
\begin{equation} \label{eq:otto_lagrangian_relativistic} 
L(\varphi,\dot\varphi) 
= 
- \int_M c^2 \sqrt{1-\frac{1}{c^2}\met\big(\dot\varphi,\dot\varphi\big)} \, \vol. 
\end{equation} 
Since the Lagrangian is right-invariant with respect to $\Diffvol(M)$, we can carry out Poisson reduction of 
the corresponding Hamiltonian system on $T^*\Diff(M)$ as described above. 

\citet{Br2003c} used such an approach to derive a relativistic heat equation. 
We are now in a position to use it for relativistic hydrodynamics. 
\begin{theorem} 
The relativistic Lagrangian \eqref{eq:otto_lagrangian_relativistic} on $\Diff(M)$ induces a Poisson reduced system 
on $\Dens(M)\times\Xcal^*(M)$. 
The Hamiltonian is given by 
\begin{equation} 
\bar H(\rho,m) = \int_M \sqrt{ \rho^2 + \frac{1}{c^2} \met^\flat(m,m) } \,\vol 
\end{equation} 
and the governing equations are 
\begin{align}\label{eq:ham_form_rel_otto} 
\begin{cases} 
&\dot m 
= 
- \LieD_v m - \rho\,\ud \Bigg( \dfrac{c^2 \rho}{ \sqrt{ \rho^2 + c^{-2} \met^\flat(m,m) } } \Bigg) \otimes \vol 
\\ 
&\dot\rho + \divv(\rho v) =0 
\end{cases} 
\end{align} 
where $v^\flat\otimes\vol = m/\sqrt{ \rho^2 + c^{-2} \met^\flat(m,m) }$. 
\end{theorem} 
\begin{proof} 
The reduced Lagrangian for \eqref{eq:otto_lagrangian_relativistic} is given by 
\begin{equation*} 
\ell(\rho,v) = \int_M -c^2 \sqrt{ 1-\frac{1}{c^2}\met (v,v) } \, \rho \,\vol \,.
\end{equation*} 
The momentum variable is given by the Legendre transformation 
\begin{equation*} 
m 
= 
\frac{\delta L}{\delta v} 
= 
\gamma \rho v^\flat \otimes\vol
\quad {\rm for} \quad 
\gamma = \frac{1}{\sqrt{1- c^{-2} \met(v,v) }} 
\end{equation*} 
with the inverse 
\begin{equation*} 
v = \frac{m^\sharp}{ \sqrt{ \rho^2 + c^{-2} \met^\flat(m,m) } }. 
\end{equation*} 
The corresponding Hamiltonian is 
\begin{equation*} 
\begin{split} 
\bar H(\rho,m) 
&= 
\pair{m,v} - \ell(\rho,v) 
= 
\int_M c^2 \sqrt{\rho^2 + \frac{\met^\flat(m,m)}{c^2}}\,\vol \,,
\end{split} 
\end{equation*} 
so that 
\begin{equation*} 
\frac{\delta \bar H}{\delta \rho} 
= 
\frac{c^2 \rho}{\sqrt{\rho^2 + c^{-2} \met^\flat(m,m) }} \,,
\end{equation*} 
and the result follows from \autoref{cor:poisson_system_reduced}. 
\end{proof} 
\begin{remark} 
As $c\to\infty$ we formally recover the classical inviscid Burgers equation in \autoref{sub:burgers}. 
Indeed, assuming $\met^\flat(m,m)$ is small in comparison with $c^2$, a Taylor expansion of the right-hand side of \eqref{eq:ham_form_rel_otto} gives
\begin{align*}
	&-\ud \Bigg( \dfrac{c^2 \rho}{ \sqrt{ \rho^2 + c^{-2} \met^\flat(m,m) } } \Bigg)  = -\ud \Bigg( 
	c^2 - \dfrac{1}{2\rho^2}\met^\flat(m,m) + O(c^{-2} \met^\flat(m,m)) \Bigg) \\ & \to \ud\big(\met^\flat(m,m)/2\rho^2\big) \quad\text{as}\quad c\to\infty .
\end{align*}
As we also have $v \to m/\rho$ as $c\to\infty$ we recover the classical inviscid Burgers equation.
\end{remark} 
\begin{remark} 
In order to  obtain the equations of relativistic hydrodynamics one needs to incorporate internal energy 
via the reduced Hamiltonian on $\Dens(M)\times \Xcal^*(M)$ given by 
\begin{equation*} 
\bar H(\rho,m) = \int_M (c^2+e(\rho)) \sqrt{\rho^2 + \frac{\met^\flat(m,m)}{c^2}}\,\vol \,,
\end{equation*} 
where $e$ is the internal energy function, cf.\ \citet{LaLi1959} and \citet{HoKu1984}. 
This gives a relativistic version of the classical barotropic equations in \autoref{sub:compressible_euler_equations}. 
\end{remark} 
\section{Fisher-Rao geometry} \label{sec:fisher_rao} 

\subsection{Newton's equations on \texorpdfstring{$\Diff(M)$}{Diff(M)}} 
\label{sub:newton_s_equation_H1} 

We now focus on another important Riemannian structure on $\Diff(M)$. 
This structure is induced by 
the Sobolev $H^1$-inner product on vector fields 
and has the same relation to the Fisher-Rao metric on $\Dens(M)$ 
as the $L^2$-metric on $\Diff(M)$ to the Wasserstein-Otto metric on $\Dens(M)$. 


%
\begin{definition} \label{def:H1dotmet} 
Let $(M,\met)$ be a compact Riemannian manifold with volume form $\mu$. 
For any $\varphi \in \Diff(M)$ and $v \in T_e\Diff(M)$ we set 
\begin{equation} \label{eq:H1met} 
\Met_{\varphi}(v\circ\varphi,v\circ\varphi) = \int_M \met(-\Delta v,v) \, \vol + F(v,v) \,,
\end{equation} 
where $\Delta$ is the Laplacian on vector fields and $F$ is a quadratic form 
depending only on the vertical (divergence-free) component of $v$. 
\end{definition} 
\begin{remark} 
From the point of view of the geometry of $\Dens(M)$ (and for most of our applications) 
only the first term on the right-hand side of \eqref{eq:H1met} is relevant. 
However, it is convenient to work with the above metric on $\Diff(M)$, in particular, because of 
its relation to a number of familiar equations, cf. \cite{KhLeMiPr2013, Mo2015} and below. 
Note also the following analogy between the Wasserstein and the Fisher-Rao structures: 
while the non-invariant $L^2$-metric 
induces a factorization of $\Diff(M)$ where one of the factors solves the optimal mass transport problem, 
the invariant metric \eqref{eq:H1met} 
induces a different factorization of $\Diff(M)$ which solves an \emph{optimal information transport} problem; 
cf.\ \cite{Mo2015}. 
\end{remark} 
Consider a potential function of the form 
\begin{equation} \label{eq:Ubar_H1} 
U(\varphi) = \bar U(\varphi^*\vol), 
\quad 
\varphi \in \Diff(M), 
\end{equation} 
where $\bar U$ is a potential functional on $\Dens(M)$.
(In this section it is convenient to work with volume forms $\varrho$ instead of $\rho$.) 
It is interesting to compare the present setting with that of Section \ref{sub:newton_s_equation}, 
where the potential function on $\Diff(M)$ was defined using pushforwards rather than pullbacks. 
As a result one works with the left cosets rather than with the right cosets, 
cf. Remark \ref{rem:left-right} below. 
\begin{theorem} \label{thm:newton_for_Ubar_H1} 
Newton's equations of the metric \eqref{eq:H1met} on $\Diff(M)$ with a potential function \eqref{eq:Ubar_H1} 
have the form 
%
\begin{align} \label{eq:reduced_newton_diff_H1} 
\begin{cases} 
&\mathcal A\dot v  
+ 
\LieD_v \mathcal A v 
+ 
\ud \Big( \frac{\delta \bar U}{\delta\varrho}(\varphi^*\vol) \circ\varphi^{-1} \Big) \otimes \vol  
= 0 
\\ 
&\dot\varphi = v\circ\varphi \,,
\end{cases} 
\end{align} 
%
%
where the inertia operator $\mathcal A\colon \Xcal(M)\to \Xcal^*(M)$ is given by 
\begin{equation} \label{eq:A} 
\mathcal A v = \big( {-}\Delta v^\flat + F(v,\cdot) \big) \otimes \vol. 
\end{equation} 
%
\end{theorem} 
%
%
%
%
\begin{proof}
The derivation of the equation in the case of zero potential can be found in \cite{Mo2015}. 
Modifications needed here follow from the calculation 
\begin{equation*} 
\frac{\ud}{\ud s}\Big|_{s=0} \bar U(\varphi_s^*\vol) 
= 
\int_M \tfrac{\delta\bar U}{\delta \varrho} \, \varphi^*\LieD_u\vol 
= 
\Big\langle \ud \big( \tfrac{\delta \bar U}{\delta\varrho} \circ\varphi^{-1} \big) \otimes\vol, u \Big\rangle \,,
\end{equation*} 
where $s \to \varphi_s \in \Diff(M)$ is the flow of the vector field $u$ in $\Xcal(M)$. 
\end{proof} 
%





We proceed with a Hamiltonian formulation. 
As in \autoref{sub:poisson_reduction} we will identify cotangent spaces $T_\varphi^*\Diff(M)$ with $\Xcal^*(M)$. 



%
\begin{proposition} \label{prop:ham_form_newton_on_diff_H1} 
The Hamiltonian form of Newton's equations \eqref{eq:reduced_newton_diff_H1} 
on $T^*\Diff(M)$ is 
\begin{align} \label{eq:ham_form_of_newton_fisher} 
\begin{cases} 
&\dfrac{\ud}{\ud t}\varphi^* m 
+ 
\ud \Big( \frac{\delta \bar U}{\delta\varrho}(\varphi^*\vol) \Big) \otimes \varphi^*\vol 
= 0 
\\ 
&\dot\varphi = v\circ\varphi 
\end{cases} 
\end{align} 
where $m = \mathcal A v\in \Xcal^*(M)$. 
\end{proposition} 
\begin{proof} 
This follows simply by pulling back by $\varphi$ the equations in \eqref{eq:reduced_newton_diff_H1} 
and applying the identity 
\begin{equation*} 
\frac{\ud}{\ud t}\varphi^*m = \varphi^*\dot m + \varphi^*\LieD_v m. 
\end{equation*} 
%
\end{proof} 
\begin{remark} 
Observe that if the potential function is zero, then the equation in \eqref{eq:ham_form_of_newton_fisher} 
expresses conservation of the momentum $\varphi^*m$ 
associated with the right invariance of the metric. 
\end{remark} 
%

\subsection{Riemannian submersion over densities} 
\label{sub:fibration_H1} 

We turn to the geometry of the fibration of $\Diff(M)$ with respect to the metric \eqref{eq:H1met}. 
\begin{definition} 
The \emph{right coset projection} $\pi\colon\Diff(M)\to \Dens(M)$ between diffeomorphisms 
and smooth probability densities is given by 
\begin{equation} \label{eq:projection_H1} 
\pi(\varphi) = \varphi^*\vol. 
\end{equation} 
\end{definition} 

As before, it turns out that the projection \eqref{eq:projection_H1} is a Riemannian submersion 
if the base space is equipped with a suitable metric. 
\begin{definition} \label{def:fisher_rao_metric} 
The \emph{Fisher-Rao metric} is the Riemannian metric on $\Dens(M)$ given by 
\begin{equation} \label{eq:fisher_rao_metric} 
\MetF_\varrho(\dot\varrho,\dot\varrho) 
= 
\int_M \frac{\dot\varrho}{\varrho} \frac{\dot\varrho}{\varrho} \, \varrho, 
\end{equation} 
where $\dot\varrho \in \Omega_0^n(M)$ represents a tangent vector at $\varrho \in \Dens(M)$. 
\end{definition} 
\begin{theorem} \label{thm:FR_riemannian_metric} 
The right coset projection \eqref{eq:projection_H1} is a Riemannian submersion 
with respect to the metric \eqref{eq:H1met} on $\Diff(M)$ and the Fisher-Rao metric on $\Dens(M)$. 
In particular, if $\dot\varphi\in T_\varphi\Diff(M)$ is horizontal, i.e., 
\begin{equation*} 
\Met_\varphi(\dot\varphi,\dot\eta) = 0, \quad\forall\, \dot\eta\in \ker(T_\varphi\pi), 
\end{equation*} 
then 
$\Met_\varphi(\dot\varphi,\dot\varphi) = \MetF_{\pi(\varphi)}(\dot\varrho,\dot\varrho)$ 
where $\dot\varrho = \pi_{\ast \varphi} \dot{\varphi}$. 
\end{theorem} 
\begin{proof} 
	See \cite[Thm.\ 4.9]{Mo2015}.
\end{proof} 
Note also that it follows from the Hodge decomposition that the horizontal distribution on $\Diff(M)$ consists of elements 
of the form $\nabla p\circ\varphi$, cf. \cite{Mo2015} for details. 
\begin{remark} \label{rem:left-right} 
Let us summarize the definition of the two metrics on $\Dens(M)$ that we discussed so far.
The Wasserstein--Otto metric (cf.~\autoref{sub:fibration}) is defined as follows:
$$
 \bar\Met^{WO}_\rho(\dot\rho,\dot\rho) = 
		\int_M |{\nabla \theta}|^2 \,\rho\mu \quad\text{where}\quad \dot\rho + {\rm div}(\rho \nabla \theta)=0 .
$$
Note that it depends on the Riemannian structure on $M$.
The Fisher-Rao metric, on the other hand, is given by the following ``universal formula":
$$
 \bar\Met^{FR}_\rho(\dot\rho,\dot\rho) = \int_M \frac{\dot\rho^2}{\rho} \vol 
$$
and is independent of the Riemannian structure on $M$.

Note also that the setting of \autoref{thm:FR_riemannian_metric} is quite different 
from that of \autoref{thm:otto_riemannian_metric}. 
In the latter, the Riemannian metric on $\Diff(M)$ is right-invariant with respect to $\Diffvol(M)$ 
and automatically descends to the quotient from the right, namely $\Diff(M)/\Diffvol(M)$. 
In the former, the metric is right-invariant with respect to $\Diff(M)$ 
and descends to the quotient from the left, namely $\Diffvol(M)\backslash \Diff(M)$. 
Thus, in \autoref{thm:FR_riemannian_metric} the right-invariance property 
is retained after taking the quotient and therefore the Fisher-Rao metric on $\Dens(M)$ 
remains right-invariant with respect to the action of $\Diff(M)$ 
(corresponding to right translation of the fibers), which is easy to verify. 
\end{remark} 
\begin{proposition} \label{prop:FRgradient} 
The gradient of a smooth function $\bar U\colon \Dens(M)\to\RR$ with respect to the Fisher-Rao metric is 
\begin{equation} \label{eq:FRgradient} 
\nabla^\MetF \bar U(\varrho) = \frac{\delta \bar U}{\delta \varrho}\varrho - \lambda\varrho, 
\end{equation} 
where $\lambda$ is a Lagrange multiplier such that $\nabla^\MetF \bar U(\varrho) \in T_\rho\Dens(M)$. 
\end{proposition} 
\begin{proof} 
Let $\dot\varrho\in\Omega^{n}_{0}(M)$ and let $\delta \bar{U}/\delta \varrho$ 
be a representative of the variational derivative in $C^{\infty}(M)/\RR$. 
We have 
\begin{equation*} \label{eq:FRgradientcalc} 
\MetF_{\varrho} \big( \nabla^{\MetF}\bar U(\varrho),\dot\varrho \big) 
= 
\Big\langle \tfrac{\delta \bar U}{\delta \varrho},\dot\varrho \Big\rangle 
= 
\int_M \tfrac{\delta \bar U}{\delta \varrho} \, \dot{\varrho} 
= 
\MetF_{\varrho} \big( \tfrac{\delta \bar U}{\delta \varrho}\varrho,\dot\varrho \big) \,,
\end{equation*} 
which yields the explicit form of the gradient. 
\end{proof} 

We end this subsection by recalling a particularly remarkable property of the Fisher-Rao metric. 
Let 
$S^\infty(M)= \big\{ f \in C^\infty(M)\mid \int_M f^2\vol = 1 \big\}$ 
be the unit sphere in the pre-Hilbert space $C^\infty(M)\subset L^2(M,\RR)$. 
\begin{theorem} \label{prop:square_root_map} 
The square root map 
\begin{equation} \label{eq:square_root_map} 
\phi\colon\Dens(M)\to S^\infty(M), \quad \varrho \mapsto  \sqrt{\rho} 
\end{equation} 
is, up to a factor $4$, a Riemannian isometry between $\Dens(M)$ equipped with the Fisher-Rao metric $\MetF$ in \eqref{eq:fisher_rao_metric} 
and the (geodesically convex) subset 
\begin{equation*} 
S_+^\infty(M) = \big\{ f\in S^\infty(M) \mid f>0  \big\}
\end{equation*} 
of the sphere $S^\infty(M)$.
\end{theorem} 
This result was first obtained by Friedrich \cite{Fr1991} and later independently 
in \cite{KhLeMiPr2013} 
in the Euler-Arnold framework of diffeomorphism groups.

\subsection{Newton's equations on \texorpdfstring{$\Dens(M)$}{Dens(M)}} 
\label{sub:newton_on_dens_fisher} 

Recall that in \autoref{sub:symplectic?} the Hamilton equations on $T^*\Dens(M)$ 
were obtained by symplectic reduction of a $\Diffvol(M)$-invariant system on $T^*\Diff(M)$. 
In the setting with the right coset projection \eqref{eq:projection_H1} and the metric \eqref{eq:H1met} 
the situation is quite different, since the Riemannian metric is not left-invariant with respect to $\Diffvol(M)$ 
(otherwise, interchanging push-forwards and pull-backs would give a completely `dual' theory).
Nevertheless, there is a \emph{zero momentum reduction} on the Hamiltonian side corresponding to 
the Riemannian submersion structure described in \autoref{sub:fibration_H1}. 
\begin{proposition} \label{pro:cohorizontal_solutions} 
The exact momenta, i.e.  tensor products  of the form 
\begin{equation} \label{eq:coinvariant_momentum_form} 
\big\{  \ud f \otimes \vol~|~ f \in C^\infty(M) \big\}, 
\end{equation} 
form an invariant set for the system \eqref{eq:ham_form_of_newton_fisher}. 
\end{proposition} 
\begin{proof} 
Substituting \eqref{eq:coinvariant_momentum_form} in \eqref{eq:ham_form_of_newton_fisher} 
we get 
\begin{equation*} 
\ud \Big( \tfrac{\ud}{\ud t}\varphi^* f \Big) \otimes \varphi^*\vol 
+ 
\ud \big( \varphi^* f \big) \otimes \tfrac{\ud}{\ud t}\varphi^*\vol 
+ 
\ud \Big( \tfrac{\delta \bar U}{\delta\rho}(\varphi^*\vol) \Big) \otimes \varphi^*\vol 
= 0, 
\end{equation*} 
where $\mathcal{A} v = \mathrm{d}f \otimes \mu$ and 
$$ 
\frac{\ud}{\ud t}\varphi^*\vol = \varphi^*\LieD_u\vol = \varphi^\ast (\divv v) \, \varphi^\ast \vol. 
$$ 
From \eqref{eq:A} we find that solutions of the form $v = \nabla p$ 
define (up to a constant) 
$f = \Delta p = \divv v$, so that 
\begin{equation*} \label{eq:descending_equation} 
\ud \Big( \tfrac{\ud}{\ud t}\varphi^* f \Big) \otimes \varphi^*\vol 
+ 
\ud (\varphi^*f)^2 \otimes \varphi^*\vol 
+ 
\ud \Big( \tfrac{\delta \bar U}{\delta\rho}(\varphi^*\vol) \Big) \otimes \varphi^*\vol 
= 0. 
\end{equation*} 
Using 
$\frac{\ud}{\ud t} (\varphi^* f) = \varphi^*\dot f + \varphi^*\LieD_v f$ 
we then obtain 
\begin{equation} \label{eq:descending_equation_reduced} 
\varphi^* \Big( \ud \Big( 
\dot f 
+ 
\LieD_v f 
+ 
f^2 
+ 
\tfrac{\delta \bar U}{\delta\rho}(\varphi^*\vol)\circ\varphi^{-1}  
\Big) \otimes \vol \Big) 
= 0 \,,
\end{equation} 
which proves the assertion. 
\end{proof}
%
%
\begin{theorem} \label{prop:newton_on_dens_fisher} 
Newton's equations with respect to the Fisher-Rao metric \eqref{def:fisher_rao_metric} on $\Dens(M)$ 
and a potential $\bar U\colon \Dens(M)\to \RR$ have the form 
\begin{equation} \label{eq:EL_eq_dens_fisher} 
\ddot\rho 
- 
\frac{\dot\rho^2}{2\rho}
+ 
\frac{\delta \bar U}{\delta \rho}\rho
= 
\lambda\rho 
\end{equation} 
where $\lambda$ is a multiplier subject to $\int_M\rho\mu = 1$.
Furthermore, the Lagrangian and Hamiltonian are 
$ 
L(\rho,\dot\rho) = \frac{1}{2} \MetF_\rho(\dot\rho,\dot\rho) - \bar U(\rho) 
$ 
and 
$ 
H(\rho,\theta) = \frac{1}{2}\pair{\theta^2,\rho} + \bar U(\rho), 
$ 
respectively. 
The corresponding Hamiltonian equations have the form 
\begin{align} \label{eq:ham_eq_dens_fisher} 
\begin{cases} 
&\dot\rho - \theta \rho = 0 
\\ 
&\dot\theta + \frac{1}{2}\theta^2 + \frac{\delta \bar U}{\delta \rho}(\rho) = \lambda . 
\end{cases} 
\end{align} 
%
Solutions of \eqref{eq:ham_eq_dens_fisher} correspond to potential solutions 
(cf.\ \autoref{pro:cohorizontal_solutions}) 
of Newton's equations \eqref{eq:ham_form_of_newton_fisher} on $\Diff(M)$.
%
\end{theorem}
\begin{proof} 
The result follows directly from the proof of \autoref{pro:cohorizontal_solutions} 
by setting $\theta = \varphi^* f$ and $\rho \mu= \varphi^*\vol$. 
\end{proof} 
%

\section{Fisher-Rao examples} 

\subsection{\texorpdfstring{The $\mu$}{mu}CH equation and Fisher-Rao geodesics} \label{sub:muCH}

The periodic $\mu$CH equation (also known in the literature as the $\mu$HS equation) 
is a nonlinear evolution equation of the form 
\begin{equation} \label{eq:muHS} 
\mu(u_t) - u_{xxt} - 2u_x u_{xx} - uu_{xxx} + 2 \mu(u) u_x = 0 
\end{equation} 
where $\mu(u) = \int_{S^1} u \, dx$. 
It was derived in \cite{KhLeMi2008} as an Euler-Arnold equation on the group of diffeomorphisms of the circle 
equipped with the right-invariant Sobolev metric given at the identity by the inner product 
\begin{equation*} \label{eq:muH1} 
\langle u, v \rangle_{H^1} = \mu(u) \mu(v) + \int_{S^1} u_x v_x \, \ud x. 
\end{equation*} 

The $\mu$CH equation is known to be bihamiltonian and admit smooth, as well as cusped, soliton-type solutions. 
It may be viewed as describing a director field in the presence of an external (e.g., magnetic) force. 
The associated Cauchy problem has been studied extensively in the literature, cf.\ 
\cite{KhLeMi2008,GuLiZh2011,QuZhLiLi2014}.
Many of its geometric properties can also be found in \cite{TiVi2011}. 
The following result was proved in \cite{Mo2015} 
\begin{proposition} \label{prop:muCH}
The $\mu$CH equation \eqref{eq:muHS} is a (right-reduced) Newton's equation 
\eqref{eq:reduced_newton_diff_H1} 
with vanishing potential on $S^1$. 
%
Geodesics of the Fisher-Rao metric \eqref{eq:fisher_rao_metric} on $\Dens(S^1)$ correspond to 
horizontal solutions of the $\mu$CH equation described by the equations 
\begin{align} 
\begin{cases} 
\dot\rho - \theta \rho = 0, 
\\ 
\dot\theta + \frac{1}{2} \theta^2 = \frac{1}{2}\int_{S^1}\theta^2\rho\, \ud x . 
\end{cases} 
\end{align} 
(As in \autoref{prop:newton_on_dens_fisher}, the relation between $u$, $\rho$ and $\theta$ is given by  
$\rho = \varphi_x$, 
where $\varphi$ is the Lagrangian flow of $u$, and $\theta = u_x\circ\varphi + \text{\rm const}$ with the constant choosen so that $\int_{S^1}\theta\rho\,\ud x = 0$.)
\end{proposition} 

Observe that the Euler-Arnold equation of the metric \eqref{eq:H1met} can be naturally viewed as 
a higher-dimensional generalization of the equation \eqref{eq:muHS}, see \cite{Mo2015}. 
Furthermore, in the one-dimensional case horizontal solutions of this equation 
can be written in terms of the derivative $u_x$. 
In higher dimensions we similarly have 
\begin{proposition} 
The geodesic equations of the Fisher-Rao metric \eqref{eq:fisher_rao_metric} on $\Dens(M)$ 
reduce to the following equations on $T^*_\mu\Dens(M)$ 
\begin{equation*} 
\dot f + \LieD_{\nabla p} f + \frac{1}{2}f^2 
= 
\frac{1}{2}\int_M f^2\, \vol, 
\quad \Delta p = f 
\end{equation*} 
where $f = \divv u$ and $\theta = f\circ\varphi$. 
\end{proposition} 
\begin{proof}
The equations follow directly from \eqref{eq:descending_equation_reduced} with $\bar U \equiv 0$. Note that 
this equation preserves the zero-mean condition for $f$, while  the constant appearing in
the right-hand side of this equation, as well as the one in   \autoref{prop:muCH}, is conditioned by the equation's solvability.
\end{proof}
%

\subsection{The infinite-dimensional Neumann problem} 
\label{sect:neumann} 
The C.~Neumann problem (1856) describes the Newtonian motion of a point on the $n$-dimensional sphere $S^n$ under the influence 
of a quadratic potential, see \autoref{sect:newton-ex}. It is known to be equivalent (up to a change of the time parameter) to 
the geodesic equations on an ellipsoid in $\RR^{n+1}$ with the induced metric, see e.g., \cite{Ne1856,Mo1983}. 

Here we describe a natural infinite-dimensional generalization of the C.~Neumann problem. 
Consider the   infinite-dimensional unit sphere 
\begin{equation*} 
S^\infty(M) = \Big\{ f \in C^\infty(M)\mid \int_M f^2 \vol = 1 \Big\} 
\end{equation*} 
in the pre-Hilbert space $C^\infty(M) \cap L^2(M,\vol)$ 
and the quadratic potential function 
\begin{equation} \label{eq:inf_neumann_potential} 
V(f) 
= 
\frac{1}{2}\langle \nabla f,\nabla f\rangle_{L^2} 
= 
\frac{1}{2}\int_M \lvert\nabla f\rvert^2 \vol. 
\end{equation} 
We seek a curve $f\colon [0,1]\to S^\infty(M)$ that minimizes the action functional 
for the Lagrangian 
\begin{equation*} 
L(f,\dot f) 
= 
\frac{1}{2} \langle \dot f,\dot f \rangle_{L^2} - \frac{1}{2} \langle \nabla f, \nabla f \rangle_{L^2} 
= 
\frac{1}{2} \int_M \big( \dot{f}^2 + f \Delta f \big) \vol. 
\end{equation*} 
\begin{proposition} 
Newton's equations associated with the infinite-dimensional Neumann problem 
with potential \eqref{eq:inf_neumann_potential} have the form 
\begin{equation} \label{eq:inf_neumann_eq} 
\ddot f -\Delta f = -\lambda f \,,
\end{equation} 
where $\lambda$ is a Lagrange multiplier subject to the constraint $\int_M f^2\vol = 1$. 
In fact, we have 
$\lambda = 2 L(f,\dot f) = \int_M (\dot f^2 + f \Delta f)\vol$. 
\end{proposition} 
\begin{proof} 
This is a simple consequence of the integration by parts formula. 
\end{proof} 

Our next objective is to show that the infinite-dimensional Neumann problem on $S^\infty(M)$ 
corresponds to Newton's equations on $\Dens(M)$ with respect to the Fisher-Rao metric and a natural choice of 
the potential function. 
The latter is given by the Fisher information functional 
\begin{equation} \label{eq:Fisher_info_func} 
I(\rho) = \frac{1}{2}\int_M \frac{\abs{\nabla \rho}^{2}}{\rho} \, \vol, 
\end{equation} 
where the density is $\rho \in \Dens(M)$.
\begin{lemma} 
The gradient of $I(\rho)$ with respect to the Fisher-Rao metric can be computed from 
either of the two expressions 
\begin{equation*} 
\begin{aligned} 
\nabla^\Met I(\rho) 
&= 
\bigg( \frac{1}{2} \frac{ \abs{\nabla\rho}^2 }{ \rho } - \Delta\rho \bigg) \vol - \lambda\rho\mu 
\\ 
&= 
-2 \Big( \sqrt\rho \, \Delta\sqrt\rho \Big) \vol - \lambda\rho\mu. 
\end{aligned} 
\end{equation*} 
\end{lemma} 
\begin{proof} 
Using the identities 
$\nabla \log{\rho} = \nabla\rho/\rho$ 
and 
$\nabla\sqrt{\rho} = \frac{1}{2} \rho^{-\frac{1}{2}} \nabla\rho$ 
we can rewrite the Fisher information functional as 
\begin{equation} 
I(\rho) 
= 
\frac{1}{2} \int_M | \nabla\log\rho |^2 \rho\mu 
= 
2\int_M \big| \nabla\sqrt\rho \big|^2 \vol. 
\end{equation} 
Differentiating the first of these expressions in the direction of the vector 
$\dot\rho $ yields 
\begin{equation} 
\begin{aligned} 
\pair{\frac{\delta I}{\delta \rho},\dot\rho} 
&= 
\int_M \Big( 
\tfrac{1}{2}\abs{\nabla\log\rho}^2\dot\rho 
+ 
\met \big( \nabla\log\rho, \nabla (\dot{\rho}/\rho) \big) \rho 
\Big) \mu
\\ 
&= 
\int_M \Big( 
\tfrac{1}{2}\abs{\nabla\log\rho}^2\dot\rho
- 
\rho^{-1} \Delta\rho \, \dot\rho
\Big) \mu
= 
\pair{\frac{1}{2}\frac{\abs{\nabla\rho}^2}{\rho^2} - \frac{\Delta\rho}{\rho},\dot\rho}. 
\end{aligned} 
\end{equation} 
Similarly, differentiating the second yields 
\begin{equation} 
\begin{aligned} 
\pair{\frac{\delta I}{\delta \rho},\dot\rho} 
= 
2 \int_M \met \Big( \nabla\sqrt\rho,\nabla \big( \dot\rho/\sqrt\rho \big) \Big) \vol 
= 
\bigg\langle {-2} \frac{\Delta\sqrt\rho}{\sqrt\rho},\dot\rho \bigg\rangle. 
\end{aligned} 
\end{equation} 
The result now follows from \autoref{prop:FRgradient}. 
\end{proof} 
\begin{proposition} \label{pro:newton_fisher_information_neumann} 
Newton's equations \eqref{eq:EL_eq_dens_fisher} on $\Dens(M)$ with respect the Fisher-Rao metric 
and the Fisher-Rao potential \eqref{eq:Fisher_info_func} are 
\begin{align*}  
& \ddot\rho - \Delta\rho - \frac{1}{2\rho} \big( \dot\rho^2 - \abs{\nabla\rho}^2 \big) 
= 
\lambda \rho, 
\end{align*} 
where 
$\lambda$ is a Lagrange multiplier for the constraint $\int_M\rho\mu = 1$. 
The map $\rho \mapsto f = \sqrt{\rho}$ establishes an isomorphism with the infinite-dimensional Neumann problem \eqref{eq:inf_neumann_eq}. 
\end{proposition} 
\begin{proof}
%
The form of the equation on $\Dens(M)$ follows from \autoref{prop:newton_on_dens_fisher}. 
It is straighforward to check that 
$V(\sqrt{\rho}) = I(\rho)/4$. 
The result then follows from isometric properties of the square root map \eqref{eq:square_root_map}. 
\end{proof} 
\begin{remark} \label{rmk:stationary_neumann} 
Of particular interest are the stationary solutions to the Neumann problem \eqref{eq:inf_neumann_eq}, 
i.e., those with $\nabla^{S^{\infty}} V(f) =  \Delta f - \lambda f = 0$, 
in which case $f$ is a normalized eigenvector of the Laplacian with eigenvalue $\lambda$. 
If $\dot f = 0$ then $\lambda = \int_M f\Delta f\, \vol = -2V(f)$. 
Consequently, the stationary solutions correspond to the principal axes of the corresponding infinite-dimensional ellipsoid 
$\pair{f,\Delta f}_{L^2} = 1$. 

It is also possible to obtain quasi-stationary solutions this way. 
Indeed, assume that the eigenspace of $\lambda$ is at least two-dimensional (for example, when $M=S^n$). 
If $f_1,f_2\in S^\infty(M)$ are two orthogonal eigenvectors with eigenvalue $\lambda$ then it is straightforward to check 
that a solution originating from $f_1$ with initial velocity $a f_2$ for $a\in\RR$ is given by 
\begin{equation*} 
f(t,x) = \cos(a t) f_1 + \sin(a t)f_2. 
\end{equation*} 
\end{remark} 
%

\subsection{The Klein-Gordon equation} 
\label{sec:klein_gordon} 

The {\it Klein-Gordon equation} 
\begin{equation} \label{eq:klein_gordon} 
\ddot f - \Delta f  = -m^2 f , \qquad m\in\RR 
\end{equation} 
describes spin-less scalar particles of mass $m$. 
It is invariant under Lorentz transformations and can be viewed as a relativistic quantum equation. 
To see how it relates to the Neumann problem of the previous subsection 
let $M\times S^1$ denote the space-time manifold equipped with the Minkowski metric 
of signature $(+++-)$ 
and consider a quadratic functional 
\begin{equation*} 
\bar V(f) = \frac{1}{2}\int_{M\times S^1} (\abs{\nabla f}^2 - \dot f^2)\vol\wedge\ud t
\end{equation*} 
which is the $L^2$-norm of the the Minkowski gradient $\bar\nabla f = \big( \nabla f, - \dot{f} \big)$. 
\begin{proposition} \label{prop:klein_gordon} 
For the space-time manifold $M\times S^1$   solutions of the infinite-dimensional Neumann problem 
with potential $\bar V$  on the hypersurface 
\begin{equation*} 
S^\infty (M\times S^1) 
= 
\Big\{ f\in C^\infty(M\times S^1) \mid \int_{M\times S^1} f^2 \, \vol\wedge \ud t = 1 \Big\} 
\end{equation*} 
satisfy the Klein-Gordon equation \eqref{eq:klein_gordon} with mass parameter 
$m^2 = 2\bar V(f)$.
%
\end{proposition}
\begin{proof} 
This is a calculation analogous to that in \autoref{rmk:stationary_neumann}. 
\end{proof} 
%

\section{Geometric properties of the Madelung transform} 
\label{sec:madelung} 

In this section we recall several results concerning the Madelung transform which provides 
a link between geometric hydrodynamics and quantum mechanics, see \cite{KhMiMo2018, KhMiMo2019}. 
It was introduced in the 1920's by E. Madelung \cite{Ma1964} in an attempt to give a hydrodynamical formulation 
of the Schr\"odinger equation. 
Using the setting developed in previous sections one can now present a number of 
surprising geometric properties of this transform. 
\begin{definition} \label{def:madelung} 
Let $\rho$ and $\theta$ be real-valued functions on $M$ with $\rho >0$. 
The \emph{Madelung transform} is defined by 
\begin{equation} \label{eq:madelung_def} 
\Phi(\rho,\theta) = \sqrt{\rho \ee^{2\ii\theta/\hbar}},  
\end{equation} 
where $\hbar$ is a parameter (Planck's constant).\footnote{In the publications \cite{KhMiMo2018,KhMiMo2019} the convention $\hbar = 2$ is used.}
\end{definition} 
Observe that $\Phi$ is a complex extension of the square root map 
described in \autoref{thm:FR_riemannian_metric}. Heuristically, the functions $\sqrt\rho$ and $\theta/\hbar$  can be interpreted as
the absolute value and argument of the complex-valued function $\psi:=  \sqrt{\rho \ee^{2\ii\theta/\hbar}}$ as in polar coordinates.
Throughout this section we assume that $M$ is a compact simply connected manifold.

\subsection{Madelung transform as a symplectomorphism} 
\label{sub:madelung_symplec} 

Let $PC^\infty(M,\CC)$ denote the complex projective space of smooth complex-valued functions on $M$.
Its elements can be represented as cosets $[\psi]$ of the $L^2$-sphere of smooth functions, 
where $\tilde\psi\in[\psi]$ if and only if $\tilde\psi = \ee^{\ii \alpha}\psi$ for some $\alpha\in\RR$. 
A tangent vector at a coset $[\psi]$ is a linear coset of the form $[\dot\psi] = \{ \dot\psi + c\psi \mid c \in \RR \}$.
Following the geometrization of quantum mechanics by \citet{Ki1979}, a natural symplectic structure on the projective space $PC^\infty(M,\CC)$ is
\begin{equation}\label{eq:symplectic_form_PCinf}
	\Omega_{[\psi]}^{PC^\infty(M,\CC)}([\dot\psi_1],[\dot\psi_2]) = 2\hbar \int_M \mathrm{Im}(\dot\psi_1\overline{\dot\psi_2})\vol.
\end{equation}

The projective space $PC^\infty(M,\CC\backslash \{0\})$ of nonvanishing complex functions 
is a submanifold of $PC^\infty(M,\CC)$. 
It turns out that the Madelung transform induces a symplectomorphism between 
$PC^\infty(M,\CC\backslash \{0\})$ and the cotangent bundle of probability densities $T^*\Dens(M)$, see \autoref{fig:diagram2}.
 
\begin{figure}
\begin{equation*}
	\xymatrix@C2.0pc{
	& 	T^*\Dens(M)\ar[d]_{\text{}} \ar[rr]_{\text{}} &
	& \ar@{<->}[ll]^{\text{}}PL^2(M,\CC\backslash\{0\})\ar[d]_{\text{}} \\ 
	& \Dens(M)\ar@{<->}[rr]^{\text{}} 	&		& *+[r]{S_+^\infty\subset L^2(M,\RR)}		 
	}
\end{equation*}

\caption{The bottom arrow  is the isometry of the density space  $\Dens(M)$ with the Fisher-Rao metric and a part $S^\infty_+$ of the infinite-dimensional sphere, while the top arrow corresponds to the Madelung transform.}	
\label{fig:diagram2}
\end{figure}

Namely, we have 
\begin{theorem}[\cite{KhMiMo2019}] \label{thm:madelung_symplectomorphism} 
The Madelung transform \eqref{eq:madelung_def} induces a map 
\begin{equation} \label{eq:madelung_symplectic} 
\Phi\colon T^*\Dens(M)\to PC^\infty(M,\CC\backslash \{0\}) 
\end{equation} 
which is a symplectomorphism (in the Fr\'echet topology of smooth functions) 
with respect to the canonical symplectic structure of $T^*\Dens(M)$ 
and the symplectic structure \eqref{eq:symplectic_form_PCinf} of $PC^\infty(M,\CC)$. 
\end{theorem} 
The Madelung transform was shown to be a symplectic submersion 
from $ T^*\Dens(M)$ to the unit sphere of nonvanishing wave functions by \citet{Re2012}. 
The stronger symplectomorphism property stated in \autoref{thm:madelung_symplectomorphism} 
is deduced using the projectivization $PC^\infty(M,\CC\backslash \{0\})$.

\subsection{Examples: linear and nonlinear Schr\"odinger equations} 
\label{sect:NLS} 

Let $\psi$ be a wavefunction on $M$ and consider the family of Schr\"odinger equations (or Gross-Pitaevsky equations) 
with Planck's constant $\hbar$ and mass $m$ of the form 
\begin{equation} \label{eq:schrodinger} 
\mathrm{i}\hbar\dot\psi = - \frac{\hbar^2}{2m}\Delta\psi +  V\psi + f(\abs{\psi}^2)\psi 
\end{equation} 
where $V\colon M\to \RR$ and $f\colon \RR_{+}\to \RR$. 
Setting $f\equiv 0$ we obtain the linear Schr\"odinger equation with potential~$V$,  
while setting $V\equiv 0$ yields a family of nonlinear Schr\"odinger equations (NLS); 
typical choices are $f(a) = \kappa a$ or $f(a) = \frac 12(a-1)^2$. 

From the point of view of geometric quantum mechanics (cf.\ \citet{Ki1979}), equation \eqref{eq:schrodinger} is Hamiltonian with respect to 
the symplectic structure \eqref{eq:symplectic_form_PCinf}, which is compatible with the complex structure of $P L^2(M,\CC)$. 
The Hamiltonian associated with \eqref{eq:schrodinger} is 
\begin{equation} \label{eq:schrodinger_ham}
H(\psi) 
= 
\frac{\hbar^2}{2m}\norm{\nabla\psi}_{L^2}^{2} 
+ 
\int_M \left( V \abs{\psi}^2 + F(\abs{\psi}^{2}) \right) \vol \,,
\end{equation} 
where $F\colon \RR_{+}\to \RR$ is a primitive of $f$. 


Observe that the $L^2$ norm of a wave function satisfying the Schrödinger equation \eqref{eq:schrodinger} 
is conserved in time. 
Furthermore, the equation is equivariant with respect to phase change 
$\psi(x)\mapsto e^{i\alpha}\psi(x)$ and hence it descends to the projective space $PC^\infty(M,\CC)$. 

%
%
%
%
\begin{proposition}[cf.\ \cite{Ma1927,Re2012}] \label{prop:schrodinger}
The Madelung transform $\Phi$ maps the family of Schr\"odinger Hamiltonians \eqref{eq:schrodinger_ham}
to a family of Hamiltonians on $T^*\Dens(M)$ given by
\begin{equation}
	\tilde H(\rho,\theta) = H(\Phi(\rho,\theta)) = \frac{1}{2m}\int_M \abs{\nabla \theta}^2\rho\mu 
	+ \frac{\hbar^2}{8m}\int_M \frac{\abs{\nabla\rho}^2}{\rho}\vol + \int_M (V\rho + F(\rho))\vol ,
\end{equation}
at the density $\rho\mu\in \Dens(M)$.
In particular, if $m=1$ we recover Newton's equations \eqref{eq:ham_eq} on $\Dens(M)$ for the potential function
\begin{equation} 
\bar U(\rho) = \frac{\hbar^2 I(\rho)}{4}  + \int_M (V\rho + F(\rho) )\vol \,,
\end{equation} 
where $I$ is Fisher's information functional \eqref{eq:Fisher_info_func}. 
The extension \eqref{eq:reduced_newton_diff} to a fluid equation on $T\Diff(M)/\Diffvol(M)\simeq \Xcal(M)\times\Dens(M)$ is 
\begin{equation}\label{eq:barotropic2} 
\left\{ 
\begin{aligned} 
&\dot v + \nabla_v v + \nabla\Big(V + f(\rho) - \frac{\hbar^2}{2}\frac{\Delta\sqrt{\rho}}{\sqrt{\rho}} \Big) = 0 
\\ 
&\dot\rho +\divv(\rho v) = 0. 
\end{aligned} \right. 
\end{equation} 
\end{proposition} 
%
%
%
%

\begin{remark}
	For the linear Schrödinger equation \eqref{eq:schrodinger}, where $f \equiv 0$,
	notice that the ``classical limit'' immediately follows from \eqref{eq:barotropic2}: as $\hbar\to 0$ we recover classical mechanics and the Hamilton-Jacobi equation as presented in \autoref{sub:hamilton_jacobi}.
\end{remark}

%
\begin{remark} 
In this section we have seen how the Schrödinger equation can be expressed as a compressible fluid equation via Madelung's transform.
Conversely, the classical equations of hydrodynamics can be formulated as nonlinear Schr\"odinger equations (since the Madelung transform is a symplectomorphism, so any Hamiltonian on $T^*\Dens(M)$ induces a corresponding Hamiltonian on $P C^\infty(M,\CC)$). 
In particular, potential solutions of the compressible Euler equations of a barotropic fluid \eqref{eq:contin} 
can be expressed as solutions to an NLS equation with Hamiltonian 
\begin{equation} \label{eq:compressible_Euler_NLS_Hamiltonian} 
H(\psi) 
= 
\frac{\hbar^2}{2}\norm{\nabla\psi}_{L^2}^{2} 
- 
\frac{\hbar^2}{2}\norm{\nabla \abs{\psi}}_{L^2}^{2} 
+ 
\int_M e \big( \abs{\psi}^{2} \big) \abs{\psi}^{2} \vol 
\end{equation} 
where $e=e(\rho)$ is the specific internal energy of the fluid. 
The choice $e=0$ gives a Schr\"odinger-type formulation for potential solutions of Burgers' equation 
describing geodesics of the Wasserstein-Otto metric \eqref{eq:otto_metric} on $\Dens(M)$. 
We thus have a geometric framework that connects optimal transport for cost functions with potentials, 
the Euler equations of compressible hydrodynamics, and the NLS-type equations described above. 
\end{remark}

\begin{remark}
	Another relevant development is the Schrödinger Bridge problem, which seeks the most likely probability law for a diffusion process in the probability space, that matches marginals at two end-points in time, as we discuss in the next section: one can interpret it as a stochastic perturbation of Wasserstein-Otto geodesics on the density space for given end-points. 
	The Madelung transform allows one to translate questions about the Schr\"odinger equation to questions about probability laws, cf.\ \citet{Za1986}.
	More recently, the Madelung transform for quantum-classical hybrid systems has been studied by \citet{GaTr2020}.
\end{remark}

\subsection{The Madelung and  Hopf-Cole transforms} 

There is a real version of the complex Madelung transform.  
\begin{definition}\label{def:hopf-cole} 
Let $\rho$ and $\theta$ be real-valued functions on $M$ with $\rho >0$ 
and let $\gamma$ be a positive constant. 
The  (symmetrised) \emph{Hopf-Cole transform} is the mapping 
$HC\colon(\rho,\theta) \mapsto (\eta^+, \eta^-)\in C^\infty(M,\RR^2)$ 
defined by 
\begin{equation} \label{eq:HC_def} 
\eta^\pm =  \sqrt{\rho\, \ee^{\pm\theta/\gamma}}. 
\end{equation} 
\end{definition} 
In \cite{LeLi2019} it is shown that this map, along with its generalizations, has the property 
that its inverse $HC^{-1}$ takes the constant symplectic structure $d\eta^-\wedge d\eta^+$ on $C^\infty(M,\RR^2)$ 
to (a multiple of) the standard symplectic structure on $T^*\Dens(M)$. 
Note that the choice $\gamma=-\mathrm{i}\hbar/2$ corresponds to the standard Madelung transform 
\eqref{eq:madelung_def}: 
the function $\eta^+$ becomes a complex-valued wave function $\psi$ 
so that the symplectic properties of $HC$ can be viewed as an extension of those of the Madelung map $\Phi$.

Consider the (viscous) Burgers equation 
$$ 
\dot v+\nabla_v v = \gamma\Delta v. 
$$ 
The second component $\eta =  \sqrt{\ee^{-\theta/\gamma}}$ of \eqref{eq:HC_def} with $\rho =1$ 
maps the potential solutions $v=\nabla \theta$, which satisfy the Hamilton-Jacobi equation 
$$ 
\dot \theta + \frac 12 |\nabla \theta|^2 =\gamma\Delta \theta, 
$$ 
to the solutions of the heat equation $\dot \eta=\gamma\Delta\eta$. 

Similarly, the Hopf-Cole map can be used to transform certain barotropic-type systems 
to heat equations. 
This can be verified directly in the example of \autoref{sect:NLS}: 
setting Planck's constant to be $\hbar = \mp 2\mathrm{i}\gamma$ 
in the Schr\"odinger equation \eqref{eq:schrodinger} with $V \equiv 0$, $f\equiv 0$ and $m=1$ 
gives the forward and the backward heat equations 
$$ 
\dot \eta^\pm=\pm\gamma\Delta\eta^\pm. 
$$ 
The corresponding barotropic fluid system, which is readily obtained from \eqref{eq:barotropic2} 
with $\hbar = \pm 2\mathrm{i}\gamma$, reads 
\begin{equation} \label{eq:barotropic3} 
\left\{ 
\begin{aligned} 
&\dot v + \nabla_v v + 2\gamma^2\nabla\frac{\Delta\sqrt{\rho}}{\sqrt{\rho}} = 0 
\\ 
&\dot\rho +\divv(\rho v) = 0. 
\end{aligned} \right. 
\end{equation} 
This is again a Newton system on $\Dens(M)$ but in this case the potential function is corrected by 
the Fisher functional with the \emph{minus} sign 
(instead of the \emph{plus} sign as in \autoref{prop:schrodinger}). 
Equipped with the two-point boundary conditions $\rho|_{t=0} = \rho_0$ and $\rho|_{t=1}=\rho_1$ 
the horizontal solutions $v=\nabla\theta$ of \eqref{eq:barotropic3} 
correspond to the solutions of a dynamical formulation of the Schr\"odinger bridge problem, 
as surveyed by \citet{Le2014}. 
In this way one can study non-conservative systems with viscosity in a symplectic setting. 
It is interesting to incorporate the incompressible Navier-Stokes equations into this framework. 
This would require a two-component version of the map in \cite{LeLi2019} 
related to the two-component Madelung transform in the Schr\"odinger's smoke example below. 
\begin{remark} \label{rem:heat} 
Equation \eqref{eq:barotropic3} displays yet another relation to the heat flow connected to 
an invariant submanifold of  $T^*\Dens(M)$. Consider the submanifold 
\begin{equation*} 
\Gamma 
= 
\big\{ (\nabla\theta,\rho) \mid \rho 
= 
e^{-{\theta/\gamma}} \big\}. 
\end{equation*} 
A straightforward calculation shows that $\Gamma$ is an invariant submanifold for \eqref{eq:barotropic3} 
and the evolution on $\Gamma$ is given by a system of decoupled equations 
\begin{equation} 
\left\{ 
\begin{aligned} 
& \dot\rho = \gamma\Delta\rho 
\\ 
&\dot\theta + \abs{\nabla\theta}^2 = \gamma\Delta\theta.
\end{aligned} 
\right. 
\end{equation} 
Furthermore, since $\log\rho$ is the variational derivative of the entropy functional $S(\rho) = \int_M (\log\rho) \rho\mu$, it follows that 
the substitution $\theta = -\gamma\log\rho$ (or $\rho=e^{-{\theta/\gamma}}$) 
corresponds to the momentum in the direction of negative entropy. 
This is related to the observation in \cite{Ot2001} that the heat flow is the $L^2$-Wasserstein gradient flow 
of the entropy functional. 
\end{remark} 
%

\subsection{Example: Schr\"odinger's smoke} 

While the Madelung transform provides a link between quantum mechanics 
and {\it compressible} hydrodynamics, in this section we describe how {\it incompressible} hydrodynamics is related to 
the so-called incompressible Schr\"odinger equation. 
The approach described here was developed in computer graphics by \citet{ChKnPiScWe2016} 
to obtain a fast algorithm that could be used to visualize realistic smoke motion. 

It is clear that the standard Madelung transform is not adequate to describe incompressible hydrodynamics 
since the group of volume-preserving diffeomorphisms lies in the kernel of the Madelung projection 
(any trajectory along $\Diffvol(M)$ projects to the constant wave function $\psi = 1$). 
Instead, one has to consider the multi-component Madelung transform, cf. \cite{KhMiMo2019}. 
For simplicity, we use two components although one can easily extend the constructions below
to the case of several components. 

Consider the diagonal action of $\Diff(M)$ on $T^*\Dens(M)\times T^*\Dens(M)$ 
and the associated momentum map given by 
\begin{equation} \label{eq:diagonal_action_codens} 
J(\rho_1,\theta_1,\rho_2,\theta_2) 
= 
\rho_1\ud\theta_1 \otimes\mu + \rho_2\ud\theta_2 \otimes\mu. 
\end{equation} 
Fix two densities $\vol_1,\vol_2\in\Dens(M)$ and consider the group intersection $\Diff_{\vol_1}(M)\cap\Diff_{\vol_2}(M)$.
This intersection is itself a group, which can be thought of as the subsgroup of e.g. $\Diff_{\vol_1}(M)$ consisting of diffeomorphisms that also preserve the ratio function $\lambda:=\vol_2/\vol_1$ on $M$. 
As in \autoref{sec:semi_direct_reduction} we (formally) consider  the quotient 
$$
T^*\Diff(M)/(\Diff_{\vol_1}(M)\cap\Diff_{\vol_2}(M))\,.
$$ 
This quotient is Poisson (assuming that it is a manifold or considering it ``at a regular point") and it can be  
regarded as a Poisson submanifold of the dual $\mathfrak{s}^*$ of the semidirect product algebra 
$\mathfrak{s} = \Xcal(M)\ltimes C^\infty(M,\RR^2)$. 
(The same type of quotients appeared in the constructions related compressible fluids \autoref{sect:fully}
 and compressible MHD \autoref{sect:comprMHD}.)
 Given a Hamiltonian $\bar H(\rho_1,\rho_2,m)$ on $\mathfrak s^*$ the governing equations are 
\begin{equation} 
\begin{cases} 
\dot m 
+ 
\LieD_v m + J\big(\rho_1,\frac{\delta \bar H}{\delta \rho_1},\rho_2,\frac{\delta \bar H}{\delta \rho_2} \big) = 0
\\ 
\dot\rho_1 + \divv(\rho_1 v) =0 
\\ 
\dot\rho_2 + \divv(\rho_2 v) =0 
\end{cases} 
\end{equation} 
where $v = \frac{\delta \bar H}{\delta m}$. 
The zero-momentum symplectic reduction, corresponding to momenta of the form 
$m = J(\rho_1,\theta_1,\rho_2,\theta_2)$, 
yields a canonical system
\begin{equation} 
\begin{cases} 
\dot\rho_1 = \frac{\delta \tilde H}{\delta \theta_1} = -\divv(\rho_1 v)
\\ 
\dot\rho_2 = \frac{\delta \tilde H}{\delta \theta_2} = -\divv(\rho_2 v)
\\ 
\dot\theta_1 = - \frac{\delta \tilde H}{\delta \rho_1} 
\\ 
\dot\theta_2 = - \frac{\delta \tilde H}{\delta \rho_2} 
\end{cases} 
\end{equation} 
 on $T^*\Dens(M)\times T^*\Dens(M)$  for the Hamiltonian 
$$ 
\widetilde H(\rho_1,\rho_2,\theta_1,\theta_2) 
= 
\bar H \big( \rho_1,\rho_2,J(\rho_1,\theta_1,\rho_2,\theta_2) \big). 
$$ 

Next, we turn to the incompressible case. 
Imposing the holonomic constraint $\rho_1 + \rho_2 = 1$ for the equations 
on $T^*(\Dens(M)\times\Dens(M))$ leads to a constrained Hamiltonian system 
\begin{equation} \label{eq:smoke_eq_constrained} 
\begin{cases} 
\dot\rho_1 = \frac{\delta \tilde H}{\delta \theta_1} 
\\ 
\dot\rho_2 = \frac{\delta \tilde H}{\delta \theta_2} 
\\ 
\dot\theta_1 = - \frac{\delta \tilde H}{\delta \rho_1} - p 
\\ 
\dot\theta_2 = - \frac{\delta \tilde H}{\delta \rho_2} - p 
\\ 
\rho_1 + \rho_2 - 1 =0 
\end{cases} 
\end{equation} 
where $p\in C^\infty(M)$ is a Lagrange multiplier. 

The induced cotangent constraint on $(\theta_1,\theta_2)$ is obtained by 
\begin{equation} 
0 = \frac{\ud}{\ud t}(\rho_1+\rho_2-1) \mu
= 
\dot\rho_1 + \dot\rho_2 
= 
\frac{\delta \tilde H}{\delta\theta_1} + \frac{\delta \tilde H}{\delta\theta_1} 
= 
-\divv((\rho_1 + \rho_2)v)
= 
-\divv(v) \,,
\end{equation} 
which implies that the vector field $v$ is divergence-free. 
Therefore, solutions of \eqref{eq:smoke_eq_constrained} correspond to zero-momentum solutions 
of the incompressible fluid equations on $T^*\Diffvol(M) \simeq \Diffvol(M) \times \Xcal_{\vol}^*(M)$ 
with the Hamiltonian 
\begin{equation} 
H(\varphi,v^\flat\otimes\vol) = \bar H(\varphi_*\vol_1,\varphi_*\vol_2,v^\flat\otimes\vol). 
\end{equation} 
In particular, the choice 
\begin{equation} \label{eq:incompressible_classical_ham} 
\bar H(\rho_1, \rho_2,m) 
= 
\tfrac{1}{2} \langle m,v \rangle, 
\quad 
m = v^\flat \otimes (\rho_1 + \rho_2) \mu
\end{equation} 
yields special solutions to the incompressible Euler equations, see \autoref{sect:incompressible} 
(and, if the constraints are dropped, special solutions to the inviscid Burgers equation 
in \autoref{sub:burgers}). 

Schr\"odinger's smoke is an approximation to the zero-momentum incompressible Euler solutions, 
where the Hamiltonian $\tilde H$ corresponding to \eqref{eq:incompressible_classical_ham} 
is replaced by a sum of two independent Hamiltonian systems
\begin{equation} 
\tilde H(\rho_1,\rho_2,\theta_1,\theta_2) 
= 
\tfrac{1}{2}\pair{\rho_1\,\ud\theta_1\otimes\vol,\nabla\theta_1} 
+ 
\tfrac{1}{2}\pair{\rho_2\ud\theta_2\otimes\vol,\nabla\theta_2} + \hbar^2 I(\rho_1) + \hbar^2 I(\rho_2). 
\end{equation} 
This approximation corresponds to dropping the $\theta_1, \theta_2$ cross-terms in the original kinetic energy 
and adding the Fisher information functionals as potentials for $\rho_1$ and $\rho_2$. 
Applying the two-component Madelung transform 
$$ 
\Psi 
= (\psi_1 , \psi_2 )
:=\Big(  \sqrt{\rho_1\ee^{2\ii\theta_1/\hbar}}, \sqrt{\rho_2\ee^{2\ii\theta_2/\hbar}} \Big) 
$$ 
and setting $\hbar = 1$
gives the \emph{incompressible Schr\"odinger equation} 
\begin{equation} 
\ii\dot\Psi = -\Delta\Psi + p\Psi \,,
\end{equation} 
where, as before, the pressure function $p\in C^\infty(M)$ is a Lagrange multiplier 
for the pointwise constraint $\abs{\Psi}^2 = 1$. 
Notice that the resulting equation is a wave-map equation on $S^3 \subset \CC^2$, 
cf. e.g. \cite{Ta2004}.
\begin{remark} 
It has been claimed that numerical solutions to the incompressible Schr\"odinger equations (ISE) 
yield realistic visualization of the dynamics of smoke, see \cite{ChKnPiScWe2016}. 
However, it is an open question in what sense (or, in which regime) these solutions  
are approximations to solutions of the incompressible Euler equations. 
\end{remark} 
%



\subsection{Madelung transform as a K\"ahler morphism} 
\label{sub:kahler_properties_of_madelung} 

We now assume that both the cotangent bundle $T^{*}\Dens(M)$ and the projective space $P C^{\infty}(M,\CC)$ 
are equipped with suitable Riemannian structures. 
Consider first the bundle $TT^\ast\Dens(M)$. 
Its elements can be described as 4-tuples 
$(\rho,\theta,\dot\rho,\dot\theta)$ 
where 
$\rho\mu\in\Dens(M)$, $[\theta] \in C^{\infty}(M)/\RR$, $\dot\rho\mu \in \Omega^{n}_0(M)$ 
and $\dot\theta \in C^{\infty}(M)$ 
are subject to the constraint 
\begin{equation}
\int_M \dot\theta \,\rho\mu = 0. 
\end{equation} 
\begin{definition} 
The \emph{Sasaki} (or \emph{Sasaki-Fisher-Rao{\rm )} metric} on $T^\ast\Dens(M)$ 
is the cotangent lift of the Fisher-Rao metric \eqref{eq:fisher_rao_metric}, namely 
\begin{equation} \label{eq:sasaki_FR_metric} 
\MetF^*_{(\rho,[\theta])} \big( (\dot\rho,\dot\theta),(\dot\rho,\dot\theta) \big) 
= 
\int_{M} \bigg( \Big( \frac{\dot\rho}{\rho} \Big)^2 + \dot\theta^2 \bigg) \rho\mu\,. 
\end{equation} 
On the projective space $PC^\infty(M,\CC)$ we define the infinite-dimensional \textit{Fubini-Study metric} 
\begin{equation} \label{eq:fubini_study} 
\Met^*_\psi(\dot\psi,\dot\psi) 
= 
\frac{ \big\langle \dot\psi,\dot\psi \big\rangle_{L^2} }{ \pair{\psi,\psi}_{L^2} } 
- 
\frac{ \big\langle \psi,\dot\psi \big\rangle_{L^2} \big\langle \dot\psi,\psi \big\rangle_{L^2}}{\pair{\psi,\psi}_{L^2}^{2}}. 
\end{equation} 
\end{definition} 
\begin{theorem}[\cite{KhMiMo2019}] \label{thm:madelung_isometry} 
The Madelung transform \eqref{eq:madelung_symplectic} with $\hbar=2$ is an isometry, up to a factor $4$, between 
the spaces $T^\ast\Dens(M)$ equipped with \eqref{eq:sasaki_FR_metric} 
and $P C^{\infty}(M,\CC\backslash \{0\})$ equipped with \eqref{eq:fubini_study}. 
\end{theorem} 
Since the Fubini-Study metric together with the complex structure of $PC^\infty(M,\CC)$ 
defines a K\"ahler structure, it follows that $T^*\Dens(M)$ also admits a natural K\"ahler structure 
which corresponds to the canonical symplectic structure. 
Note that an almost complex structure on $T^*\Dens(M)$, which is related via the Madelung transform 
to the Wasserstein-Otto metric, does not integrate to a complex structure, cf. \cite{Mo2015b}. 
In fact, it was shown in \cite{KhMiMo2019} that the corresponding complex structure becomes integrable 
(and considerably simpler) when the Fisher-Rao metric is used in place of the Wasserstein-Otto metric. 
It would be interesting to write down the K\"ahler potentials for all metrics compatible 
with the corresponding complex structure on $T^*\Dens(M)$ and identify those that are invariant 
under the action of the diffeomorphism group. 
\begin{example} \label{ex:2HS} 
The 2-component Hunter-Saxton (2HS) equation is the following system 
\begin{equation} \label{eq:two_HS_eq} 
\left\{ 
\begin{array}{l} 
\dot u_{xx} = -2 u_x u_{xx} - u u_{xxx} + \sigma\sigma_x
\\ 
\dot\sigma = - (\sigma u)_x \,,
\end{array} \right. 
\end{equation} 
where $u$ and $\sigma$ are time-dependent periodic functions on the real line. 
It can be viewed as a high-frequency limit of the two-component Camassa-Holm equation, 
cf.\ \cite{HaWu2011}. 

It turns out that \eqref{eq:two_HS_eq} describes the geodesic flow of a right-invariant ${\dot H}^1$-type metric 
on the semidirect product $\mathcal{G} = \Diff_0(S^1)\ltimes C^{\infty}(S^1,S^1)$ 
of the group of circle diffeomorphisms that fix a prescribed point and the space of $S^1$-valued maps of a circle. 
Furthermore, there is an isometry between subsets of the group $\mathcal{G}$ and the unit sphere in the space of 
wave functions $\{ \psi \in C^{\infty}(S^1,\CC)~|~\norm{\psi}_{L^2} = 1 \}$, 
see \cite{Le2013b}. 
In \cite{KhMiMo2019} it is proved that the 2HS equation \eqref{eq:two_HS_eq} with initial data 
satisfying $\int_{S^1} \sigma \,\ud x = 0$ is equivalent to the geodesic equation of 
the Sasaki-Fisher-Rao metric \eqref{eq:sasaki_FR_metric} on $T^\ast\Dens(S^1)$ 
and the Madelung transformation induces a K\"ahler map to geodesics in $PC^{\infty}(S^1,\CC)$ 
equipped with the Fubini-Study metric. 
\end{example} 
Note also that (subject to the $t$-invariant condition $\sigma=0$) the 2-component Hunter-Saxton equation 
\eqref{eq:two_HS_eq} reduces to the standard Hunter-Saxton equation. 
This is a consequence of the fact that horizontal geodesics on $T^*\Dens(M)$ with the Sasaki-Fisher-Rao metric 
descend to geodesics on $\Dens(M)$ with the Fisher-Rao metric. 

\section{Casimirs in hydrodynamics} \label{sec:casimirs} 

In this section we start by surveying results on Casimirs for inviscid incompressible fluids, and then continue with compressible and magnetic hydrodynamics.
Recall that a Casimir on the dual of a Lie algebra $\mathfrak{g}^*$ is a function $f \in C^\infty(\mathfrak{g}^*)$ that is invariant under the coadjoint action of the corresponding group $G$.
Note that Casimirs are first integrals for Hamiltonian dynamics on $\mathfrak{g}^*$ for any choice of Hamiltonian functions.


\subsection{Casimirs for ideal fluids} \label{sec:ideal} 

The Hamiltonian description of the dynamics of an ideal fluid
gives some insight into the nature of its first integrals. 
Recall that the Euler equation is a Hamiltonian system on the dual space $\Xcal_\vol^*(M)$ 
with respect to the Poisson-Lie structure and with the fluid energy as the Hamiltonian, see  \autoref{sect:incompressible}. 
In this setting we have 
\begin{proposition}[\cite{OvKhCh1992, ArKh1998}] \label{I_h} 
For the group $\Diffvol(M)$ the following functionals are Casimirs on the dual space 
$\Xcal_\vol^*(M) = \Omega^1(M)/\ud C^\infty(M)$ 
(the space of cosets $[u] \in \Omega^1(M)/\ud C^\infty(M)$). 

If $\operatorname{dim}(M)=2m+1$, then the functional 
$$ 
I([u]) = \int_M u\wedge(\ud u)^m 
$$ 
is a Casimir function on $\Xcal_\vol^*(M)$. 
 
If $\operatorname{dim}(M)=2m$, then the functionals 
$$ 
I_h([u]) = \int_M h\left(\frac{(\ud u)^m}{\mu}\right)\,\mu 
$$ 
are Casimir functions on $\Xcal_\vol^*(M)$ for any measurable function $h:\RR \to \RR$. 
\end{proposition} 
Here, the quotient $(\ud u)^m/\mu$ of a $2m$-form and the volume form is a function, 
which being composed with $h$ can be integrated against the volume form $\mu$ over $M$. 
\begin{proof} 
First, we have to check that $I$ and $I_h$ are well-defined functionals on $\Omega^1(M)/\ud C^\infty(M)$. 
Note that for any exact $1$-form $\ud f$ we have $I(\ud f)=0$ and $I_h(\ud f)=0$. 
Similarly, we find that each of the functionals $I$ and $I_h$ depends on a coset but not on a representative, 
e.g., $I(u) = I(u+\ud f) = I([u])$. 
Furthermore, the group $\Diffvol(M)$ acts on $\Omega^1(M)/\ud C^\infty(M)$ 
by change of coordinates 
$ 
[u] \mapsto \varphi^*[u] 
$ 
for any $\varphi \in \Diffvol(M)$. 
Since both $I$ and $I_h$ are defined in a coordinate-free way, they are invariant under this action. 
\end{proof} 
\begin{corollary} 
For a velocity field $v$ satisfying the incompressible Euler equations in $M$ the 
functionals $I([u])$ and  $I_h([u])$ computed for the 1-forms $u:=v^\flat$ (related to $v$ by the Riemannian metric on $M$)
are first integrals 
in odd and even dimension,  respectively.
\end{corollary} 
\begin{proof} 
The Euler equations for $v$ in the Lie algebra $\Xcalvol(M)$
become Hamiltonian when rewritten for $u=v^\flat$ with respect to the standard Lie-Poisson bracket 
on the dual space $\Xcal_\vol^*(M)=\Omega^1(M)/\ud C^\infty(M)$. For this Hamiltonian system the trajectories
always remain  tangent to coadjoint orbits of $\Diffvol(M)$. 
By Proposition \ref{I_h}, the functions $I$ and $I_h$ are constant on coadjoint orbits 
and hence are constant along the Euler trajectories. 
\end{proof} 
\begin{remark} 
The functionals  $I$ and $I_h$ are Casimirs of the Lie-Poisson bracket on 
$\Xcal_\vol^*(M)$, 
i.e. they yield conservation laws for {\it any} Hamiltonian equation on this space. 
In particular, both $I$ and $I_h$ are first integrals of the Euler equations for an {\it arbitrary metric} on $M$. 
They express ``kinematic symmetries'' of the hydrodynamical system, 
while the energy is an invariant related to the system's ``dynamics.'' 
\end{remark} 
\begin{example} 
If $M$ is a domain in $\RR^3$ then   the function 
$$ 
 I(v) = \int_{M} u \wedge \ud u = \int_{M} ( v,{\rm curl} \, v) \, \ud^3x 
$$ 
is a first integral of the Euler equations, where the 1-form $u=v^\flat$ is related to the velocity field $v$ 
by means of the Euclidean metric. The last integral has a natural geometric meaning of 
the \emph{helicity} of the vector field $\xi={\rm curl}\,v $ defined by $\iota_\xi\vol=\ud u$. 
\end{example} 
\begin{example} 
Similarly, if $M$ is a domain in $\mathbb{R}^2$ we find infinitely many first integrals of the Euler equations, 
namely 
$$ 
I_h(v) = \int_{M}h({\rm curl}\, v) \, \ud^2x, 
$$ 
where 
${\rm curl}\, v = \partial v_1/\partial x_2 - \partial v_2/\partial x_1$ 
is the {\it vorticity function} on $M\subset\RR^2$. 
\end{example} 
\begin{remark} 
While the functions $I$ and $I_h$ on $\Omega^1(M)/\ud\Omega^0(M)$ are Casimirs, 
generally speaking, 
they do \emph{not} form a \emph{complete} set of invariants of the coadjoint representation. 

In the 2D case the complete set of invariants includes a measured Reeb graph of 
the vorticity function ${\rm curl}\,v$ and circulation data of the field $v$ on the surface $M$, 
see \cite{IzKhMo2016}. 
In the 3D case the invariant $I$ is shown to be unique among $C^1$-Casimirs \cite{EnPeLi2016}, 
while there are more invariants of ergodic nature (such as pairwise linkings of the trajectories of the vorticity field) 
that are not continuous functionals \cite{Ar2014}. 
\end{remark} 
%

 \subsection{Casimirs for barotropic fluids} 

In many respects the behaviour of barotropic compressible fluids is similar to that of incompressible fluids 
(while  the fully compressible fluids resemble thermodynamical rather than mechanical systems). 
In particular, their Hamiltonian description suggests similar sets of Casimir invariants of motion. 
While the incompressible Euler equations on a manifold $M$ are geodesic equations 
on the group $\Diffvol(M)$ and hence a Hamiltonian system on the corresponding dual space $\Xcal_\vol^*(M)$, 
the equations of compressible barotropic fluids \eqref{eq:contin} are known to be related 
to the semidirect product group $S=\Diff(M)\ltimes C^\infty(M)$, 
see Section \ref{sub:compressible_semidirect}. 
Its Lie algebra is $\mathfrak s=\Xcal(M)\ltimes C^\infty(M)$ 
and 
the corresponding dual space $\mathfrak s^*=\Xcal^*(M)  \oplus  \Omega^n(M)$ 
was described in \autoref{sub:poisson_reduction}. 
 
The equations of  barotropic fluids are Hamiltonian equations on $\mathfrak s^*$ 
with the Lie-Poisson bracket given by the formula \eqref{eq:poisson_bracket_reduced} 
and the invariants of the corresponding coadjoint action, i.e. the Casimir functions, 
are the first integrals of the equations of motion. 

Recall that the smooth part of the dual of  the semidirect product algebra 
$\mathfrak s=\Xcal(M)\ltimes C^\infty(M)$ can be identified with 
$\mathfrak s^*=\Omega^1(M)\otimes\Omega^n(M)\oplus \Omega^n(M)$ 
via the pairing 
$$ 
 \big\langle (v,f), (\alpha\otimes\varrho,\varrho) \big\rangle 
 = 
 \int_M (\iota_v\alpha)\varrho + \int_M f\varrho. 
$$ 
In what follows we restrict to the subset $\Omega_+^n(M)$ of $\Omega^n(M)$ 
corresponding to everywhere positive densities on $M$. 
It turns out that the   equations of incompressible fluid also have an infinite number of conservation laws 
in the even-dimensional case and possess at least one first integral in the odd-dimensional case,
see \autoref{sec:ideal} and  \cite{ArKh1998, OvKhCh1992}. 

The following proposition shows that Casimir functions for a barotropic fluid
are similar to the ones for an incompressible fluid. 
\begin{proposition}[\cite{HoMaRa1998, OvKhCh1992}] \label{prop:I_h} 
Let $\alpha \in \Omega^1(M)$ and $\varrho \in \Omega_+^n(M)$. 
If $\dim{M} = 2m+1$ then the functional 
$$ 
I(\alpha\otimes\varrho,\varrho) = \int_M \alpha\wedge(\ud\alpha)^m 
$$ 
is a Casimir function on $\mathfrak s^*=\Xcal^*(M) \oplus \Omega^n(M)$.

If $\dim{M} = 2m$ then for any measurable function $h\colon\RR\to\RR$ the functional 
$$ 
I_h(\alpha\otimes\varrho,\varrho) = \int_M h \bigg( \frac{(\ud\alpha)^m}{\varrho} \bigg) \varrho 
$$ 
is a Casimir function on $\mathfrak s^* = \Xcal^*(M) \oplus \Omega^n(M)$. 
\end{proposition} 
\begin{proof} 
The proof is based on the fact that the coadjoint action of the group $\Diff(M)\ltimes C^\infty(M)$ 
on the dual space $\mathfrak s^*=\Xcal^*(M) \oplus \Omega^n(M)$ 
is given by 
\begin{equation} \label{eq:diffcoad} 
\Ad^*_{(\varphi,f)^{-1}}(\alpha\otimes\varrho,\varrho) 
= 
\big( 
(\varphi^*\alpha + \varphi^*\ud f)\otimes\varphi^*(\varrho), \varphi^*(\varrho) 
\big). 
\end{equation} 
Thus, $\alpha$ and $\varrho$ transform according to the rules 
$\alpha \mapsto \varphi^*\alpha + \ud \varphi^* f$ and $\varrho\mapsto \varphi^*\varrho$ 
and it is now straightforward to check that the functionals $I$ and $I_h$ are invariant under such transformations. 
Indeed, up to the change of coordinates by a diffeomorphism $\varphi$, 
the 1-form $\alpha$ changes within its coset $[\alpha]$ and the functionals $I$ and $I_h$ 
are well defined on the cosets. 
\end{proof} 
The above argument shows that, in a certain sense, a barotropic fluid ``becomes incompressible''
when viewed in a coordinate system which ``moves with the flow.'' 
The Hamiltonian approach makes it possible to apply Casimir functions to study stability of 
barotropic fluids and inviscid MHD systems: 
their dynamics is confined to coadjoint orbits of the corresponding groups and Casimir functions 
can be used to describe the corresponding conditional extrema of the Hamiltonians. 

%

\subsection{Casimirs for magnetohydrodynamics} 

We start with the 3D incompressible magnetohydrodynamics described in \autoref{sub:mhd}, 
cf. equations \eqref{eq:mhd}. 
In this case the configuration space of a magnetic fluid is the semidirect product 
$\IMH = \Diffvol(M) \ltimes \mathfrak X_\vol^*(M)$ 
of the volume preserving diffeomorphism group and the dual space 
$\mathfrak X_\vol^*(M) = \Omega^1(M)/\ud\Omega^0(M)$ 
of the Lie algebra of divergence free vector fields on a $3$-manifold $M$. 
The semidirect product algebra is ${\mathfrak{imh}} = \mathfrak X_\vol(M)\ltimes\mathfrak X_\vol^*(M)$ 
and its action is given by formula \eqref{eq:diffad}. 
The corresponding dual space is 
$$ 
{\mathfrak{imh}}^* 
= 
\mathfrak X_\vol^*(M) \oplus \mathfrak X_\vol(M) 
= 
\Omega^1(M)/\ud\Omega^0(M) \oplus \mathfrak X_\vol(M) 
$$ 
and the Poisson brackets on ${\mathfrak{imh}}^*$ are given by \eqref{eq:poisson_bracket_reduced}, 
interpreted accordingly. 
\begin{proposition}[\cite{ArKh1998, HoMaRaWe1985}] \label{prop:I_mhd} 
Let $M$ be a manifold with $H_1(M)=0$, and 
let $[\alpha] \in \Omega^1(M)/\ud\Omega^0(M)$ and ${\mathbf B} \in \mathfrak X_\vol(M)$. 
Then the magnetic helicity 
$$ 
I({\mathbf B}) = \int_M ({\mathbf B},{\rm curl}^{-1}\,{\mathbf B})\,\mu 
$$ 
and the cross-helicity 
$$ 
J(\alpha,{\mathbf B}) = \int_M \iota_{\mathbf B}\alpha \,\mu 
$$ 
are Casimir functions on $\mathfrak{imh}^*$.
\end{proposition} 
The condition $H_1(M)=0$ ensures that any magnetic field $\mathbf B$ 
has a vector potential ${\rm curl}^{-1}\,{\mathbf B}$. 
It turns out that these are the only Casimirs for incompressible magnetohydrodynamics: 
any other sufficiently smooth Casimir is a function of these two, cf.~\cite{EnPeLi2016}.

Consider now the setting of compressible magnetohydrodynamics on a Riemannian manifold 
of arbitrary dimension, see \eqref{eq:mhd_poisson_reduced} above. Recall also 
from \autoref{sect:comprMHD} that the semidirect product group associated 
with the compressible MHD equations is 
$$ 
\CMH = \Diff(M) \ltimes \big( C^\infty(M) \oplus \Omega^{n-2}(M)/\ud\Omega^{n-3}(M) \big). 
$$ 
The corresponding Lie algebra is 
$$ 
{\mathfrak{cmh}} 
= 
\Xcal(M) \ltimes \big( C^\infty(M) \oplus \Omega^{n-2}(M)/\ud\Omega^{n-3}(M) \big) 
$$ 
with dual 
$$ 
{\mathfrak{cmh}}^* 
= 
\Xcal^*(M) \oplus \Omega^n(M) \oplus \Omega_{cl}^2(M), 
$$ 
where $\Omega_{cl}^2(M)$ is the space of closed 2-forms referred to as ``magnetic 2-forms.'' 
Recall that if $M$ is a three-fold then a magnetic vector field $\mathbf B$ 
and a magnetic 2-form $\beta\in \Omega_{cl}^2(M)$ 
are related by $\iota_{\mathbf B}\mu = \beta$. 
We again confine our constructions to positive densities $\Omega_+^n(M)$. 
\begin{proposition} \label{prop:I_chd} 
Let $\alpha \in \Xcal^*(M)$, $\varrho \in \Omega^n(M)$ and $\beta \in \Omega_{cl}^2(M)$. 
If $\dim{M} = 2n+1$ then the generalized cross-helicity functional 
\begin{equation} \label{eq:gen_cross_helicity} 
J(\alpha,\varrho,\beta) = \int_M \alpha \wedge \beta^{n} 
\end{equation} 
is a Casimir function on ${\mathfrak{cmh}}^*$.

If $\dim{M} = 2n+1$ 
and $H_2(M)=0$, so that $\ud\gamma=\beta$ for some $1$-form $\gamma$, 
then 
$$ 
I(\beta)=\int_M \gamma\wedge \beta^{n} 
$$ 
is a Casimir function on ${\mathfrak{cmh}}^*$. 

If $\dim{M} = 2n$ then for any measurable function $h\colon \RR\to\RR$ the functional 
$$ 
I_{h}(\rho, \beta) 
= 
\int_M h \bigg( \frac{\beta^{n}}{\varrho} \bigg) \varrho 
$$ 
is a Casimir function on ${\mathfrak{cmh}}^*$. 
\end{proposition} 
If $\mathbf{B}$ is a vector field on $M$ defined by $\iota_{\mathbf B}\varrho = \beta^n$ 
then the functional $J$ can be equivalently written as 
$$ 
J(\alpha,\varrho,\beta) 
= 
\int_M \alpha \wedge \iota_{\mathbf B}\varrho 
= 
\int_M  \iota_{\mathbf B} \alpha \, \varrho \,.
$$ 
In the three-dimensional  ($n=1$) and incompressible ($\rho=1$) case it reduces to 
the cross-helicity functional $J(\alpha,{\mathbf B})$ of \autoref{prop:I_mhd}. 
\begin{proof} 
The coadjoint action is 
\begin{equation} \label{eq:diffcoad2} 
\Ad^*_{(\varphi,f,[P])^{-1}}(\alpha\otimes\varrho,\varrho,\beta) 
= 
\big( 
( \varphi^*\alpha + \varphi^*\iota_u\beta +  \varphi^*\ud f ) \otimes \varphi^*\varrho, 
  \varphi^*\varrho, 
  \varphi^*\beta 
  \big) 
\end{equation} 
where the vector field $u$ is defined by the condition $\iota_u\varrho = \ud P$. 
Since both $\varrho$ and $\beta$ are transported by $\varphi$, the only non-trivial functional to check 
is the generalized cross-helicity $J$. 

For this purpose we first note that since $\beta$ is closed then so is $\beta^n$. 
Hence, the change of variables formula gives 
\begin{align} 
J(\varphi^*\alpha + \varphi^*\iota_u\beta +  \varphi^*\ud f, \varphi^*\varrho,\varphi^*\beta) 
&= 
\int_M (\alpha + \iota_u\beta + \ud f)\wedge \beta^{n}  
\\ 
&= 
J(\alpha,\varrho,\beta) + \int_M \iota_u\beta \wedge \beta^{n} + \int_M   \ud (f \beta^{n}) ,
\end{align} 
where the last term on the right-hand side vanishes by Stokes' theorem 
while the $(2n+1)$-form $\iota_u\beta \wedge \beta^{n}$ vanishes pointwise on $M$. 
The latter holds since 
evaluating this form on any $2n+1$ linearly independent vectors tangent to $M$ is equivalent to 
evaluating $\beta^{n+1}$ on any linearly dependent set of $2n+2$ tangent vectors containing $u$, 
which is evidently zero. 
\end{proof} 
\begin{remark} 
In 2D and 3D such Casimirs in different terms were described in \cite{HoMaRaWe1985}.
Other differential-geometric invariants of hydrodynamical equations 
include Ertel-type invariants \cite{We2018}, local invariants \cite{AnDa2009, AnWe2018}, 
invariants of Lagrangian type \cite{BeFr2017}, 
and many others. 
\end{remark} 
%

\appendix 

\section{Symplectic and Poisson reductions} 
\label{sub:symplectic_reduction} 

\subsection{Symplectic reduction} \label{sub:symplectic_first} 
In \autoref{sub:poisson_reduction} and \autoref{sub:symplectic?} 
we described Poisson reduction on $T^*\Diff(M)$ with respect to the cotangent action of $\Diffvol(M)$. 
This lead to reduced dynamics on the Poisson manifold 
$T^*\Diff(M)/\Diffvol(M) \simeq \Dens(M) \times \Xcal^*(M)$ 
(\autoref{thm:poisson_reduction}). 
Furthermore, any Hamiltonian system descends to symplectic leaves 
and $T^*\Dens(M)$ with the canonical symplectic structure is one of the symplectic leaves of 
$T^*\Diff(M)/\Diffvol(M)$. 
In this appendix we shall describe \emph{symplectic reduction} which leads to the same manifold $T^*\Dens(M)$ --- 
the symplectic quotient $T^*\Diff(M)\sslash\Diffvol(M)$ corresponding to the cotangent bundle $T^*\Dens(M)$
equipped with the canonical symplectic structure (for a more thorough treatment, see \cite{MaMiOrPeRa2007}).

As before, let 
$\Xcalvol(M) = \big\{ u \in \Xcal(M) \mid \LieD_u \vol = 0 \big\}$ 
be the Lie algebra of $\Diffvol(M)$. 
Recall that the dual space is naturally isomorphic to 
$\Xcal_\vol^*(M) = \Omega^1(M)/\ud C^\infty(M)$, 
see \autoref{thm:dual_inc}. 
\begin{lemma} \label{lem:momentum_map_diffvol} 
The (smooth) dual $\Xcalvol(M)^*$ can be identified with the quotient space 
\begin{equation} \label{eq:xcalvoldual} 
\Xcal^*(M)/(\ud C^\infty(M)\otimes \vol) 
= 
(\Omega^1(M)\otimes \Dens(M))/(\ud C^\infty(M)\otimes \vol), 
\end{equation} 
where $\otimes$ is taken over smooth functions on $M$. 
The cotangent left action of $\Diffvol(M)$ on $T^*\Diff(M)$ is Hamiltonian. 
The associated momentum map $J\colon T^*\Diff(M)\to \Xcalvol(M)^*$ 
is given by 
\begin{equation} \label{eq:momentum_map_diffvol} 
J(\varphi,m) = \varphi^* m + \ud C^\infty(M)\otimes \vol 
\end{equation} 
where $\varphi\in\Diff(M)$ and $m \in \Xcal(M) \simeq T_\varphi^*\Diff(M)$.
The momentum map is equivariant, i.e., 
\begin{equation} 
J\big(\eta\cdot(\varphi,m)\big) = \eta_* J(\varphi,m) 
\end{equation} 
for all $\eta\in\Diffvol(M)$. 
\end{lemma} 
\begin{proof} 
From the Hodge decomposition it follows that $\pair{m,v} = 0$ for all $v\in\Xcalvol(M)$ 
if and only if 
$m=\ud\theta\otimes\vol$ for some $\theta \in C^\infty(M)$. 
This proves \eqref{eq:xcalvoldual}. 

From the standard Lie-Poisson theory (see e.g. \cite{MaRa1999,MaMiOrPeRa2007})  we find that the momentum map for $\Diff(M)$ 
acting on $T^*\Diff(M)$ is given by $(\varphi,m) \mapsto \varphi^*m$. 
Since $\Diffvol(M)$ is a subgroup of $\Diff(M)$, it follows from \eqref{eq:xcalvoldual} 
that the momentum map must be \eqref{eq:momentum_map_diffvol}. 

Regarding the equivariance statement, we have 
\begin{equation} 
\begin{split} 
\eta_* J(\varphi,m) 
&= 
\eta_*\varphi^*m + \eta_*(\ud C^\infty(M)\otimes \vol) 
\\
&= 
(\varphi\circ\eta^{-1})^* m + \ud\eta_* C^\infty(M)\otimes \eta_*\vol 
\\ 
&= 
(\varphi\circ\eta^{-1})^* m + \ud C^\infty(M)\otimes \vol 
\\ 
&= 
J(\varphi\circ\eta^{-1},m) 
= 
J\big(\eta\cdot(\varphi,m)\big),
\end{split} 
\end{equation} 
as required.
\end{proof} 
\begin{lemma} \label{lem:zero_momentum_levelset} 
The zero momentum level set 
\begin{equation} \label{eq:zero_momentum_levelset} 
J^{-1}([0]) 
= 
\big\{ (\varphi,\ud\theta\otimes \varphi_*\vol) \mid \varphi\in\Diff(M), \theta\in C^{\infty}(M) \big\} 
\end{equation} 
is invariant under the action of $\Diffvol(M)$, i.e., 
for any $\eta\in\Diffvol(M)$ and $(\varphi,m)\in J^{-1}([0])$ one has $\eta\cdot (\varphi,m) \in J^{-1}([0])$. 
\end{lemma} 
\begin{proof} 
We have $[0] = \ud C^{\infty}(M)\otimes\vol$ so that 
if $m = \ud \theta \otimes \varphi_*\vol$ then 
\begin{equation} 
\begin{split} 
J(\varphi,m) 
&= 
\varphi^*(\ud \theta \otimes \varphi_*\vol) + \ud C^{\infty}(M)\otimes\vol 
\\ 
&= 
\ud \varphi^*\theta \otimes\vol + \ud C^{\infty}(M)\otimes\vol 
=
\ud C^{\infty}(M)\otimes\vol 
= [0]. 
\end{split} 
\end{equation} 
Next, assume that $J(\varphi,m) = [0]$ and write 
$m = \alpha\otimes\varphi_*\vol$ 
for some $\alpha\in \Omega^1(M)$. 
Since $\varphi^*m \in [0]$ it follows that $\varphi^*\alpha$ must be exact, 
i.e., $\varphi^*\alpha= \ud\theta$. 
Thus, 
$\alpha = \varphi_*\ud\theta = \ud\varphi_*\theta$, 
and so $\alpha$ is exact. 
The fact that $J^{-1}([0])$ is invariant under $\Diffvol(M)$ follows from the equivariance property 
in \autoref{lem:momentum_map_diffvol}, since $\eta_*[0] = [0]$ for all $\eta\in\Diffvol(M)$. 
This concludes the proof. 
\end{proof} 
To identify the symplectic structure of the quotient we shall first identify the momentum map 
associated with the action of $\Diff(M)$ on $T^*\Dens(M)$. In what follows we will use the notation $\varrho\in \Dens(M)$ 
for the density $\varrho=\rho\mu$ corresponding to the density function $\rho$.

\begin{lemma} \label{lem:momentum_map_on_dens2} 
The associated momentum map 
$I \colon T^*\Dens(M) \to \Xcal^*(M)$ 
for the left cotangent action of $\Diff(M)$ on $T^*\Dens(M)$ is given by 
\begin{equation} 
I(\varrho,\theta) = \ud\theta\otimes\varrho\,. 
\end{equation} 
\end{lemma} 
\begin{proof} 
The smooth dual of $\Omega^n(M)$ is $C^\infty(M)$ with the natural pairing 
\begin{equation} 
\pair{\theta,\dot\varrho} = \int_M \theta\,\dot\varrho\,. 
\end{equation} 
Since $T_\varrho \Dens(M) = \Omega_0^n(M)$ is a subspace of $\Omega^n(M)$, 
it follows that 
\begin{equation} 
\begin{split} 
T^*\Dens(M) &= \Dens(M)\times\Omega^n(M)^*/\ker(\pair{\,\cdot\,,\Omega^n_0(M)}) 
\\ 
&= 
\Dens(M)\times C^\infty(M)/\RR. 
\end{split} 
\end{equation} 
The infinitesimal left action of $\Xcal(M)$ is $u\cdot\varrho = -\LieD_u\varrho$ 
and the momentum map $I\colon (\varrho,\theta)\to \Xcal^*(M)$ is then given by 
\begin{equation} 
\pair{I(\varrho,\theta),u} = \pair{\theta,-\LieD_u\varrho} 
\qquad 
{\rm ~~for ~all~~}\, u\in\Xcal(M). 
\end{equation} 
By Cartan's formula we obtain 
\begin{equation} 
\pair{I(\varrho,\theta), u} 
= 
\pair{\LieD_u\theta, \varrho} 
= 
\pair{\interior_u \ud\theta, \varrho} 
= 
\pair{\ud\theta\otimes\varrho, u},
\end{equation} 
which proves the lemma. 
\end{proof} 

The main result of this section is 
\begin{theorem} \label{thm:symplectic_reduction} 
The zero momentum symplectic quotient 
\begin{equation} 
T^*\Diff(M) \sslash \Diffvol(M) = J^{-1}([0])/\Diffvol(M) 
\end{equation} 
is isomorphic, as a symplectic manifold, to $T^*\Dens(M)$ 
and 
the symplectomorphism $T^*\Dens(M)\to T^*\Diff(M)\sslash\Diffvol(M)$ 
is given by 
\begin{equation} \label{eq:symplectomorphism_reduction} 
(\varrho,\theta) \mapsto \Big(\varrho, I(\varrho,\theta)\Big). 
\end{equation} 
\end{theorem} 
Thus $T^*\Dens(M)$ can be viewed as a symplectic leaf of the Poisson manifold $T^*\Diff(M)/\Diffvol(M)$. 
\autoref{thm:symplectic_reduction} is an infinite-dimensional variant of 
the following general result:  For a homogeneous space $ B=G/H $  the zero momentum 
reduction space $T^*G\sslash H$  is symplectomorphic to $T^*B$ through the mapping
		\begin{equation}
			(q,p) \mapsto \left( q,I(q,p) \right)\,,
		\end{equation}
		where $I$ is the momentum map for the natural action of $G$ on $T^*B\ni  (q,p)$, see \citet{MaRa1999}

\begin{proof} 
Let $m=\ud\theta\otimes\varphi_*\vol$ so that $(\varphi,m) \in J^{-1}([0])$. 
If $\eta \in \Diffvol(M)$ then 
\begin{equation} 
\eta\cdot (\varphi,m) = (\varphi\circ\eta^{-1},\ud\theta\otimes\varphi_*\vol). 
\end{equation} 
By Moser-Hamilton's result in \autoref{lem:moser} it  follows that in the Fr\'echet category we have
%
\begin{equation} 
J^{-1}([0])/\Diffvol(M) 
\simeq 
\big\{ (\varrho,m) \in \Dens(M)\times \Xcal^*(M) \mid m = \ud\theta\otimes\varrho \big\}. 
\end{equation} 
Thus, the symplectic quotient $T^*\Diff(M)\sslash\Diffvol(M)$ is naturally identified 
with a subbundle of the Poisson manifold $T^*\Diff(M)/\Diffvol(M) \simeq \Dens(M)\times\Xcal^*(M)$ 
in \autoref{thm:poisson_reduction}. 
By conservation of momentum this subbundle is invariant under the flow of any Hamiltonian. 
To prove that it is a symplectic leaf it suffices to show that the map 
corresponding to \eqref{eq:symplectomorphism_reduction} 
\begin{equation} 
\Phi\colon(\varrho,\theta) \mapsto (\varrho,I(\varrho,\theta)) 
\end{equation} 
is a diffeomorphism and Poisson. 
The former follows from the fact that the kernel of $\ud$ on $C^\infty(M)/\RR$ is trivial. 
It thus remains to show that 
\begin{equation} 
\{F\circ \Phi,G\circ\Phi \} = \{ F,G\}\circ\Phi 
\end{equation} 
for any $F,G\in C^\infty(\Dens(M)\times\Xcal^*(M))$. 

We have 
\begin{align} \nonumber 
\Big\langle \frac{\delta F\circ \Phi}{\delta\varrho}(\varrho,\theta),\dot{\varrho} \Big\rangle 
&= 
\frac{\ud}{\ud\epsilon}\Big|_{\epsilon =0 } F(\varrho 
+ 
\epsilon\dot{\varrho},\ud\theta\otimes (\varrho+\epsilon \dot{\varrho})) 
\\ \label{eq:dPhi_drho} 
&= 
\pair{\ud\theta\otimes\dot{\varrho} , \underbrace{\frac{\delta F}{\delta m}}_{v_F}} 
+ 
\Big\langle \dot{\varrho},\frac{\delta F}{\delta \varrho} \Big\rangle 
= 
\Big\langle \dot{\varrho},\LieD_{v_{F}}\theta+ \frac{\delta F}{\delta\varrho} \Big\rangle 
\end{align} 
and 
\begin{equation} \label{eq:dPhi_dtheta} 
\begin{split} 
\Big\langle \frac{\delta F\circ \Phi}{\delta\theta}(\varrho,\theta),\dot{\theta} \Big\rangle 
&= 
\frac{\ud}{\ud\epsilon}\Big|_{\epsilon =0 } F(\varrho,\ud(\theta+\epsilon\dot\theta)\otimes \varrho) 
\\ 
&= 
\Big\langle \ud\dot\theta\otimes\varrho,v_F \Big\rangle 
= 
-\Big\langle \LieD_{v_F}\varrho,\dot\theta \Big\rangle. 
\end{split} 
\end{equation} 
Combining \eqref{eq:dPhi_drho} and \eqref{eq:dPhi_dtheta} we get 
\begin{align*} \label{eq:Phi_poisson_calculations} 
\poisson{F\circ\Phi,G\circ\Phi}(\varrho,\theta) 
&= 
\Big\langle -\LieD_{v_G}\varrho,\LieD_{v_F}\theta+\frac{\delta F}{\delta\varrho} \Big\rangle 
- 
\Big\langle -\LieD_{v_F}\varrho,\LieD_{v_G}\theta+\frac{\delta G}{\delta\varrho} \Big\rangle 
\\ 
&= 
\Big\langle \varrho, (\LieD_{v_F}\LieD_{v_G}-\LieD_{v_G}\LieD_{v_F})\theta \Big\rangle 
+ 
\Big\langle \varrho,\LieD_{v_G}\frac{\delta F}{\delta\varrho}- \LieD_{v_F}\frac{\delta G}{\delta\varrho} \Big\rangle 
\\ 
&= 
\Big\langle \varrho,\interior_{\LieD_{v_F}v_G}\ud\theta \Big\rangle 
- 
\Big\langle \varrho,\LieD_{v_F}\frac{\delta G}{\delta\varrho}- \LieD_{v_G}\frac{\delta F}{\delta\varrho} \Big\rangle 
\\ 
&= 
\Big\langle \ud\theta\otimes\varrho,\LieD_{v_F}v_G \Big\rangle 
- 
\Big\langle \varrho,\LieD_{v_F}\frac{\delta G}{\delta\varrho}- \LieD_{v_G}\frac{\delta F}{\delta\varrho} \Big\rangle 
= 
\poisson{F,G}\circ \Phi(\varrho,\theta). 
\end{align*} 
This concludes the proof. 
\end{proof} 
\subsection{Reduction and  momentum map for semidirect product groups} 
\label{sec:semi_direct_reduction} 
We exhibit here geometric structures behind the semidirect product reduction 
generalizing the considerations of Sections \ref{sub:compressible_semidirect}-\ref{sub:mhd}. 
The main point of this appendix is that the semidirect product approach is just a convenient way of 
presenting various Newton's systems on $\Diff(M)$ for which the symmetry group 
is a proper subset of $\Diff(M)$: this way various quotient spaces appear as invariant sets in the vector space
which is the dual of an appropriate  Lie algebra.
\smallskip

\newcommand{\algG}{\mathfrak{n}}

Let $\subG$ be a subgroup of $\Diff(M)$. 
Suppose that $\Diff(M)$ acts from the left on a linear space $V$ (a left representation of $\Diff(M)$). 
For instance, for compressible fluids in  \autoref{sect:fully}
 and compressible MHD in \autoref{sect:comprMHD} the space $V$ was taken to be the spaces of functions $C^\infty(M)$ 
or the dual of the space of divergence-free vector fields $\Omega^1(M)/\ud C^\infty(M)$, while 
 $N$ can be a subgroup of volume-preserving diffeomorphisms $\Diffvol(M)$. However the consideration below is more general.

The quotient space of left cosets $\Diff(M)/\subG$ is acted upon from the left by $\Diff(M)$. 
Assume now that the quotient $\Diff(M)/\subG$ is a manifold and it can be embedded as an orbit in $V$, while  
$\gamma\colon \Diff(M)/\subG \to V$ denotes the embedding. 
Since the action of $\Diff(M)$ on $V$ induces a linear left dual action on $V^*$ 
we can construct the semidirect product $S=\Diff(M)\ltimes V^*$. Let $\mathfrak s^*$ be the dual of the 
 corresponding semidirect product algebra $\mathfrak s$.
\begin{proposition} \label{prop:semidirect_embedding}
The quotient $T^*\Diff(M)/\subG$ is naturally embedded via a Poisson map in the Lie-Poisson space 
$\mathfrak{s}^*$. 
\end{proposition} 
\begin{proof} 
The Poisson embedding is given by 
\begin{equation} \label{eq:semi_direct_map2} 
([\varphi],m) \mapsto (m,\gamma([\varphi])) 
\end{equation} 
where we used the identifications 
$$ 
T^*\Diff(M)/\subG \simeq \Diff(M)/\subG \times \mathfrak{g}^*=  \Diff(M)/\subG \times (\Omega^1\otimes \Dens(M)) 
$$ 
and 
$\mathfrak{s}^* \simeq \mathfrak{g}^*\times V= (\Omega^1\otimes \Dens(M))\times V$. 
Recall that 
the Lie algebra of $\Diff(M)$ is the space $\mathfrak{X}(M)$ of vector fields on $M$ 
whose dual is 
$\Xcal^*(M) = \Omega^1\otimes \Dens(M)$. 
The action of $S$ on $\mathfrak{s}^*$ is given by 
\begin{equation} 
(\varphi, a) \cdot (m,b) 
= 
\mathrm{Ad}^*_{(\varphi, a)}(m,b) 
= 
\big( \varphi^*m - {\mathscr M}(a,b), \varphi^* b \big), 
\end{equation} 
where $\varphi \in \Diff(M)$ and 
$ {\mathscr M}\colon V^*\times V \to \Xcal^*(M)$ is the momentum map 
associated with the cotangent lifted action of $\Diff(M)$ on $V^*$. 
The corresponding infinitesimal action of $\mathfrak{s}$ is 
\begin{equation} 
\label{eq:inf_ham_action2} 
(v, \dot a)\cdot (m,b) 
= 
\mathrm{ad}^*_{(v, \dot a)}(m,b) 
= 
\big(  \LieD_v m - {\mathscr M}(\dot a,b),  \LieD_v b \big). 
\end{equation} 
Since the second component is only acted upon by $\varphi$ (or $v$) but not $a$ (or $\dot a$), 
it follows from the embedding of $\Diff(M)/\subG$ as an orbit in $V$ that we have a natural Poisson action 
of $S$ (or $\mathfrak{s}$) on $T^*\Diff(M)/\subG$ via the Poisson embedding \eqref{eq:semi_direct_map2}. 
Notice that the momentum map of $S$ (or $\mathfrak{s}$) acting on $\mathfrak{s}^*$ is tautological, 
i.e. the identity: 
this follows from the fact that the Hamiltonian vector field on $\mathfrak{s}^*$ 
for $H(m,b) = \pair{m, v} + \pair{b,\dot a}$ is given by~\eqref{eq:inf_ham_action2}. 
\end{proof} 

We now return to the standard symplectic reduction (without semidirect products). 
The dual $\algG^*$ of the subalgebra $\algG\subset \Xcal(M)$ is naturally identified with 
the affine cosets of $\Xcal^*(M)$ such that 
\begin{equation} 
m \in [m_0] \iff \pair{m-m_0, v} = 0 
\quad 
\text{for any} 
\;\; 
v \in \algG. 
\end{equation} 
The momentum map of the subgroup $\subG$ acting on $\Xcal^*(M)$ by $\varphi^*$ 
is then given by 
$m \mapsto [m]$, 
since the momentum map of $\Diff(M)$ acting on $\Xcal^*(M)$ is the identity. 
If $\pair{m,\algG} = 0$, i.e., $m\in (\Xcal(M)/\algG)^*$, then $m\in [0]$ is in the zero momentum coset. 
Since we also have $T^*(\Diff(M)/\subG) \simeq \Diff(M)/\subG\times (\Xcal(M)/\algG)^*$ 
this gives us an embedding as a symplectic leaf in $T^*\Diff(M)/\subG\simeq \Diff(M)/\subG\times \Xcal^*(M)$. 
The restriction to this leaf is called the \emph{zero-momentum symplectic reduction}. 

Turning next to the semidirect product reduction, we now have Poisson embeddings of 
$T^*(\Diff(M)/\subG)$ in $T^*\Diff(M)/\subG$ and of $T^*\Diff(M)/\subG$ in $\mathfrak{s}^*$. 
The combined embedding of $T^*(\Diff(M)/\subG)$ as a symplectic leaf in $\mathfrak{s}^*$ 
is given by the map 
\begin{equation} \label{eq:embed2} 
([\varphi], a) \mapsto ( {\mathscr M}(a,\gamma([\varphi])),\gamma([\varphi]))\,. 
\end{equation} 
This implies that we have a Hamiltonian action of $S$ (or $\mathfrak{s}$) 
on the zero-momentum symplectic leaf $T^*(\Diff(M)/\subG)$ inside $T^*\Diff(M)/\subG$, 
which in turn lies inside $\mathfrak{s}^*$. 

Since $S$ provides a natural symplectic action on $\mathfrak{s}^*$ and since $\Diff(M)/\subG$ 
is an orbit in $V \simeq V^{**}$ we have, by restriction, a natural action of $S$ on $T^*\Diff(M)/\subG$. 
Furthermore, since the momentum map associated with the group $S$ acting on $\mathfrak{s}^*$ is the identity, 
the Poisson embedding map \eqref{eq:semi_direct_map2} is the momentum map for $S$ 
acting on $T^*\Diff(M)/\subG$. 
Thus, the momentum map of $S$ acting  on $T^*(\Diff(M)/\subG)$ is given by~\eqref{eq:embed2}. 

The above consideration leads to the Madelung transform. 
\begin{theorem}[\cite{KhMiMo2019}] 
Semidirect product reduction and Poisson embedding $T^*(\Diff(M)/\subG) \to \mathfrak{s}^*$ 
for the subgroup $\subG= \Diffvol(M)$ 
coincides with the inverse of the Madelung transform defined in \autoref{sec:madelung}. 
\end{theorem} 
%


\section{Tame Fr\'echet manifolds} \label{sect:tame} 

A natural functional-analytic setting for the results presented in this paper is that of tame Fr\'echet spaces, 
cf. Hamilton \cite{Ha1982}. An alternative setting for groups of diffeomorphisms deals with Sobolev $H^s$ completions 
(or any reasonably strong Banach topology) 
of the corresponding function spaces \cite{EbMa1970}. 
If $s>\dim{M}/2 +1$ then the Sobolev completions of the diffeomorphism groups 
$\Diff^s(M)$ and $\mathrm{Diff}_\mu^s(M)$ are smooth Hilbert manifolds but not Banach Lie groups since, 
e.g., the left multiplication and the inversion maps are not even uniformly continuous in the $H^s$ topology. 

\subsection{Tame Fr\'echet structures on diffeomorphism groups}\label{sect:tame_groups}
On the other hand, both $\Diff(M)$ and $\Diffvol(M)$ can be equipped with the structure of 
tame Fr\'echet Lie groups. In this setting $\Diff_\mu(M)$ becomes a closed tame Lie subgroup 
of $\Diff(M)$ which can be viewed as a tame principal bundle over the quotient space 
$\Dens(M) = \Diff(M)/\Diffvol(M)$ 
of either left or right cosets. 
Furthermore, the tangent bundle $T\Diff(M)$ over $\Diff(M)$ is also a tame manifold. 
However, since the dual of a Fr\'echet space, which itself is not a Banach space, is never a Fr\'echet space, 
to avoid working with currents on $M$ it is expedient to restrict to a suitable subset of the (full) cotangent bundle 
over $\Diff(M)$. 

More precisely, consider the tensor product $T^\ast M \otimes \Lambda^n M$ of the cotangent bundle and 
the vector bundle of $n$-forms on $M$ and define another bundle over $\Diff(M)$ 
whose fibre over $\varphi \in \Diff(M)$ is the space of smooth sections of the pullback bundle 
$\varphi^{-1}(T^\ast M \otimes \Lambda^n M)$ over $M$. 
We will refer to this object as (the smooth part of) the cotangent bundle of $\Diff(M)$ 
and denote it also by $T^\ast \Diff(M)$. 
We will write $\Xcal^*(M) = T^*_\id\Diff(M)$ and $\Xcal^{**}(M) = \Xcal(M)$.
Throughout the paper we will assume that derivatives of various Hamiltonian functions 
can be viewed as maps to the smooth cotangent bundle of the phase space.

\begin{lemma} \label{lem:cotangent_tame}
$T^*\Diff(M)$ is a tame Fr\'echet manifold and the map 
$$ 
\Diff(M) \times \Xcal^*(M) \ni (\varphi,m) \longmapsto (\varphi,m\circ\varphi)\in T^*\Diff(M) 
$$ 
is an isomorphism of tame Fr\'echet manifolds. 
\end{lemma} 
\begin{proof} 
Recall that $\Diff(M)$ is an open subset of $C^\infty(M,M)$ and observe that $T^\ast \Diff(M)$ 
is the inverse image of $\Diff(M)$ under the smooth tame projection 
$m \to \pi  \circ m$
between tame Fr\'echet manifolds $C^\infty(M,T^\ast M \otimes \Lambda^n M)$ and $C^\infty(M,M)$. 
The argument is routine: the space  $T^*\Diff(M)$ is trivialized by the fiber mapping since $\varphi$ is a diffeomorphism, while 
the fact that the fiber mapping is smooth and tame with a smooth tame inverse 
follows since $\Diff(M)$ is a tame Fr\'echet Lie group (all group operations are smooth tame maps). 
\end{proof} 
Let $v_\varphi \in T_\varphi \Diff(M)$ and $m_\varphi \in T^\ast_\varphi \Diff(M)$. 
As before in \eqref{eq:pairing_diff} we have the pairing 
%
\begin{equation*}  \label{eq:pairing} 
( v_\varphi, m_\varphi ) \mapsto \langle v_\varphi, m_\varphi \rangle_\varphi 
= 
\int_M \iota_{v_\varphi\circ\varphi^{-1}} m_\varphi\circ\varphi^{-1} 
\end{equation*} 
%
between the fibers $T_\varphi \Diff(M)$ and $T^\ast_\varphi \Diff(M)$. 

Our goal in this section is to describe Poisson reduction of $T^*\Diff(M)$ 
with respect to the right action of $\Diffvol(M)$ as a smooth tame principal bundle. 
We will use the Poisson bivector for the canonical symplectic structure on $T^*\Diff(M)$, which we identify 
with its right trivialization $\Diff(M)\times \Xcal^*(M)$ as in \autoref{lem:cotangent_tame}. 
By construction, each element of $\mathfrak{X}^\ast(M)$ can be viewed as 
a tensor product $m = \alpha \otimes \varrho$ of a 1-form and a volume form on $M$. 
Choose $\varrho = \vol$ 
and note that 
for each $(\varphi,m)\in \Diff(M)\times\Xcal^*(M)$ the Poisson bivector $\Lambda$ 
on $\Diff(M)\times\Xcal^*(M)$ is a bilinear form on $T^*_\varphi\Diff(M)\times\Xcal(M)$ 
defined by
\begin{align*} 
\Lambda_{(\varphi,m)} \big( (n_{1\varphi}, v_1),(n_{2\varphi}, v_2) \big) 
=& 
\langle m,[v_1,v_2] \rangle_\id 
\\ \nonumber 
&- 
\pair{n_{1\varphi},v_2\circ\varphi}_{\varphi} 
+ 
\pair{n_{2\varphi},v_1\circ\varphi}_{\varphi} .
\end{align*} 
%
\begin{lemma} \label{lem:poisson_isomorphism} 
The bivector $\Lambda$ induces a smooth tame vector bundle isomorphism 
\begin{equation*} 
\Gamma \colon T^\ast ( \Diff(M){\times}\Xcal^\ast (M) ) \to T ( \Diff(M){\times}\Xcal^\ast (M) ) 
\end{equation*} 
which at any point $(\varphi, m)$ is given by 
\begin{equation} \label{eq:GGamma}
\Gamma_{(\varphi,m)}(n_\varphi,v) 
= 
\big( v\circ\varphi, - \LieD_v m -n_\varphi\circ\varphi^{-1} \big) 
\end{equation} 
for any $n_\varphi \in T^\ast_\varphi \Diff(M)$ and $v \in \mathfrak{X}(M)$. 
\end{lemma} 
\begin{proof} 
First, observe that one can identify the tangent and cotangent bundles of 
$\Diff(M)\times\mathfrak{X}^\ast(M)$ 
with 
$T\Diff(M)\times\mathfrak{X}^\ast(M)\times\mathfrak{X}^\ast(M)$ 
and 
$T^\ast\Diff(M)\times\mathfrak{X}^\ast(M)\times\mathfrak{X}(M)$ 
respectively. 
The formula in \eqref{eq:GGamma} can be verified by a direct calculation from 
\begin{equation} 
\big\langle \Gamma_{(\varphi,m)} ( m_{1\varphi}, v_1 ), ( m_{2\varphi}, v_2 ) \big\rangle_{(\varphi,m)} 
= 
\Lambda_{(\varphi,m)} \big( (n_{1\varphi}, v_1), (n_{2\varphi}, v_2) \big) 
\end{equation} 
for any $n_{1\varphi}, n_{2\varphi} \in T^\ast_\varphi \Diff(M)$ and $v_1,v_2 \in \mathfrak{X}(M)$ 
using integration by parts and the assumption that $M$ has no boundary.
Smoothness of $\Gamma$ follows from the fact that all the operations in \eqref{eq:Gamma} 
are smooth tame maps. 
The inverse of $\Gamma$ is given by
\begin{equation}
	\Gamma_{(\varphi,m)}^{-1}(\dot\varphi,\dot m) = \big( \dot\varphi\circ\varphi^{-1}, -(\dot m + \LieD_{\dot\varphi\circ\varphi^{-1}}m )\circ\varphi  \big).
\end{equation}
Again, all the operations are smooth tame maps, which concludes smoothness also of the inverse.
\end{proof} 
\begin{remark} 
That $\Gamma$ is a symplectomorphism corresponds to the fact that 
$T^*\Diff(M)\simeq \Diff(M)\times\Xcal^*(M)$ is a symplectic manifold with canonical symplectic structure 
$\Omega_{(\varphi,m)}(\cdot,\cdot)= \big\langle \Gamma_{(\varphi,m)}^{-1}(\cdot,\cdot) \big\rangle$. 
The space $T^*\Dens(M) = \Dens(M)\times C^\infty(M)/\RR$ is a tame Fr\'echet manifold,
since so are both $\Dens(M)$ and $C^\infty(M)/\RR$. 
\end{remark} 

Next, consider the Poisson bivector $\bar\Lambda$ defined on the tame Fr\'echet manifold 
$\Dens(M)\times\Xcal^*(M)$ by 
\begin{align*} 
\bar\Lambda_{(\varrho,m)} \big( (\theta_{1}, v_1),(\theta_{2}, v_2) \big) 
=& 
\langle m,[v_1,v_2] \rangle_\id 
\\ \nonumber 
&+ 
\pair{\theta_{1},\LieD_{v_2}\varrho} 
- 
\pair{\theta_{2},\LieD_{v_1}\varrho} 
\end{align*} 
for the density $\varrho=\rho\mu$.
\begin{lemma} \label{lem:poisson_isomorphism2} 
The bivector $\bar\Lambda$ induces a smooth tame vector bundle homomorphism 
\begin{equation*} 
\bar\Gamma \colon T^\ast ( \Dens(M){\times}\Xcal^\ast (M) ) \to T ( \Dens(M){\times}\Xcal^\ast (M) ) 
\end{equation*} 
which at any point $(\varrho, m)$ is given by 
\begin{equation} \label{eq:Gamma} 
\bar\Gamma_{(\varrho,m)}(\theta,v) 
= 
\big( -\LieD_v\varrho, - \LieD_v m -\ud \theta\otimes \varrho  \big) 
\end{equation} 
for any $\theta \in T^\ast_\varrho \Dens(M)$ and $v \in \mathfrak{X}(M)$. 
\end{lemma} 
\begin{proof} 
The proof follows the same steps as the proof of \autoref{lem:poisson_isomorphism} 
with the adjustment that now $\bar\Gamma_{(\varrho,m)}$ is only a homomorphism, rather than an isomorphism, of vector bundles.
\end{proof} 
\begin{remark} 
The Hamiltonian equations on 
$\Diff(M)\times\Xcal^*(M)\simeq T^*\Diff(M)$ 
and on 
$\Dens(M)\times\Xcal^*(M)$, 
as discussed in the previous sections, 
can now be written as 
\begin{equation} 
(\dot\varphi,\dot m) 
= 
\Gamma(DH(\varphi,m)) 
\quad\text{and}\quad 
(\dot\varrho,\dot m) = \bar\Gamma(D\bar H(\varrho,m)). 
\end{equation} 
Notice that $\bar\Gamma$ corresponds to a Poisson structure, but not to a symplectic structure as $\Gamma$ does; $\bar\Gamma$ is not invertible whereas $\Gamma$ is.
\end{remark}
The next theorem is the main result of this section. 
\begin{theorem} \label{thm:poisson_reduction_left} 
The following diagram 
\begin{equation} \label{eq:left_Poisson_bundle} 
\xymatrix@C-0.0pc{ 
\Diffvol(M)\ar@{^{(}->}[r] & T^*\Diff(M)\ar[d]^{\Pi\colon\!(\varphi,\,m_\varphi) 
\mapsto 
(\varphi_*\vol,\,m_\varphi\circ\,\varphi^{-1})} 
\\ 
& \Dens(M)\times \Xcal^*(M) 
} 
\end{equation} 
is a smooth tame principal bundle. 
The projection $\Pi$ is a Poisson submersion with respect to the Poisson structure on $\Dens(M)\times\Xcal^*(M)$. 
Solutions to the Hamiltonian equations for a $\Diff_\mu$-invariant Hamiltonian on $T^*\Diff(M)$ 
project to solutions of the Hamiltonian equations 
for the (unique) Hamiltonian $\bar H$ on $\Dens(M)\times \Xcal^\ast(M)$ 
satisfying 
$H(\varphi,m_\varphi) = \bar H(\varphi_*\vol,m_\varphi\circ\varphi^{-1})$. 
\end{theorem} 
\begin{proof} 
First, consider the map $(\varphi,m_\varphi)\to (\varphi,m)$ from $T^*\Diff(M)$ to the product $\Diff(M)\times \Xcal^*(M)$ 
and observe that it is a smooth tame vector bundle isomorphism, as in \autoref{lem:cotangent_tame}. 
The cotangent action of $\eta \in \Diffvol(M)$ on $(\varphi,m)$ acts on the first component 
by composition $(\varphi\circ\eta, m)$ 
and is clearly also a smooth tame map. 
Furthermore, we have 
$$ 
(\Diff(M)\times\Xcal^*(M))/\Diffvol(M) \simeq (\Diff(M)/\Diffvol(M))\times\Xcal^*(M). 
$$ 
The fact that $\Diff(M)$ is a smooth tame principal bundle over $\Dens(M)$ with fiber $\Diffvol(M)$ 
follows from the Nash-Moser-Hamilton theorem, cf. e.g., \cite[Thm. III.2.5.3]{Ha1982}.
Consequently, $T^*\Diff(M)$ is a smooth tame principal bundle over $\Dens(M)\times\Xcal^*(M)$ 
with fiber $\Diffvol(M)$. 

The projection is a Poisson submersion and  smooth solutions are mapped to smooth solutions: this  
follows from \autoref{lem:poisson_isomorphism} and \autoref{lem:poisson_isomorphism2} 
together with a straightforward calculation showing that 
$T\Pi\circ\Lambda(DH) = \bar\Lambda(D\bar H)$ 
whenever $H = \bar H\circ \Pi$.
\end{proof} 
\begin{remark} 
We point out that the situation is more complicated if one works with Banach spaces such as 
Sobolev $H^s$ or H\"older $C^{k,\alpha}$. 
In those settings the results in \autoref{lem:cotangent_tame}, \autoref{lem:poisson_isomorphism},  \autoref{lem:poisson_isomorphism2} and \autoref{thm:poisson_reduction_left} 
need not hold. For example, the bundle projection in \autoref{thm:poisson_reduction_left} 
typically fails to be Lipschitz continuous in the $H^s$ topology. 
\end{remark} 
%

\subsection{Short-time existence of compressible Euler equations} 
We include here a local existence result that applies to all the examples in  \autoref{sect:WOexamples}.
To this end consider the compressible Euler equations on a compact manifold $M$ 
in the form 
\begin{align} \label{eq:compressible_barotropic_general} 
\begin{cases} 
\dot v + \nabla_v v + \nabla (W\circ\rho) + \nabla V = 0  
\\ 
\dot\rho + \divv(\rho v) = 0 
\end{cases} 
\end{align} 
where $W$ is the thermodynamical work function defined in \eqref{eq:thermodyn_work}. 
The equations discussed previously can be captured by different choices of the functions 
$W$ and $V$. 
If $W$ is strictly increasing
 then short-time solutions of these equations can be obtained using 
standard techniques (see e.g. \cite[Thm.\ III.2.1.2]{Ha1982} for a result 
for the shallow water equations \eqref{eq:shallow} corresponding to $W(\rho)=\rho$). 
\begin{theorem} 
For any $v_{0}\in \Xcal(M)$, $\rho_{0} \in \Dens(M)$ and any smooth function 
$W\colon \RR_+ \to \RR$ such that $W'>0$ there exists a unique smooth solution $(v, \rho)$ of the equations  \eqref{eq:compressible_barotropic_general} 
satisfying $v(t_0)=v_0$, $\rho(t_0)=\rho_0$ and defined in some open neighbourhood of $t=t_0$. 
\end{theorem} 
\begin{proof} 
The basic idea is to transform \eqref{eq:compressible_barotropic_general} so that its linearization 
becomes a symmetric linear system. 
This can be achieved by a substitution $\rho = f\circ \sigma$ where $\sigma$ is a new density function 
and $f: \mathbb{R} \to \mathbb{R}$ is the solution of the following scalar initial value problem 
\begin{equation} \label{eq:EQ} 
f' = \sqrt{\frac{f}{W'(f)}} \, ,
\quad 
f(0)= \min_{x\in M}\rho_0(x). 
\end{equation} 
Compactness of $M$ together with the assumption $W'>0$ assure that the right-hand side of \eqref{eq:EQ} 
is Lipschitz continuous in the interval given by the range of $\rho_0$. 
Thus, there is a smooth solution $f$ whose range covers the range of $\rho_0$. 
Since $f(0) > 0$ and $W'>0$ this solution is strictly increasing. 

It follows that the corresponding linearized equations form a symmetric linear system 
in a neighborhood of the density $\rho = \rho_0$ and thus admit a unique tame solution 
by the general theory of symmetric systems. 
Applying the Nash-Moser-Hamilton theorem completes the proof. 
\end{proof} 

Using the results in \autoref{sect:tame_groups} it is possible to deduce from the above theorem 
short-time existence results for each of the equations considered in \autoref{sect:WOexamples}: 
the Newton systems on $\Diff(M)$, the Poisson systems on $\Dens(M)\times\Xcal^*(M)$ 
or the canonical Hamiltonian systems on $T^*\Dens(M)$.

\bibliographystyle{amsplainnat}
\bibliography{newton_diffeos2} 
 
\end{document}